\documentclass{amsart}

\newtheorem{theorem}{Theorem}
\newtheorem{lemma}[theorem]{Lemma}

\newtheorem{fact}[theorem]{Fact}
\newtheorem{algorithm}[theorem]{Algorithm}
\theoremstyle{definition}
\newtheorem{definition}[theorem]{Definition}
\theoremstyle{remark}
\newtheorem{remark}[theorem]{Remark}
\newtheorem{corollary}[theorem]{Corollary}
\numberwithin{theorem}{section}

\usepackage{amssymb}
\usepackage{empheq}
\usepackage{stmaryrd}
\usepackage{enumerate}
\usepackage{multicol}
\usepackage[shortlabels]{enumitem}
\usepackage{calc}
\usepackage{tikz}
\usepackage{listings}
\usepackage{bussproofs}

\lstdefinestyle{derivation} {
    breakatwhitespace=false,         
    breaklines=true,                 
    captionpos=b,                    
    keepspaces=true,                 
    numbers=left,                    
    numbersep=5pt,
    escapechar=\&
}

\makeatletter
\DeclareFontFamily{OMX}{MnSymbolE}{}
\DeclareSymbolFont{MnLargeSymbols}{OMX}{MnSymbolE}{m}{n}
\SetSymbolFont{MnLargeSymbols}{bold}{OMX}{MnSymbolE}{b}{n}
\DeclareFontShape{OMX}{MnSymbolE}{m}{n}{
    <-6>  MnSymbolE5
   <6-7>  MnSymbolE6
   <7-8>  MnSymbolE7
   <8-9>  MnSymbolE8
   <9-10> MnSymbolE9
  <10-12> MnSymbolE10
  <12->   MnSymbolE12
}{}
\DeclareFontShape{OMX}{MnSymbolE}{b}{n}{
    <-6>  MnSymbolE-Bold5
   <6-7>  MnSymbolE-Bold6
   <7-8>  MnSymbolE-Bold7
   <8-9>  MnSymbolE-Bold8
   <9-10> MnSymbolE-Bold9
  <10-12> MnSymbolE-Bold10
  <12->   MnSymbolE-Bold12
}{}

\let\llangle\@undefined
\let\rrangle\@undefined
\DeclareMathDelimiter{\llangle}{\mathopen}%
                     {MnLargeSymbols}{'164}{MnLargeSymbols}{'164}
\DeclareMathDelimiter{\rrangle}{\mathclose}%
                     {MnLargeSymbols}{'171}{MnLargeSymbols}{'171}
\makeatother

\newcommand{\Wedge}[1]{\ensuremath{\sideset{}{^{#1}}\bigwedge}}

\newcommand{\quasi}[1]{\ensuremath{\llangle{#1}\rrangle}}
\newcommand{\real}[1]{\ensuremath{\llbracket{#1}\rrbracket}}

\usetikzlibrary{automata,positioning}
\usetikzlibrary{arrows}

\usepackage[left=2.0cm,
                right=2.0cm,
                top=2.5cm,
                bottom=2.5cm,
                headheight=12pt,
                a4paper]{geometry}

\begin{document}

\title{On Intermediate Justification Logics}

\author{Nicholas Pischke}
\address{Hoch-Weiseler Str. 46, Butzbach, 35510, Hesse, Germany}
\email{pischkenicholas@gmail.com}


\begin{abstract}
We study abstract intermediate justification logics, that is arbitrary intermediate propositional logics extended with a subset of specific axioms of (classical) justification logics. For these, we introduce various semantics by combining either Heyting algebras or Kripke frames with the usual semantic machinery used by Mkrtychev's, Fitting's or Lehmann's and Studer's models for classical justification logics. We prove unified completeness theorems for all intermediate justification logics and their corresponding semantics using a respective propositional completeness theorem of the underlying intermediate logic. Further, by a modification of a method of Fitting, we prove unified realization theorems for a large class of intermediate justification logics and accompanying intermediate modal logics.
\end{abstract}

\maketitle

\section{Introduction}
Justification logics originated in the 90's from the studies of Artemov (see \cite{Art1995,Art2001}) regarding the provability interpretation of the modal logic $\mathbf{S4}$ (as initiated by G\"odel in \cite{Goe1933}) and the connected problem of formalizing the Brouwer-Heyting-Kolmogorov interpretation of intuitionistic propositional logic. From there, the prototype justification logic (the logic of proofs $\mathbf{LP}$) was substantially generalized and the resulting family of justification logics gained importance in the context of general (explicit) epistemic reasoning (see the survey \cite{Art2008}) with two recent textbooks on the subject \cite{AF2019,KS2019}. 

The original semantics for the logic of proofs was its intended arithmetical interpretation in Peano arithmetic but since then, various other semantics have been proposed which apply not only to the logic of proofs but to the whole family of justification logics. Notable instances important in this paper are the syntactic models of Mkrtychev \cite{Mkr1997} as well as the possible-world models of Fitting \cite{Fit2005} and the recent subset semantics of Lehmann and Studer \cite{LS2019}. These other semantical access points have been instrumental not only in demonstrating the strength of justification logics in modeling general epistemic scenarios and in understanding the ontology of justification terms and formulae in (classical) justification logics but also in inner-logical investigations for properties like decidability (see e.g. \cite{Kuz2000,Kuz2008,Mkr1997}).

The main theorem on the logic of proofs, besides arithmetical completeness, is the so called realization theorem which establishes a correspondence between $\mathbf{S4}$ and the logic of proofs where, constructively, every $\Box$ in a modal formula can be replaced with a suitable justification term such that the resulting formula is a theorem of $\mathbf{LP}$. This property was not only essential to the original motivation of the logic of proofs but is central also in the study of the whole framework of justification logics as it has analogues for all other known classical representatives, giving a central correspondence between justification and modal logics.\\

Besides the classical justification logics, there is a growing literature on non-classical justification logics, in particular encompassing various lines of research originating from the formalization of explicit, but vague, knowledge. In particular, there are the works on many-valued justification logics (see \cite{Gha2014,Gha2016,Pis2020,Pis2019}) and on intuitionistic justification logics (see \cite{KMS2017,MS2016,MS2018}). In fact, the G\"odel justification logics from \cite{Gha2014,Pis2020,Pis2019} also relate to the latter, with G\"odel logic as the base logic being one of the prime examples of an intermediate logic, originating from Dummet's work \cite{Dum1959} (in turn influenced by G\"odel's remarks on intuitionistic logic \cite{Goe1932}).\\

We give a unified theory regarding semantics and realization for the above examples of intuitionistic, G\"odel as well as classical justification logics \emph{and beyond} by introducing abstract intermediate justification logics, that is intermediate propositional logics over the justification language extended with a collection of designated justification axioms.\\

Semantically, starting at the two typical semantical access points for the underlying intermediate logics of (1) algebraic semantics based on Heyting algebras and of (2) the semantics of Kripke based on partial orders, we extended these algebraic and order theoretic approaches by the usual (appropriately adapted) semantic machinery for treating justification modalities from the classical models of Mkrtychev, Fitting as well as Lehmann and Studer. Here, the algebraic approach extends the classes of classical (or G\"odel) Mkrtychev, Fitting and subset models by allowing the models to take values not only in $\{0,1\}$ (or $[0,1]$ as in the G\"odel case) but in arbitrary Heyting algebras. The approach via intuitionistic Kripke frames extends the previous considerations for semantics of intuitionistic justification logics by new model classes as well as a wider range of applicable logics. All these considerations culminate in general unified completeness theorems based on a semantical characterization of the underlying intermediate logic.

Concerning realization, we modify the approach of Fitting towards non-constructive classical realization from \cite{Fit2016} and prove a very general, though non-constructive, unified realization theorem between these intermediate justification logics and a class of intermediate modal logics defined later. This is a direct example of the applicability of the previous semantic considerations, as this proof of the realization theorem relies on model theoretic constructions using Fitting's models over intuitionistic Kripke frames.
\section{Intermediate justification logics}
\subsection{Syntax and proof calculi}
We consider the propositional language
\[
\mathcal{L}_0:\phi::=\bot\mid p\mid (\phi\land\phi)\mid (\phi\lor\phi)\mid (\phi\rightarrow\phi)
\]
where $p\in Var:=\{p_i\mid i\in\mathbb{N}\}$. We introduce negation as the abbreviation $\neg\phi:=\phi\rightarrow\bot$. We also define
\[
\bigwedge_{i=1}^n\phi_n:=\phi_1\land\dots\land\phi_n
\]
for some $\phi_1,\dots,\phi_n\in\mathcal{L}_0$. The same applies to $\lor$. In order to define intermediate logics and later intermediate justification logics, we need to briefly review some notions regarding propositional substitutions.

A \emph{substitution in $\mathcal{L}_0$} is a function $\sigma:Var\to\mathcal{L}_0$. This function $\sigma$ naturally extends to $\mathcal{L}_0$ by commuting with the connectives $\land,\lor,\rightarrow$ and $\bot$ and we write $\sigma(\phi)$ also for the image of this extended function.

Using this definition of substitutions, we can now give the following definition of an intermediate justification logic.
\begin{definition}
An \emph{intermediate logic} (over $\mathcal{L}_0$) is a set $\mathbf{L}\subsetneq\mathcal{L}_0$ which satisfies:
\begin{enumerate}
\item the schemes $(A1)$ - $(A9)$ are contained in $\mathbf{L}$;
\item $\mathbf{L}$ is closed under \emph{modus ponens}, that is $\phi\rightarrow\psi,\phi\in\mathbf{L}$ implies $\psi\in\mathbf{L}$;
\item $\mathbf{L}$ is closed under substitution in $\mathcal{L}_0$.
\end{enumerate}
Here, the schemes $(A1)$ - $(A9)$ are given by:
\begin{enumerate}
\item [$(A1)$] $\phi\rightarrow (\psi\rightarrow\phi)$;
\item [$(A2)$] $(\phi\rightarrow(\chi\rightarrow\psi))\rightarrow ((\phi\rightarrow\chi)\rightarrow (\phi\rightarrow\psi))$;
\item [$(A3)$] $(\phi\land\psi)\rightarrow\phi$;
\item [$(A4)$] $(\phi\land\psi)\rightarrow\psi$;
\item [$(A5)$] $\phi\rightarrow(\psi\rightarrow(\phi\land\psi))$;
\item [$(A6)$] $\phi\rightarrow(\phi\lor\psi)$;
\item [$(A7)$] $\psi\rightarrow(\phi\lor\psi)$;
\item [$(A8)$] $(\phi\rightarrow\psi)\rightarrow ((\chi\rightarrow\psi)\rightarrow ((\phi\lor\chi)\rightarrow\psi))$;
\item [$(A9)$] $\bot\rightarrow\phi$.
\end{enumerate}
\end{definition}
We denote the smallest intermediate propositional logic, that is the logic given by the axiom schemes $(A1)$ - $(A9)$ in $\mathcal{L}_0$ closed under modus ponens, by $\mathbf{IPC}$. Given a set of formulae $\Gamma\subseteq\mathcal{L}_0$, we write
\[
\Gamma\vdash_\mathbf{L}\phi\text{ iff }\exists\gamma_1,\dots,\gamma_n\in\Gamma\;\left(\bigwedge_{i=1}^n\gamma_i\rightarrow\phi\in\mathbf{L}\right).
\]
On the side of justification logics, we consider the following set of \emph{justification terms}
\[
Jt:t::=x\mid c\mid [t+t]\mid [t\cdot t]\mid\; !t
\]
where $x\in V:=\{x_i\mid i\in\mathbb{N}\}$ and $c\in C:=\{c_i\mid i\in\mathbb{N}\}$ and the resulting multi-modal language
\[
\mathcal{L}_J:\phi::=\bot\mid p\mid (\phi\land\phi)\mid (\phi\lor\phi)\mid (\phi\rightarrow\phi)\mid t:\phi
\]
where $p\in Var$ and $t\in Jt$. Naturally, the same abbreviations as for $\mathcal{L}_0$ also apply here. Given a set $\Gamma,\Delta\subseteq\mathcal{L}_J$, we write $\Gamma +\Delta$ for the smallest set containing $\Gamma\cup\Delta$ which is closed under modus ponens. 

In order to formulate intermediate justification logics, we consider especially \emph{substitutions in $\mathcal{L}_J$}. These are again functions $\sigma:Var\to\mathcal{L}_J$ which extend uniquely to $\mathcal{L}_J$ by commuting with $\land,\lor,\rightarrow,\bot$ and the justification modalities `$t:$`. We again write $\sigma(\phi)$ for the image of a formula $\phi\in\mathcal{L}_J$ under this extension. By $\overline{\Gamma}$, we denote the closure of $\Gamma$ under substitutions in $\mathcal{L}_J$.
\begin{definition}
Let $\mathbf{L}$ be an intermediate propositional logic. Given the axiom schemes
\begin{enumerate}
\item [$(J)$] $t:(\phi\rightarrow\psi)\rightarrow (s:\phi\rightarrow [t\cdot s]:\psi)$,
\item [$(+)$] $t:\phi\rightarrow [t+s]:\phi$, $t:\phi\rightarrow [s+t]:\phi$,
\item [$(F)$] $t:\phi\rightarrow\phi$,
\item [$(I)$] $t:\phi\rightarrow !t:t:\phi$,
\end{enumerate}
we consider the following \emph{justification logics based on $\mathbf{L}$}:
\begin{enumerate}
\item $\mathbf{LJ}_0:=\overline{\mathbf{L}}+ (J)+ (+)$;
\item $\mathbf{LJT}_0:=\mathbf{LJ}_0+ (F)$;
\item $\mathbf{LJ4}_0:=\mathbf{LJ}_0+ (I)$;
\item $\mathbf{LJT4}_0:=\mathbf{LJ}_0+ (F)+ (I)$.
\end{enumerate}
\end{definition}
For any choice $\mathbf{LJL}_0\in\{\mathbf{LJ}_0,\mathbf{LJT}_0,\mathbf{LJ4}_0,\mathbf{LJT4}_0\}$, given $\Gamma\cup\{\phi\}\subseteq\mathcal{L}_J$, we write
\[
\Gamma\vdash_{\mathbf{LJL}_0}\phi\text{ iff }\exists\gamma_1,\dots,\gamma_n\in\Gamma\;\left(\bigwedge_{i=1}^n\gamma_i\rightarrow\phi\in\mathbf{LJL}_0\right).
\]
Specific instances of intermediate propositional logics and of the resulting intermediate justification logics which we consider explicitly in this paper are
\[
\mathbf{G}:=\mathbf{IPC}+(LIN),\quad\mathbf{C}:=\mathbf{IPC}+(LEM),
\]
with the schemes
\begin{enumerate}
\item [($LIN$)] $(\phi\rightarrow\psi)\lor (\psi\rightarrow\phi)$,
\item [($LEM$)] $\phi\lor\neg\phi$,
\end{enumerate}
over $\mathcal{L}_0$.
\begin{definition}
Let $\mathbf{L}$ be an intermediate propositional logic and $\mathbf{LJL}_0\in\{\mathbf{LJ}_0,\mathbf{LJT}_0,\mathbf{LJ4}_0,\mathbf{LJT4}_0\}$. A \emph{constant specification for $\mathbf{LJL}_0$} is a set $CS$ of formulae from $\mathcal{L}_J$ of the form
\[
c_{i_n}:\dots:c_{i_1}:\phi
\]
where $n\geq 1$, $c_{i_k}\in CS$ for all $k$ and $\phi$ is an axiom instance of $\mathbf{LJL}_0$, that is $\phi\in\overline{\mathbf{L}}$ or $\phi$ is an instance of the justification axiom schemes $(J),(+),(F),(I)$ (depending on $\mathbf{LJL}_0$).
\end{definition}
Constant specifications can be used to augment proof systems and increase the amount of justified formulae which they can prove.
\begin{definition}
Let $\mathbf{L}$ be an intermediate propositional logic and $\mathbf{LJL}_0\in\{\mathbf{LJ}_0,\mathbf{LJT}_0,\mathbf{LJ4}_0,\mathbf{LJT4}_0\}$ and let $CS$ be a constant specification for $\mathbf{LJL}_0$. We write $\Gamma\vdash_{\mathbf{LJL}_{CS}}\phi$ for $\Gamma\cup CS\vdash_{\mathbf{LJL}_0}\phi$ with $\Gamma\cup\{\phi\}\subseteq\mathcal{L}_J$. We also write $\mathbf{LJL}_{CS}\vdash\phi$ for $\emptyset\vdash_{\mathbf{LJL}_{CS}}\phi$.
\end{definition}
Note that any $\mathbf{LJL}_0$ and thus any $\mathbf{LJL}_{CS}$ fulfils the classical deduction theorem.\\

An important instance of a constant specification for $\mathbf{LJL}_0$ is the \emph{total constant specification} (that is the maximal constant specification w.r.t. $\subseteq$ in the sense of the above definition) and we denote it by $TCS_{\mathbf{LJL}_0}$. We write $\mathbf{LJL}$ for $\mathbf{LJL}_{TCS_{\mathbf{LJL}_0}}$. The total constant specification will be important later on in the proof of the realization theorems. For this, we already note the following lemma, a straightforward generalization of the classical lifting lemma of justification logics:
\begin{lemma}[Lifting Lemma]
Let $\mathbf{L}$ be an intermediate logic and let $\mathbf{LJL}\in\{\mathbf{LJ},\mathbf{LJT},\mathbf{LJ4},\mathbf{LJT4}\}$. Let $\{\gamma_1,\dots,\gamma_n,\phi\}\subseteq\mathcal{L}_J$. If
\[
\{\gamma_1,\dots,\gamma_n\}\vdash_{\mathbf{LJL}}\phi,
\]
then for any $s_1,\dots,s_n\in Jt$, there is a $t\in Jt$ such that
\[
\{s_1:\gamma_1,\dots,s_n:\gamma_n\}\vdash_{\mathbf{LJL}}t:\phi.
\]
\end{lemma}
A proof for the classical case, which transfers to the intermediate cases immediately, can be found e.g. in \cite{AF2019}. 

In particular, $\mathbf{LJL}$ has \emph{internalization}, that is $\mathbf{LJL}\vdash\phi$ implies that there is a term $t\in Jt$  such that $\mathbf{LJL}\vdash t:\phi$. It should also be noted that the justification variables of $t$ are among the combined justification variables of the $s_i$. In particular, if the terms $s_i$ are free from justification variables, then $t$ is so as well. 
\subsection{Extended propositional languages}
In later sections, it will be convenient to consider intermediate logics over different sets of variables. For this, we consider the language
\[
\mathcal{L}_0(X):\phi::=\bot\mid x\mid (\phi\land\phi)\mid (\phi\lor\phi)\mid (\phi\rightarrow\phi)
\]
where $X$ is a countably infinite set of variables. The same notational abbreviations as before also apply here. Note that naturally $\mathcal{L}_0(Var)=\mathcal{L}_0$. A particular choice different from $Var$ for $X$ in the following will be the set
\[
Var^\star:=Var\cup\{\phi_t\mid\phi\in\mathcal{L}_J,t\in Jt\}.
\]
Here, we write $\mathcal{L}_0^\star:=\mathcal{L}_0(Var^\star)$.

For the following definition, note that any bijection $t:Var\to X$ can be naturally extended to a bijection $t:\mathcal{L}_0\to\mathcal{L}_0(X)$ through recursion on $\mathcal{L}_0$ by commuting with $\land,\lor,\rightarrow$ and $\bot$. Also, such a bijection $t:Var\to X$ always exists as both $X$ and $Var$ are countably infinite.
\begin{definition}
Let $\mathbf{L}$ be an intermediate logic and let $t:Var\to X$ be a bijection extended to $t:\mathcal{L}_0\to\mathcal{L}_0(X)$ by commuting with $\land,\lor,\rightarrow$ and $\bot$. We write $\mathbf{L}(X):=t[\mathbf{L}]$.
\end{definition}
Note that here also $\mathbf{L}(Var)=\mathbf{L}$.
\begin{remark}
In the above definition, it is indeed not important which bijection $f:Var\to X$ is fixed as $\mathbf{L}$ is closed under substitutions. Further, naturally $\mathbf{L}(X)$ is closed under modus ponens and under substitutions of variables in $X$ by formulae in $\mathcal{L}_0(X)$.
\end{remark}
Given $\mathbf{L}(X)$ and $\Gamma\cup\{\phi\}\subseteq\mathcal{L}_0(X)$, we write $\Gamma\vdash_{\mathbf{L}(X)}\phi$ if as before
\[
\exists\gamma_1,\dots,\gamma_n\in\Gamma\;\left(\bigwedge_{i=1}^n\gamma_i\rightarrow\phi\in\mathbf{L}(X)\right).
\]
In the following, we will also write $\mathbf{L}^\star$ for the particular case of $\mathbf{L}(Var^\star)$.
\section{Algebraic semantics for intermediate justification logics}
We move on to the first main line of semantics for intermediate justification logics studied here, extending the model-theoretic approaches of Mkrtychev, Fitting as well as Lehmann and Studer to take values in arbitrary Heyting algebras. The models which we introduce and the techniques used later to prove corresponding completeness theorems are similar to those from \cite{Pis2020} where completeness theorems of the particular case of G\"odel justification logics with respect to models over the particular Heyting algebra $\mathbf{[0,1]_G}$ (see the last section) were considered. 
\subsection{Heyting algebras and propositional semantics}
We give some preliminaries on Heyting algebras and their relevant notions as a primer for the later definitions.
\begin{definition}
A \emph{Heyting algebra} is a structure $\mathbf{A}=\langle A,\land^\mathbf{A},\lor^\mathbf{A},\rightarrow^\mathbf{A},0^\mathbf{A},1^\mathbf{A}\rangle$ such that $\langle A,\land^\mathbf{A},\lor^\mathbf{A},0^\mathbf{A},1^\mathbf{A}\rangle$ is a bounded lattice with largest element $1^\mathbf{A}$ and smallest element $0^\mathbf{A}$ and $\rightarrow^\mathbf{A}$ is a binary operation with
\begin{enumerate}
\item $x\rightarrow^\mathbf{A}x=1^\mathbf{A}$,
\item $x\land^\mathbf{A}(x\rightarrow^\mathbf{A} y)=x\land^\mathbf{A} y$,
\item $y\land^\mathbf{A}(x\rightarrow^\mathbf{A} y)=y$,
\item $x\rightarrow^\mathbf{A} (y\land^\mathbf{A} z)=(x\rightarrow^\mathbf{A} y)\land^\mathbf{A}(x\rightarrow^\mathbf{A} z)$,
\end{enumerate}
where we write $a\leq^\mathbf{A}b$ for $a\land^\mathbf{A}b=a$.
\end{definition}
Note that this order $\leq^\mathbf{A}$ on $A$ is always a partial order. Given a Heyting algebra $\mathbf{A}$, we write $\neg^\mathbf{A}x:=x\rightarrow^\mathbf{A}0^\mathbf{A}$. $\mathbf{A}$ is called a \emph{Boolean algebra} if $x\rightarrow^\mathbf{A}y=\neg^\mathbf{A}x\lor^\mathbf{A}y$ for all $x,y\in A$.\\

We collect some facts about Heyting algebras which are of use later.
\begin{lemma}
Let $\mathbf{A}=\langle A,\land^\mathbf{A},\lor^\mathbf{A},\rightarrow^\mathbf{A},0^\mathbf{A},1^\mathbf{A}\rangle$ be a Heyting algebra. Then, for all $x,y,z,w\in A$:
\begin{enumerate}
\item $x\land^\mathbf{A}y\leq^\mathbf{A}z$ iff $x\leq^\mathbf{A}y\rightarrow^\mathbf{A}z$;
\item $x\leq^\mathbf{A}y$ iff $x\rightarrow^\mathbf{A}y=1^\mathbf{A}$;
\item $1\rightarrow^\mathbf{A}x=x$;
\item if $x\leq^\mathbf{A}y$, then $y\rightarrow^\mathbf{A}z\leq^\mathbf{A}x\rightarrow^\mathbf{A}z$;
\item $(x\rightarrow^\mathbf{A}y)\land^\mathbf{A}(z\rightarrow^\mathbf{A}w)\leq^\mathbf{A}(x\land^\mathbf{A}z)\rightarrow^\mathbf{A}(y\land^\mathbf{A}w)$.
\end{enumerate}
\end{lemma}
These properties are quite immediate from the definition of Heyting algebras. For a modern reference on basic properties of Heyting algebras, see e.g. \cite{Ono2019}.

Another particular property of Heyting algebras important in this note is that of \emph{completeness}.
\begin{definition}
A Heyting algebra $\mathbf{A}$ is \emph{complete} if every set $X\subseteq A$ has a \emph{join} and a \emph{meet} with respect to $\leq^\mathbf{A}$, that is for every $X\subseteq A$ there are $s_X,i_X\in A$ such that:
\begin{itemize}
\item $\forall x\in X\left(x\leq^\mathbf{A} s_X\right)$ and if $x\leq^\mathbf{A} s$ for all $x\in X$, then $s_X\leq^\mathbf{A} s$;
\item $\forall x\in X\left(i_X\leq^\mathbf{A} x\right)$ and if $i\leq^\mathbf{A} x$ for all $x\in X$, then $i\leq^\mathbf{A} i_X$.
\end{itemize}
\end{definition}
We denote these (unique) joins and meets, $s_X$ and $i_X$, by $\bigvee X$ and $\bigwedge X$, respectively. Given a class of Heyting algebras $\mathsf{C}$, we write $\mathsf{C}_{fin}$ for the subclass of all finite algebras and $\mathsf{C}_{com}$ for the subclass of all complete Heyting algebras in $\mathsf{C}$. Naturally, every finite Heyting algebra is complete.\\

Given an (extended) propositional language $\mathcal{L}_0(X)$, we can give an algebraic interpretation using various classes of particular Heyting algebras.
\begin{definition}
Let $\mathbf{A}$ be a Heyting algebra. A \emph{propositional evaluation of }$\mathcal{L}_0(X)$ is a function $f:\mathcal{L}_0(X)\to A$ which satisfies the following equations:
\begin{enumerate}
\item $f(\bot)=0^\mathbf{A}$;
\item $f(\phi\land\psi)=f(\phi)\land^\mathbf{A}f(\psi)$;
\item $f(\phi\lor\psi)=f(\phi)\lor^\mathbf{A}f(\psi)$;
\item $f(\phi\rightarrow\psi)=f(\phi)\rightarrow^\mathbf{A}f(\psi)$.
\end{enumerate}
\end{definition}
We denote the set of all $\mathbf{A}$-valued propositional evaluations of $\mathcal{L}_0(X)$ by $\mathsf{Ev}(\mathbf{A};\mathcal{L}_0(X))$.
\begin{definition}
Let $\mathsf{C}$ be a class of Heyting algebras and $\Gamma\cup\{\phi\}\subseteq\mathcal{L}_0(X)$. We write $\Gamma\models_\mathsf{C}\phi$ if
\[
\forall\mathbf{A}\in\mathsf{C}\forall f\in\mathsf{Ev}(\mathbf{A};\mathcal{L}_0(X))\big( f[\Gamma]\subseteq\{1^\mathbf{A}\}\text{ implies }f(\phi)=1^\mathbf{A}\big).
\]
\end{definition}
If in particular $C=\{\mathbf{A}\}$, we write $\models_\mathbf{A}$ for the corresponding relation.
\begin{definition}
Let $\mathbf{L}$ be an intermediate logic and let $X$ be a set of variables. We say that $\mathbf{L}(X)$ is \emph{(strongly) complete} with respect to a class $\mathsf{C}$ of Heyting algebras if for any $\Gamma\cup\{\phi\}\subseteq\mathcal{L}_0(X)$: $\Gamma\vdash_{\mathbf{L}(X)}\phi$ iff $\Gamma\models_\mathsf{C}\phi$.
\end{definition}
Although not particularly important for the rest of the paper, every intermediate logic actually has at least one class of Heyting algebras with respect to which it is strongly complete (namely its variety). We collect this in the following fact.
\begin{fact}
For every intermediate logic $\mathbf{L}$ and any set of variables $X$, there is a class of Heyting algebras $\mathsf{C}$ such that $\mathbf{L}(X)$ is strongly complete with respect to $\mathsf{C}$.
\end{fact}
For a modern reference of the proof, see again e.g. \cite{Ono2019}. Correspondingly, we introduce the following notation. We write $\mathsf{C}\in\mathsf{Alg}(\mathbf{L}(X))$, $\mathsf{C}\in\mathsf{Alg}_{com}(\mathbf{L}(X))$ or $\mathsf{C}\in\mathsf{Alg}_{fin}(\mathbf{L}(X))$ if $\mathsf{C}$ is a class of Heyting algebras, of complete Heyting algebras or of finite Heyting algebras with respect to which $\mathbf{L}(X)$ is strongly complete. Note that here
\[
\mathsf{C}\in\mathsf{Alg}(\mathbf{L}(X))\text{ iff }\mathsf{C}\in\mathsf{Alg}(\mathbf{L}(Y))
\]
for arbitrary sets of variables $X,Y$ and similarly for $\mathsf{Alg}_{com}(\mathbf{L}(X))$ and $\mathsf{Alg}_{fin}(\mathbf{L}(X))$.
\subsection{Algebraic Mkrtychev models}
The first class of semantics which we consider are algebraic Mkrtychev models. The classical Mkrtychev models were introduced in \cite{Mkr1997}, originally for the logic of proofs, and mark the first non-provability semantics. The generalization of the Mkrtychev models to the other classical justification logics $\mathbf{CJ}_0,\mathbf{CJT}_0,\mathbf{CJ4}_0$ and $\mathbf{CJT4}_0$ is due to Kuznets \cite{Kuz2000}. In some contexts, in particular \cite{AF2019,KS2019}, these models are also called basic models. The following algebraic models also generalize the work on $[0,1]$-valued Mkrtychev models in \cite{Gha2014,Pis2020} for the G\"odel justification logics.
\begin{definition}[Algebraic Mkrtychev model]\label{def:algmkrtmod}
Let $\mathbf{A}$ be a Heyting algebra. An (\emph{$\mathbf{A}$-valued}) \emph{algebraic Mkrtychev model} is a structure $\mathfrak{M}=\langle\mathbf{A},\mathcal{V}\rangle$ such that $\mathcal{V}:\mathcal{L}_J\to A$ fulfils
\begin{enumerate}
\item $\mathcal{V}(\bot)=0^\mathbf{A}$,
\item $\mathcal{V}(\phi\land\psi)=\mathcal{V}(\phi)\land^\mathbf{A}\mathcal{V}(\psi)$,
\item $\mathcal{V}(\phi\lor\psi)=\mathcal{V}(\phi)\lor^\mathbf{A}\mathcal{V}(\psi)$,
\item $\mathcal{V}(\phi\rightarrow\psi)=\mathcal{V}(\phi)\rightarrow^\mathbf{A}\mathcal{V}(\psi)$,
\end{enumerate}
for all $\phi,\psi\in\mathcal{L}_J$ and such that it satisfies
\begin{enumerate}[(i)]
\item $\mathcal{V}(t:(\phi\rightarrow\psi))\land^\mathbf{A}\mathcal{V}(s:\phi)\leq^\mathbf{A}\mathcal{V}([t\cdot s]:\psi)$,
\item $\mathcal{V}(t:\phi)\lor^\mathbf{A}\mathcal{V}(s:\phi)\leq^\mathbf{A}\mathcal{V}([t+s]:\phi)$,
\end{enumerate}
for all $t,s\in Jt$ and $\phi,\psi\in\mathcal{L}_J$.
\end{definition}
We write $\mathfrak{M}\models\phi$ if $\mathcal{V}(\phi)=1^\mathbf{A}$ and $\mathfrak{M}\models\Gamma$ if $\mathfrak{M}\models\gamma$ for all $\gamma\in\Gamma$ where $\Gamma\subseteq\mathcal{L}_J$.
\begin{definition}
Let $\mathfrak{M}=\langle\mathbf{A},\mathcal{V}\rangle$ be an $\mathbf{A}$-valued algebraic Mkrtychev model. We call $\mathfrak{M}$
\begin{enumerate}
\item \emph{factive} if $\mathcal{V}(t:\phi)\leq^\mathbf{A}\mathcal{V}(\phi)$ for all $\phi\in\mathcal{L}_J$ and all $t\in Jt$, and
\item \emph{introspective} if $\mathcal{V}(t:\phi)\leq^\mathbf{A}\mathcal{V}(!t:t:\phi)$ for all $\phi\in\mathcal{L}_J$ and all $t\in Jt$.
\end{enumerate}
\end{definition}
\begin{definition}
Let $\mathsf{C}$ be a class of Heyting algebras. Then:
\begin{enumerate}
\item $\mathsf{CAMJ}$ denotes the class of all $\mathbf{A}$-valued Mkrtychev models for all $\mathbf{A}\in\mathsf{C}$;
\item $\mathsf{CAMJT}$ denotes the class of all factive $\mathsf{CAMJ}$-models;
\item $\mathsf{CAMJ4}$ denotes the class of all introspective $\mathsf{CAMJ}$-models;
\item $\mathsf{CAMJT4}$ denotes the class of all factive and introspective $\mathsf{CAMJ}$-models.
\end{enumerate}
\end{definition}
\begin{definition}
Let $\mathbf{A}$ be a Heyting algebra and let $\mathfrak{M}=\langle\mathbf{A},\mathcal{V}\rangle$ be an algebraic Mkrtychev model. Further, let $CS$ be a constant specification (for some proof calculus). We say that $\mathfrak{M}$ \emph{respects} $CS$ if $\mathcal{V}(c:\phi)=1^\mathbf{A}$ for all $c:\phi\in CS$.
\end{definition}
If $\mathsf{C}$ is a class of algebraic Mkrtychev models, then we denote the subclass of all models from $\mathsf{C}$ respecting a constant specification $CS$ by $\mathsf{C}_{CS}$.
\begin{definition}
Let $\mathsf{C}$ be a class of algebraic Mkrtychev models and let $\Gamma\cup\{\phi\}\subseteq\mathcal{L}_J$. We write:
\begin{enumerate}
\item $\Gamma\models_\mathsf{C}\phi$ if $\forall\mathfrak{M}=\langle\mathbf{A},\mathcal{V}\rangle\in\mathsf{C}\left(\Wedge{\mathbf{A}}\{\mathcal{V}(\gamma)\mid\gamma\in\Gamma\}\leq^\mathbf{A}\mathcal{V}(\phi)\right)$;
\item $\Gamma\models_\mathsf{C}^1\phi$ if $\forall\mathfrak{M}=\langle\mathbf{A},\mathcal{V}\rangle\in\mathsf{C}\Big(\mathfrak{M}\models\Gamma\Rightarrow\mathfrak{M}\models\phi\Big)$.
\end{enumerate}
\end{definition}
\begin{lemma}\label{lem:algmkrtmodsoundness}
Let $\mathbf{L}$ be an intermediate logic, $\mathbf{LJL}_0\in\{\mathbf{LJ}_0,\mathbf{LJT}_0,\mathbf{LJ4}_0,\mathbf{LJT4}_0\}$, let $CS$ be a constant specification logic, and let $\mathsf{C}\in\mathsf{Alg}(\mathbf{L})$. Let $\mathsf{CAMJL}$ be the class of algebraic Mkrtychev models corresponding to $\mathbf{LJL}_0$ and $\mathsf{C}$. For any $\Gamma\cup\{\phi\}\subseteq\mathcal{L}_J$:
\[
\Gamma\vdash_{\mathbf{LJL}_{CS}}\phi\text{ implies }\Gamma\models_{\mathsf{CAMJL}_{CS}}\phi.
\]
\end{lemma}
\begin{proof}
We only show that $\vdash_{\mathbf{LJL}_0}\phi$ implies $\models_{\mathsf{CAMJL}}\phi$. This already suffices for the strong completeness statement above by the following argument using the deduction theorem for the respective logics and compactness of the provability relations:
\begin{align*}
\Gamma\vdash_{\mathbf{LJL}_{CS}}\phi&\text{ impl. }\exists\Gamma_0\subseteq\Gamma\cup CS\text{ finite }\left(\Gamma_0\vdash_{\mathbf{LJL}_0}\phi\right)\\
                                                             &\text{ impl. }\exists\Gamma_0\subseteq\Gamma\cup CS\text{ finite }\left(\vdash_{\mathbf{LJL}_0}\bigwedge\Gamma_0\rightarrow\phi\right)\\
                                                             &\text{ impl. }\exists\Gamma_0\subseteq\Gamma\cup CS\text{ finite }\left(\models_{\mathsf{CAMJL}}\bigwedge\Gamma_0\rightarrow\phi\right)\\
                                                             &\text{ impl. }\exists\Gamma_0\subseteq\Gamma\cup CS\text{ finite }\forall\mathfrak{M}=\langle\mathbf{A},\mathcal{V}\rangle\in\mathsf{CAMJL}\left(\sideset{}{^\mathbf{A}}\bigwedge\{\mathcal{V}(\gamma)\mid\gamma\in\Gamma_0\}\leq^\mathbf{A}\mathcal{V}(\phi)\right)\\
                                                             &\text{ impl. }\forall\mathfrak{M}=\langle\mathbf{A},\mathcal{V}\rangle\in\mathsf{CAMJL}\left(\sideset{}{^\mathbf{A}}\bigwedge\{\mathcal{V}(\gamma)\mid\gamma\in\Gamma\cup CS\}\leq^\mathbf{A}\mathcal{V}(\phi)\right)\\
                                                             &\text{ impl. }\forall\mathfrak{M}=\langle\mathbf{A},\mathcal{V}\rangle\in\mathsf{CAMJL}_{CS}\left(\sideset{}{^\mathbf{A}}\bigwedge\{\mathcal{V}(\gamma)\mid\gamma\in\Gamma\}\leq^\mathbf{A}\mathcal{V}(\phi)\right)\\
                                                             &\text{ impl. }\Gamma\models_{\mathsf{CAMJL}_{CS}}\phi.
\end{align*}
We show that $\vdash_{\mathbf{LJL}_0}\phi$ implies $\models_{\mathsf{CAMJL}}\phi$ as follows: by the definition of $\mathbf{LJL}_0$, it suffices to show that $\models_{\mathsf{CAMJL}}\phi$ for $\phi\in\overline{\mathbf{L}}$ as well as for $\phi\in (J)\cup(+)$ or even (depending on the choice of $\mathbf{LJL}_0$) $\phi\in(F)\cup(I)$ and that it is preserved under modus ponens. The latter is immediate. For the former, note that in the case of $\phi$ being a justification axiom, the choice of $\mathsf{CAMJL}$ is such that all models satisfy conditions (i) and (ii) of Definition \ref{def:algmkrtmod} (validating $(J)$ and $(+)$) and (depending on $\mathbf{LJL}_0$) are factive or introspective given $(F)$ or $(I)$ and thus validate those immediately. 

If now $\phi\in\overline{\mathbf{L}}$, then by definition there is a substitution $\sigma:Var\to\mathcal{L}_J$ and a formula $\psi\in\mathbf{L}$ such that $\phi=\sigma(\psi)$. Let $\mathbf{A}\in\mathsf{C}$ and $\mathfrak{M}=\langle\mathbf{A},\mathcal{V}\rangle$ be a $\mathsf{CAMJ}$-model. Then, we may define
\[
f:\chi\mapsto\mathcal{V}(\sigma(\chi))
\]
for $\chi\in\mathcal{L}_0$. By definition of $\mathfrak{M}$ and properties of $\sigma$, we have that $f$ is a well-defined evaluation on $\mathbf{A}$. By the choice of $\mathsf{C}$, we have that $\psi\in\mathbf{L}$ implies $f(\psi)=1^\mathbf{A}$ and thus $\mathcal{V}(\sigma(\psi))=\mathcal{V}(\phi)=1$. As $\mathfrak{M}$ was arbitrary, we have $\models_{\mathsf{CAMJ}}\phi$.
\end{proof}
\subsection{Algebraic Fitting models}
The second algebraic semantics which we consider is based on algebraic Fitting models, derived from the fundamental possible-world semantics of Fitting \cite{Fit2005} which combined the earlier work of Mkrtychev on syntactic evaluations with the usual semantics of non-explicit modal logics based on modal Kripke models. As a generalization, we allow the accessibility, evidence and evaluation functions to take values in Heyting algebras. We have to restrict these to complete Heyting algebras however, as we want certain algebraic equations to be satisfied which involve infima and suprema. The algebraic Fitting models presented here again generalize the previously introduced many-valued Fitting models from \cite{Gha2014,Pis2020} from the context of the G\"odel justification logics.
\begin{definition}\label{def:algfittingmod}
Let $\mathbf{A}$ be a complete Heyting algebra. An (\emph{$\mathbf{A}$-valued}) \emph{algebraic Fitting model} is a structure $\mathfrak{M}=\langle\mathbf{A},\mathcal{W},\mathcal{R},\mathcal{E},\mathcal{V}\rangle$ with
\begin{itemize}
\item $\mathcal{W}\neq\emptyset$,
\item $\mathcal{R}:\mathcal{W}\times \mathcal{W}\to A$,
\item $\mathcal{E}:\mathcal{W}\times Jt\times \mathcal{L}_J\to A$,
\item $\mathcal{V}:\mathcal{W}\times\mathcal{L}_J\to A$,
\end{itemize}
such that it fulfils the conditions
\begin{enumerate}
\item $\mathcal{V}(w,\bot)=0^\mathbf{A}$,
\item $\mathcal{V}(w,\phi\land\phi)=\mathcal{V}(w,\phi)\land^\mathbf{A}\mathcal{V}(w,\psi)$,
\item $\mathcal{V}(w,\phi\lor\phi)=\mathcal{V}(w,\phi)\lor^\mathbf{A}\mathcal{V}(w,\psi)$,
\item $\mathcal{V}(w,\phi\rightarrow\phi)=\mathcal{V}(w,\phi)\rightarrow^\mathbf{A}\mathcal{V}(w,\psi)$,
\item $\mathcal{V}(w,t:\phi)=\mathcal{E}_w(t,\phi)\land^\mathbf{A}\Wedge{\mathbf{A}}\{\mathcal{R}(w,v)\rightarrow^\mathbf{A} \mathcal{V}(v,\phi)\mid v\in \mathcal{W}\}$,
\end{enumerate}
for all $w\in \mathcal{W}$ and such that it satisfies
\begin{enumerate}[(i)]
\item $\mathcal{E}_w(t,\phi\rightarrow\psi)\land^\mathbf{A}\mathcal{E}_w(s,\phi)\leq^\mathbf{A}\mathcal{E}_w(t\cdot s,\psi)$,
\item $\mathcal{E}_w(t,\phi)\lor^\mathbf{A}\mathcal{E}_w(s,\phi)\leq^\mathbf{A}\mathcal{E}_w(t+s,\phi)$,
\end{enumerate}
for all $w\in\mathcal{W}$, all $t,s\in Jt$ and all $\phi,\psi\in\mathcal{L}_J$.
\end{definition}
We write $(\mathfrak{M},w)\models\phi$ for $\mathcal{V}(w,\phi)=1^\mathbf{A}$ and $(\mathfrak{M},w)\models\Gamma$ if $(\mathfrak{M},w)\models\gamma$ for all $\gamma\in\Gamma$.
\begin{definition}
Let $\mathfrak{M}=\langle\mathbf{A},\mathcal{W},\mathcal{R},\mathcal{E},\mathcal{V}\rangle$ be an $\mathbf{A}$-valued Fitting model. We call $\mathfrak{M}$
\begin{enumerate}[(i)]
\item \emph{reflexive} if $\forall w\in \mathcal{W}\left(\mathcal{R}(w,w)=1^\mathbf{A}\right)$,
\item \emph{transitive} if $\forall w,v,u\in \mathcal{W}\left(\mathcal{R}(w,v)\land^\mathbf{A}\mathcal{R}(v,u)\leq^\mathbf{A}\mathcal{R}(w,u)\right)$,
\item \emph{monotone} if $\forall v,w\in\mathcal{W}\forall t\in Jt\forall\phi\in\mathcal{L}_J\left(\mathcal{E}_w(t,\phi)\land^\mathbf{A}\mathcal{R}(w,v)\leq^\mathbf{A}\mathcal{E}_v(t,\phi)\right)$,
\item \emph{introspective} if it is transitive, monotone and satisfies
\[
\mathcal{E}_w(t,\phi)\leq^\mathbf{A} \mathcal{E}_w(!t,t:\phi)
\]
for all $w\in\mathcal{W}$ and all $t\in Jt$, $\phi\in\mathcal{L}_J$,
\item \emph{accessibility-crisp} if $\forall w,v\in \mathcal{W}\left( \mathcal{R}(w,v)\in\{0^\mathbf{A},1^\mathbf{A}\}\right)$.
\end{enumerate}
\end{definition}
\begin{definition}
Let $\mathsf{C}$ be a class of complete Heyting algebras. Then:
\begin{enumerate}
\item $\mathsf{CAFJ}$ denotes the class of all $\mathbf{A}$-valued Fitting models for all $\mathbf{A}\in\mathsf{C}$;
\item $\mathsf{CAFJT}$ denotes the class of all reflexive $\mathsf{CAFJ}$-models;
\item $\mathsf{CAFJ4}$ denotes the class of all introspective $\mathsf{CAFJ}$-models;
\item $\mathsf{CFJT4}$ denotes the class of all $\mathsf{CAFJ4}$-models which are reflexive.
\end{enumerate}
\end{definition}
By $\mathsf{C^c}$, we denote the class of all accessibility-crisp models in $\mathsf{C}$ for some class $\mathsf{C}$ of algebraic Fitting models.
\begin{definition}
Let $\mathbf{A}$ be a complete Heyting algebra and let $\mathfrak{M}=\langle\mathbf{A},\mathcal{W},\mathcal{R},\mathcal{E},\mathcal{V}\rangle$ be an $\mathbf{A}$-valued algebraic Fitting model. We say that $\mathfrak{M}$ \emph{respects a constant specification} $CS$ (for some proof system) if
\[
\mathcal{V}(w,c:\phi)=1^\mathbf{A}
\]
for all $w\in\mathcal{W}$ and all $c:\phi\in CS$.
\end{definition}
Given a class $\mathsf{C}$ of algebraic Fitting models, we denote the subclass of all algebraic Fitting models in $\mathsf{C}$ respecting a constant specification $CS$ (for some proof system) by $\mathsf{C}_{CS}$.
\begin{definition}
Let $\mathsf{C}$ be a class of algebraic Fitting models and $\Gamma\cup\{\phi\}\subseteq\mathcal{L}_J$. We write:
\begin{enumerate}
\item $\Gamma\models_\mathsf{C}\phi$ if $\forall\mathfrak{M}=\langle\mathbf{A},\mathcal{W},\mathcal{R},\mathcal{E},\mathcal{V}\rangle\in\mathsf{C}\forall w\in \mathcal{W}\left(\Wedge{\mathbf{A}}\{\mathcal{V}(w,\gamma)\mid\gamma\in\Gamma\}\leq^\mathbf{A} \mathcal{V}(w,\phi)\right)$;
\item $\Gamma\models^1_\mathsf{C}\phi$ if $\forall\mathfrak{M}=\langle\mathbf{A},\mathcal{W},\mathcal{R},\mathcal{E},\mathcal{V}\rangle\in\mathsf{C}\forall w\in \mathcal{W}\Big((\mathfrak{M},w)\models\Gamma\Rightarrow (\mathfrak{M},w)\models\phi\Big)$.
\end{enumerate}
\end{definition}
\begin{lemma}\label{lem:algfittingmodsoundness}
Let $\mathbf{L}$ be an intermediate logic and $\mathbf{LJL}_0\in\{\mathbf{LJ}_0,\mathbf{LJT}_0,\mathbf{LJ4}_0,\mathbf{LJT4}_0\}$. Let $CS$ be a constant specification for $\mathbf{LJL}_0$ and let $\mathsf{C}\in\mathsf{Alg}_{com}(\mathbf{L})$. Let $\mathsf{CAFJL}$ be the class of algebraic Fitting models corresponding to $\mathbf{LJL}_0$ and $\mathsf{C}$. For any $\Gamma\cup\{\phi\}\subseteq\mathcal{L}_J$:
\[
\Gamma\vdash_{\mathbf{LJL}_{CS}}\phi\text{ implies }\Gamma\models_{\mathsf{CAFJL}_{CS}}\phi.
\]
\end{lemma}
\begin{proof}
As before, we only show $\vdash_{\mathbf{LJL}_0}\phi$ implies $\models_\mathsf{CAFJL}\phi$. The argument from the proof of Lemma \ref{lem:algmkrtmodsoundness} about how to obtain strong soundness can be straightforwardly adapted to the case of algebraic Fitting models.

To see that $\vdash_{\mathbf{LJL}_0}\phi$ implies $\models_\mathsf{CAFJL}\phi$, note that it is again enough to show the claim for $\phi\in\overline{\mathbf{L}}$ or $\phi$ being a justification axiom (depending on $\mathbf{LJL}_0$).

If $\phi\in\overline{\mathbf{L}}$, then by the choice of $\mathsf{C}$ one may repeat the argument from the proof of Lemma \ref{lem:algmkrtmodsoundness} \emph{locally} for every $\mathcal{V}(w,\cdot)$ with $w\in\mathcal{W}$ to obtain $\models_{\mathsf{CAFJ}}\phi$.

As the algebraic Fitting models are slightly more complex in their evaluation of the justification modalities, we actually show the validity of the justification axiom schemes in their respective model classes. For this, let $\mathfrak{M}=\langle\mathbf{A},\mathcal{W},\mathcal{R},\mathcal{E},\mathcal{V}\rangle\in\mathsf{CAFJ}$ and let $w\in\mathcal{W}$.
\begin{enumerate}
\item Consider the axiom scheme $(J)$. Then, we have
\begin{align*}
&\Wedge{\mathbf{A}}\{\mathcal{R}(w,v)\rightarrow^\mathbf{A} \mathcal{V}(v,\phi\rightarrow\psi)\mid v\in \mathcal{W}\}\land^\mathbf{A}\Wedge{\mathbf{A}}\{\mathcal{R}(w,v)\rightarrow^\mathbf{A} \mathcal{V}(v,\phi)\mid v\in \mathcal{W}\}\\
&\qquad\qquad\leq^\mathbf{A}\Wedge{\mathbf{A}}\{\mathcal{R}(w,v)\rightarrow^\mathbf{A} \mathcal{V}(v,\psi)\mid v\in \mathcal{W}\}
\end{align*}
Further we have $\mathcal{E}_w(t,\phi\rightarrow\psi)\land^\mathbf{A}\mathcal{E}_w(s,\phi)\leq^\mathbf{A}\mathcal{E}_w(t\cdot s,\psi)$ by condition (i) of Definition \ref{def:algfittingmod}. Thus:
\begin{align*}
&\mathcal{V}(w,t:(\phi\rightarrow\psi))\land^\mathbf{A}\mathcal{V}(w,s:\phi)\\
&=\mathcal{E}_w(t,\phi\rightarrow\psi)\land^\mathbf{A}\Wedge{\mathbf{A}}\{\mathcal{R}(w,v)\rightarrow^\mathbf{A} \mathcal{V}(v,\phi\rightarrow\psi)\mid v\in \mathcal{W}\}\land^\mathbf{A}\\
&\qquad\qquad\mathcal{E}_w(s,\phi)\land^\mathbf{A}\Wedge{\mathbf{A}}\{\mathcal{R}(w,v)\rightarrow^\mathbf{A} \mathcal{V}(v,\phi)\mid v\in \mathcal{W}\}\\
&\leq^\mathbf{A}\mathcal{E}_w(t\cdot s,\psi)\land^\mathbf{A}\Wedge{\mathbf{A}}\{\mathcal{R}(w,v)\rightarrow^\mathbf{A} \mathcal{V}(v,\psi)\mid v\in \mathcal{W}\}\\
&=\mathcal{V}(w,[t\cdot s]:\psi).
\end{align*}
The claim follows from the above as by residuation, we have $\mathcal{V}(w,t:(\phi\rightarrow\psi))\leq^\mathbf{A}\mathcal{V}(w,s:\phi)\rightarrow^\mathbf{A}\mathcal{V}(w,[t\cdot s]:\psi)$.
\item Consider the axiom scheme $(+)$. We only show $\mathcal{V}(w,t:\phi)\leq^\mathbf{A}\mathcal{V}(w,[t+s]:\phi)$. The other part follows similarly. By condition (ii) of Definition \ref{def:algfittingmod}, we get
\[
\mathcal{E}_w(t,\phi)\leq^\mathbf{A}\mathcal{E}_w(t,\phi)\lor^\mathbf{A}\mathcal{E}_w(s,\phi)\leq^\mathbf{A}\mathcal{E}_w(t+s,\phi)
\]
and thus
\begin{align*}
\mathcal{V}(w,t:\phi)&=\mathcal{E}_w(t,\phi)\land^\mathbf{A}\Wedge{\mathbf{A}}\{\mathcal{R}(w,v)\rightarrow^\mathbf{A} \mathcal{V}(v,\phi)\mid v\in \mathcal{W}\}\\
                     &\leq^\mathbf{A}\mathcal{E}_w(t+s,\phi)\land^\mathbf{A}\Wedge{\mathbf{A}}\{\mathcal{R}(w,v)\rightarrow^\mathbf{A} \mathcal{V}(v,\phi)\mid v\in \mathcal{W}\}\\
                     &=\mathcal{V}(w,[t+s]:\phi).
\end{align*}
\item Consider the axiom scheme $(F)$ and assume that $\mathfrak{M}$ is reflexive. We have $\mathcal{R}(w,w)=1^\mathbf{A}$ and thus:
\begin{align*}
\mathcal{V}(w,t:\phi)&=\mathcal{E}_w(t,\phi)\land^\mathbf{A}\Wedge{\mathbf{A}}\{\mathcal{R}(w,v)\rightarrow^\mathbf{A} \mathcal{V}(v,\phi)\mid v\in \mathcal{W}\}\\
                     &\leq^\mathbf{A}\Wedge{\mathbf{A}}\{\mathcal{R}(w,v)\rightarrow^\mathbf{A} \mathcal{V}(v,\phi)\mid v\in \mathcal{W}\}\\
                     &\leq^\mathbf{A}\mathcal{R}(w,w)\rightarrow^\mathbf{A}\mathcal{V}(w,\phi)\\
                     &=\mathcal{V}(w,\phi).
\end{align*}
\item Consider the axiom scheme $(I)$. Assume that $\mathfrak{M}$ is introspective. By the transitivity of $\mathcal{R}$, we obtain at first
\[
\Wedge{\mathbf{A}}\{\mathcal{R}(w,u)\rightarrow^\mathbf{A}\mathcal{V}(u,\phi)\mid u\in\mathcal{W}\}\leq^\mathbf{A}\mathcal{R}(w,v)\rightarrow^\mathbf{A}\Wedge{\mathbf{A}}\{\mathcal{R}(v,u)\rightarrow^\mathbf{A}\mathcal{V}(u,\phi)\mid u\in\mathcal{W}\}.
\]
To see this, note that we have
\begin{align*}
\mathcal{R}(w,v)\land^\mathbf{A}\Wedge{\mathbf{A}}\{\mathcal{R}(w,u)\rightarrow^\mathbf{A}\mathcal{V}(u,\phi)\mid u\in\mathcal{W}\}&\leq^\mathbf{A}\mathcal{R}(w,v)\land^\mathbf{A}\left(\mathcal{R}(w,u)\rightarrow^\mathbf{A}\mathcal{V}(u,\phi)\right)\\
&\leq^\mathbf{A}\mathcal{R}(w,v)\land^\mathbf{A}\left(\left(\mathcal{R}(w,v)\land^\mathbf{A}\mathcal{R}(v,u)\right)\rightarrow^\mathbf{A}\mathcal{V}(u,\phi)\right)\\
&=\mathcal{R}(w,v)\land^\mathbf{A}\left(\mathcal{R}(w,v)\rightarrow^\mathbf{A}\left(\mathcal{R}(v,u)\right)\rightarrow^\mathbf{A}\mathcal{V}(u,\phi)\right)\\
&\leq^\mathbf{A}\mathcal{R}(w,u)\rightarrow^\mathbf{A}\mathcal{V}(u,\phi)
\end{align*}
for all $u\in\mathcal{W}$. By taking the infimum over $u$,the yields
\[
\mathcal{R}(w,v)\land^\mathbf{A}\Wedge{\mathbf{A}}\{\mathcal{R}(w,u)\rightarrow^\mathbf{A}\mathcal{V}(u,\phi)\mid u\in\mathcal{W}\}\leq^\mathbf{A}\Wedge{\mathbf{A}}\{\mathcal{R}(v,u)\rightarrow^\mathbf{A}\mathcal{V}(u,\phi)\mid u\in\mathcal{W}\}.
\]
Further, we obtain by monotonicity that
\[
\mathcal{E}_w(t,\phi)\leq^\mathbf{A}\mathcal{R}(w,v)\rightarrow^\mathbf{A}\mathcal{E}_v(t,\phi),
\]
and hence
\begin{align*}
\mathcal{V}(w,t:\phi)&=\mathcal{E}_w(t,\phi)\land^\mathbf{A}\Wedge{\mathbf{A}}\{\mathcal{R}(w,u)\rightarrow^\mathbf{A}\mathcal{V}(u,\phi)\mid u\in\mathcal{W}\}\\
&\leq^\mathbf{A}\left(\mathcal{R}(w,v)\rightarrow^\mathbf{A}\mathcal{E}_v(t,\phi)\right)\land^\mathbf{A}\left(\mathcal{R}(w,v)\rightarrow^\mathbf{A}\Wedge{\mathbf{A}}\{\mathcal{R}(v,u)\rightarrow^\mathbf{A}\mathcal{V}(u,\phi)\mid u\in\mathcal{W}\}\right)\\
&\leq^\mathbf{A}\mathcal{R}(w,v)\rightarrow^\mathbf{A}\left(\mathcal{E}_v(t,\phi)\land^\mathbf{A}\Wedge{\mathbf{A}}\{\mathcal{R}(v,u)\rightarrow^\mathbf{A}\mathcal{V}(u,\phi)\mid u\in\mathcal{W}\}\right)\\
&=\mathcal{R}(w,v)\rightarrow^\mathbf{A}\mathcal{V}(v,t:\phi).
\end{align*}
By taking the infimum, we get
\[
\mathcal{E}_w(t,\phi)\land^\mathbf{A}\Wedge{\mathbf{A}}\{\mathcal{R}(w,u)\rightarrow^\mathbf{A}\mathcal{V}(u,\phi)\mid u\in\mathcal{W}\}\leq^\mathbf{A}\Wedge{\mathbf{A}}\{\mathcal{R}(w,v)\rightarrow^\mathbf{A}\mathcal{V}(v,t:\phi)\mid v\in\mathcal{W}\}.
\]
and thus, we obtain for any $v\in\mathcal{W}$ by introspectivity:
\begin{align*}
\mathcal{V}(w,t:\phi)&\leq^\mathbf{A}\mathcal{E}_w(!t,t:\phi)\land^\mathbf{A}\Wedge{\mathbf{A}}\{\mathcal{R}(w,v)\rightarrow^\mathbf{A}\mathcal{V}(v,t:\phi)\mid v\in\mathcal{W}\}\\
&=\mathcal{V}(w,!t:t:\phi).
\end{align*}
\end{enumerate}
\end{proof}
\subsection{Algebraic subset models}
The last algebraic semantics which we consider is based on algebraic generalizations of the subset models for classical justification logic by Lehmann and Studer \cite{LS2019}. Similar as with the previous algebraic Fitting models, we allow all involved functions to take arbitrary values in Heyting algebras, again restricting ourselves to complete Heyting algebras to be able to formulate certain regularity conditions. 
\begin{definition}[Algebraic subset model]\label{def:algsubsetmod}
Let $\mathbf{A}$ be a complete Heyting algebra with domain $A$. An (\emph{$\mathbf{A}$-valued}) \emph{algebraic subset model} is a structure $\mathfrak{M}=\langle \mathbf{A},\mathcal{W},\mathcal{W}_0,\mathcal{E},\mathcal{V}\rangle$ with
\begin{itemize}
\item $\mathcal{W}\neq\emptyset$,
\item $\mathcal{W}_0\subseteq \mathcal{W}$, $\mathcal{W}_0\neq\emptyset$,
\item $\mathcal{E}:Jt\times \mathcal{W}\times \mathcal{W}\to A$,
\item $\mathcal{V}:\mathcal{W}\times\mathcal{L}_J\to A$,
\end{itemize}
such that for all $w\in \mathcal{W}_0$, $\mathcal{V}$ fulfils the conditions
\begin{enumerate}
\item $\mathcal{V}(w,\bot)=0^\mathbf{A}$,
\item $\mathcal{V}(w,\phi\land\psi)=\mathcal{V}(w,\phi)\land^\mathbf{A}\mathcal{V}(w,\psi)$,
\item $\mathcal{V}(w,\phi\lor\psi)=\mathcal{V}(w,\phi)\lor^\mathbf{A}\mathcal{V}(w,\psi)$,
\item $\mathcal{V}(w,\phi\rightarrow\psi)=\mathcal{V}(w,\phi)\rightarrow^\mathbf{A}\mathcal{V}(w,\psi)$,
\item $\mathcal{V}(w,t:\phi)=\sideset{}{^\mathbf{A}}\bigwedge\{\mathcal{E}_t(w,v)\rightarrow^\mathbf{A} \mathcal{V}(v,\phi)\mid v\in \mathcal{W}\}$,
\end{enumerate}
and such that it is \emph{regular}, that is for all $w\in \mathcal{W}_0$:
\begin{enumerate}[(i)]
\item $\mathcal{E}_{s+t}(w,v)\leq^\mathbf{A}\mathcal{E}_s(w,v)\land^\mathbf{A}\mathcal{E}_t(w,v)$ for all $v\in \mathcal{W}$;
\item for all $v\in\mathcal{W}$:
\[
\mathcal{E}_{s\cdot t}(w,v)\leq^\mathbf{A}\sideset{}{^\mathbf{A}}\bigwedge\{\mathfrak{M}_{s,t}^{w}(\psi)\rightarrow^\mathbf{A}\mathcal{V}(v,\psi)\mid\psi\in\mathcal{L}_J\}
\]
with
\[
\mathfrak{M}_{s,t}^{w}(\psi):=\sideset{}{^\mathbf{A}}\bigvee\{\mathcal{V}(w,s:(\phi\rightarrow\psi))\land^\mathbf{A} \mathcal{V}(w,t:\phi)\mid\phi\in\mathcal{L}_J\}.
\]
\end{enumerate}
\end{definition}
We write $(\mathfrak{M},w)\models\phi$ for $\mathcal{V}(w,\phi)=1^\mathbf{A}$ and $(\mathfrak{M},w)\models\Gamma$ for $\mathcal{V}(w,\gamma)=1^\mathbf{A}$ for all $\gamma\in\Gamma$.\\

The function $\mathcal{E}$ is actually a straightforward $\mathbf{A}$-valued generalization of the $E$-function from \cite{LS2019} as it is in fact nothing more than a different representation of the function
\[
\mathcal{E}:Jt\times \mathcal{W}\to A^\mathcal{W}
\]
which maps terms and worlds to \emph{$A$-valued subsets} of $\mathcal{W}$.
\begin{definition}
Let $\mathfrak{M}=\langle \mathbf{A},\mathcal{W},\mathcal{W}_0,\mathcal{E},\mathcal{V}\rangle$ be an $\mathbf{A}$-valued subset model. We call $\mathfrak{M}$
\begin{enumerate}[(i)]
\item \emph{reflexive} if $\forall w\in \mathcal{W}_0\forall t\in Jt\left(\mathcal{E}_t(w,w)=1^\mathbf{A}\right)$,
\item \emph{introspective} if $\forall w\in \mathcal{W}_0\forall v\in \mathcal{W}\forall t\in Jt\left(\mathcal{E}_{!t}(w,v)\leq^\mathbf{A}\sideset{}{^\mathbf{A}}\bigwedge\{\mathcal{V}(w,t:\phi)\rightarrow^\mathbf{A}\mathcal{V}(v,t:\phi)\mid\phi\in\mathcal{L}_J\}\right)$,
\item \emph{accessibility-crisp} if $\forall t\in Jt\forall w,v\in \mathcal{W}_0\left( \mathcal{E}_t(w,v)\in\{0^\mathbf{A},1^\mathbf{A}\}\right)$.
\end{enumerate}
\end{definition}
\begin{definition}
Let $\mathsf{C}$ be a class of complete Heyting algebras. Then:
\begin{enumerate}
\item $\mathsf{CASJ}$ denotes the class of all $\mathbf{A}$-valued subset models for all $\mathbf{A}\in\mathsf{C}$;
\item $\mathsf{CASJT}$ denotes the class of all reflexive $\mathsf{CASJ}$-models;
\item $\mathsf{CASJ4}$ denotes the class of all introspective $\mathsf{CASJ}$-models;
\item $\mathsf{CASJT4}$ denotes the class of all reflexive and introspective $\mathsf{CASJ}$-models.
\end{enumerate}
\end{definition}
Given a class $\mathsf{C}$ of algebraic subset models, we denote the class of all accessibility-crisp models in $\mathsf{C}$ by $\mathsf{C^c}$.
\begin{definition}
Let $\mathbf{A}$ be a complete Heyting algebra and let $\mathfrak{M}=\langle\mathbf{A},\mathcal{W},\mathcal{W}_0,\mathcal{E},\mathcal{V}\rangle$ be an $\mathbf{A}$-valued algebraic subset model. Further, let $CS$ be a constant specification (for some proof calculus). We say that $\mathfrak{M}$ \emph{respects} $CS$ if $\mathcal{V}(w,c:\phi)=1^\mathbf{A}$ for all $c:\phi\in CS$ and all $w\in \mathcal{W}_0$.
\end{definition}
Given a  class $\mathsf{C}$ of algebraic subset models, we write $\mathsf{C}_{CS}$ for the class of all models from $\mathsf{C}$ which respect $CS$. As before, there are two natural consequence relations to consider here.
\begin{definition}
Let $\Gamma\cup\{\phi\}\subseteq\mathcal{L}_J$ and $\mathsf{C}$ be a class of algebraic subset models. We write
\begin{enumerate}
\item $\Gamma\models_\mathsf{C}\phi$ if $\forall\mathfrak{M}=\langle\mathbf{A},\mathcal{W},\mathcal{W}_0,\mathcal{E},\mathcal{V}\rangle\in\mathsf{C}\forall w\in \mathcal{W}_0\left(\sideset{}{^\mathbf{A}}\bigwedge\{\mathcal{V}(w,\gamma)\mid\gamma\in\Gamma\}\leq^\mathbf{A}\mathcal{V}(w,\phi)\right)$;
\item $\Gamma\models^1_\mathsf{C}\phi$ if $\forall\mathfrak{M}=\langle\mathbf{A},\mathcal{W},\mathcal{W}_0,\mathcal{E},\mathcal{V}\rangle\in\mathsf{C}\forall w\in \mathcal{W}_0\Big((\mathfrak{M},w)\models\Gamma\Rightarrow (\mathfrak{M},w)\models\phi\Big)$.
\end{enumerate}
\end{definition}
We write $\Gamma\models_{\mathbf{A}\mathsf{JL}}\phi$ or $\Gamma\models^1_{\mathbf{A}\mathsf{JL}}\phi$ for $\Gamma\models_{\{\mathbf{A}\}\mathsf{JL}}\phi$ or $\Gamma\models^1_{\{\mathbf{A}\}\mathsf{JL}}\phi$, respectively.
\begin{lemma}\label{lem:algsubsetmodsoundness}
Let $\mathbf{L}$ be an intermediate logic and $\mathbf{LJL}_0\in\{\mathbf{LJ}_0,\mathbf{LJT}_0,\mathbf{LJ4}_0,\mathbf{LJT4}_0\}$. Further, let $CS$ be a constant specification for $\mathbf{LJL}_0$ and let $\mathsf{C}\in\mathsf{Alg}_{com}(\mathbf{L})$. Let $\mathsf{CASJL}$ be the class of algebraic subset models corresponding to $\mathbf{LJL}_0$ and $\mathsf{C}$. For any $\Gamma\cup\{\phi\}\subseteq\mathcal{L}_J$, we have:
\[
\Gamma\vdash_{\mathbf{LJL}_{CS}}\phi\text{ implies }\Gamma\models_{\mathsf{CASJL}_{CS}}\phi.
\]
\end{lemma}
\begin{proof}
By the same reasoning as in Lemmas \ref{lem:algmkrtmodsoundness} and \ref{lem:algfittingmodsoundness}, we only show $\vdash_{\mathbf{LJL}_0}\phi$ implies $\models_{\mathsf{CASJL}}\phi$. Similarly, it suffices to show the claim for $\phi\in\overline{\mathbf{L}}$ as well as the justifications axioms (based on $\mathbf{LJL}_0$).

We may repeat the argument from the previous soundness proofs that $\phi\in\overline{\mathbf{L}}$ implies $\models_{\mathsf{CASJ}}\phi$ by constructing a similar propositional evaluation $f$ locally for every $\mathcal{V}(w,\cdot)$ over every $w\in\mathcal{W}_0$.

We thus only show the validity of (1) ($J$), (2) ($+$), (3) ($F$) and (4) ($I$) in their respective model classes. For this, let $\mathfrak{M}=\langle A,\mathcal{W},\mathcal{W}_0,\mathcal{E},\mathcal{V}\rangle$ be a $\mathsf{CASJL}_{CS}$-model and let $w\in \mathcal{W}_0$.
\begin{enumerate}
\item We show
\[
\mathcal{V}(w,t:(\phi\rightarrow\psi))\land^\mathbf{A} \mathcal{V}(w,s:\phi)\leq^\mathbf{A} \mathcal{V}(w,[t\cdot s]:\psi)\tag{$\dagger$}.
\]
For this, let $v\in \mathcal{W}$. We then have 
\begin{align*}
\mathcal{V}(w,t:(\phi\rightarrow\psi))\land^\mathbf{A} \mathcal{V}(w,s:\phi)&\leq^\mathbf{A}\sideset{}{^\mathbf{A}}\bigvee\{\mathcal{V}(w,t:(\phi\rightarrow\psi))\land^\mathbf{A} \mathcal{V}(w,s:\phi)\mid\phi\in\mathcal{L}_J\}\\
                                                        &\leq^\mathbf{A}\mathcal{E}_{t\cdot s}(w,v)\rightarrow^\mathbf{A} \mathcal{V}(v,\psi)
\end{align*}
through condition (ii) of Definition \ref{def:algsubsetmod}. Therefore, we get
\[
\mathcal{V}(w,t:(\phi\rightarrow\psi))\land^\mathbf{A} \mathcal{V}(w,s:\phi)\leq^\mathbf{A}\sideset{}{^\mathbf{A}}\bigwedge\{\mathcal{E}_{t\cdot s}(w,v)\rightarrow^\mathbf{A} \mathcal{V}(v,\psi)\mid v\in \mathcal{W}\}=\mathcal{V}(w,[t\cdot s]:\psi)
\]
as $v$ was arbitrary,  which is ($\dagger$).
\item Let $v\in \mathcal{W}$. We have
\begin{align*}
\mathcal{V}(w,t:\phi)&=\sideset{}{^\mathbf{A}}\bigwedge\{\mathcal{E}_t(w,u)\rightarrow^\mathbf{A} \mathcal{V}(u,\phi)\mid u\in \mathcal{W}\}\\
           &\leq^\mathbf{A}\mathcal{E}_t(w,v)\rightarrow^\mathbf{A} \mathcal{V}(v,\phi)\\
           &\leq^\mathbf{A}\mathcal{E}_{t+s}(w,v)\rightarrow^\mathbf{A} \mathcal{V}(v,\phi)
\end{align*}
through condition (i) in Definition \ref{def:algsubsetmod}. As $v$ was arbitrary, we obtain
\[
\mathcal{V}(w,t:\phi)\leq^\mathbf{A}\sideset{}{^\mathbf{A}}\bigwedge\{\mathcal{E}_{t+s}(w,v)\rightarrow^\mathbf{A} \mathcal{V}(v,\phi)\mid v\in \mathcal{W}\}=\mathcal{V}(w,[t+s]:\phi).
\]
One shows similarly that $\mathcal{V}(w,s:\phi)\leq^\mathbf{A} \mathcal{V}(w,[t+s]:\phi)$.
\item $\mathfrak{M}$ is reflexive by assumption. Therefore, we have $\mathcal{E}_t(w,w)=1^\mathbf{A}$ as $w\in \mathcal{W}_0$ and thus
\begin{align*}
\mathcal{V}(w,t:\phi)&=\sideset{}{^\mathbf{A}}\bigwedge\{\mathcal{E}_t(w,v)\rightarrow^\mathbf{A}\mathcal{V}(v,\phi)\mid v\in \mathcal{W}\}\\
           &\leq^\mathbf{A}\mathcal{E}_t(w,w)\rightarrow^\mathbf{A}\mathcal{V}(w,\phi)\\
           &\leq^\mathbf{A}\mathcal{V}(w,\phi).
\end{align*}
Therefore, $\mathcal{V}(w,t:\phi)\rightarrow^\mathbf{A}\mathcal{V}(w,\phi)=1$.
\item $\mathfrak{M}$ is introspective by assumption. Thus, we have
\[
\mathcal{V}(w,t:\phi)\leq^\mathbf{A} \mathcal{E}_{!t}(w,v)\rightarrow^\mathbf{A}\mathcal{V}(v,t:\phi)
\]
for any $v\in \mathcal{W}$ by the introspectivity and we therefore get:
\[
\mathcal{V}(w,t:\phi)\leq^\mathbf{A}\sideset{}{^\mathbf{A}}\bigwedge\{\mathcal{E}_{!t}(w,v)\rightarrow^\mathbf{A}\mathcal{V}(v,t:\phi)\mid v\in \mathcal{W}\}=\mathcal{V}(w,!t:t:\phi).
\]
\end{enumerate}
\end{proof}
\section{Completeness for algebraic semantics}
To approach completeness, we translate the language $\mathcal{L}_J$ to $\mathcal{L}_0^\star$ by introducing the translation
\[
\star:\mathcal{L}_J\to\mathcal{L}_0^\star
\]
using recursion on $\mathcal{L}_J$ with the following clauses:
\begin{itemize}
\item $\bot^\star:=\bot$;
\item $p^\star:= p$;
\item $(\phi\circ\psi)^\star:=\phi^\star\circ\psi^\star$ with $\circ\in\{\land,\lor,\rightarrow\}$;
\item $(t:\phi)^\star:=\phi_t$.
\end{itemize}
Using the above translation, we can convert formulae containing justification modalities into formulae of $\mathcal{L}_0^\star$ and use semantic results for the intermediate logic in question \emph{over $\mathcal{L}_0^\star$} to derive results for the corresponding intermediate justification logic. This approach, especially in the context of algebra-valued modal logics, goes back to Caicedo and Rodriguez work \cite{CR2010} (see also \cite{Vid2015}) and was previously also applied in the context of many-valued justification logics (see \cite{Pis2020}).

For this, the following lemma provides a way to interpret modal systems in extended propositional systems. For this, given a proof calculus $\mathbf{S}$ over a language $\mathcal{L}$, we write $Th_{\mathbf{S}}:=\{\phi\in\mathcal{L}\mid\;\vdash_\mathbf{S}\phi\}$.
\begin{lemma}\label{lem:startrans}
Let $\mathbf{L}$ be an intermediate logic and $\mathbf{LJL}_0\in\{\mathbf{LJ}_0,\mathbf{LJT}_0,\mathbf{LJ4}_0,\mathbf{LJT4}_0\}$ and $CS$ be a constant specification for $\mathbf{LJL}_0$. For any $\Gamma\cup\{\phi\}\subseteq\mathcal{L}_J$:
\[
\Gamma\vdash_{\mathbf{LJL}_{CS}}\phi\text{ iff }\Gamma^\star\cup(Th_{\mathbf{LJL}_{CS}})^\star\vdash_{\mathbf{L}^\star}\phi^\star.
\]
\end{lemma}
\begin{proof}
We prove both directions separately. In any way, recall that $\star$ is a bijection between $\mathcal{L}_J$ and $\mathcal{L}_0^\star$.\\

For the direction from left to right, notice that it suffices to show $\Gamma^\star\cup(Th_{\mathbf{LJL}_{CS}})^\star\vdash_{\mathbf{L}^\star}\phi^\star$ for
\begin{enumerate}[(i)]
\item $\phi\in\Gamma$, or
\item $\phi\in CS$, or
\item $\phi\in\mathbf{LJL}_0$,
\end{enumerate}
and that it is preserved under modus ponens. The latter is obvious by definition of $\star$. For (i) of the former, we have $\phi^\star\in\Gamma^\star$ and thus $\Gamma^\star\cup(Th_{\mathbf{LJL}_{CS}})^\star\vdash_{\mathbf{L}^\star}\phi^\star$. For (ii) and (iii), we have $\vdash_{\mathbf{LJL}_{CS}}\phi$ and thus $\phi^\star\in (Th_{\mathbf{LJL}_{CS}})^\star$. Hence, we also obtain $\Gamma^\star\cup(Th_{\mathbf{LJL}_{CS}})^\star\vdash_{\mathbf{L}^\star}\phi^\star$.\\

For the direction from right to left, note that also here it suffices to show $\Gamma\vdash_{\mathbf{LJL}_{CS}}\phi$ for
\begin{enumerate}[(a)]
\item $\phi^\star\in\Gamma^\star$, or
\item $\phi^\star\in(Th_{\mathbf{LJL}_{CS}})^\star$, or
\item $\phi^\star\in\mathbf{L}^\star$,
\end{enumerate}
and that also here, it is preserved under modus ponens. The latter is again immediate. For (a) of the former, we have $\phi\in\Gamma$ which gives $\Gamma\vdash_{\mathbf{LJL}_{CS}}\phi$ directly. For (b), we have $\vdash_{\mathbf{LJL}_{CS}}\phi$ by definition. For (c), we have $\phi\in\overline{\mathbf{L}}$. To see this, note that by the definition of $\mathbf{L}^\star=\mathbf{L}(\mathcal{L}_0^\star)$, we have $\phi^\star=\sigma(\psi)$ for some $\psi\in\mathbf{L}$ and some bijection $t:Var\to Var^\star$. Now, the function
\[
\sigma_t:p\mapsto\begin{cases}q&\text{if }t(p)=q\\s:\phi&\text{if }t(p)=\phi_s\end{cases}
\]
is a substitution from $Var$ to $\mathcal{L}_J$ and we get
\[
\phi=\sigma_t(\psi).
\]
Thus, we have $\phi\in\overline{\mathbf{L}}$ and hence $\phi\in\mathbf{LJL}_0$, i.e. $\Gamma\vdash_{\mathbf{LJL}_{CS}}\phi$.
\end{proof}
The rest of this section is devoted countermodel constructions, converting algebraic evaluations of $\mathcal{L}_0^\star$ into corresponding algebraic Mkrtychev, Fitting or subset models and deriving corresponding completeness results for the intermediate justification logics from this.
\subsection{Completeness w.r.t. algebraic Mkrtychev models}
\begin{definition}
Let $\mathbf{L}$ be an intermediate logic and let $\mathbf{LJL}_0\in\{\mathbf{LJ}_0,\mathbf{LJT}_0,\mathbf{LJ4}_0,\mathbf{LJT4}_0\}$ where $CS$ is a constant specification for $\mathbf{LJL}_0$. Let $\mathbf{A}$ be a Heyting algebra and $v\in\mathsf{Ev}(\mathbf{A};\mathcal{L}_0^\star)$. The \emph{canonical algebraic Mkrtychev model w.r.t. $\mathbf{A}$ and $v$} is the structure $\mathfrak{M}_{\mathbf{A},v}^{c,M}(\mathbf{LJL}_{CS}):=\langle\mathbf{A},\mathcal{V}^c\rangle$ defined by:
\[
\mathcal{V}^c(\phi):=v(\phi^\star).
\]
\end{definition}
\begin{lemma}\label{lem:algmkrtmodwelldef}
For any Heyting algebra $\mathbf{A}$, any $v\in\mathsf{Ev}(\mathbf{A};\mathcal{L}_0^\star)$ with $v[(Th_{\mathbf{LJL}_{CS}})^\star]\subseteq\{1^\mathbf{A}\}$ and any choice of $\mathbf{LJL}_{CS}$, $\mathfrak{M}^{c,M}_{\mathbf{A},v}(\mathbf{LJL}_{CS})$ is a well-defined $\mathbf{A}$-valued algebraic Fitting model. Further:
\begin{enumerate}[(a)]
\item if $(F)$ is an axiom scheme of $\mathbf{LJL}_{CS}$, then $\mathfrak{M}^{c,M}_{\mathbf{A},v}(\mathbf{LJL}_{CS})$ is factive;
\item if $(I)$ is an axiom scheme of $\mathbf{LJL}_{CS}$, then $\mathfrak{M}^{c,M}_{\mathbf{A},v}(\mathbf{LJL}_{CS})$ is introspective. 
\end{enumerate}
\end{lemma}
\begin{proof}
As $v\in\mathsf{Ev}(\mathbf{A};\mathcal{L}_0^\star)$, we have items (1) - (4) from Definition \ref{def:algmkrtmod}. Then, as additionally $v[(Th_{\mathbf{LJL}_{CS}})^\star]\subseteq\{1^\mathbf{A}\}$, we get
\begin{align*}
\mathcal{V}^c(t:(\phi\rightarrow\psi))\land^\mathbf{A} \mathcal{V}^c(s:\phi)&=v((\phi\rightarrow\psi)_t)\land^\mathbf{A}v(\phi_s)\\
&\leq^\mathbf{A}v(\psi_{[t\cdot s]})\\
&=\mathcal{V}^c([t\cdot s]:\psi)
\end{align*}
and
\begin{align*}
\mathcal{V}^c(t:\phi)\lor^\mathbf{A} \mathcal{V}^c(s:\phi)&=v(\phi_t)\lor^\mathbf{A}v(\phi_s)\\
&\leq^\mathbf{A}v(\phi_{[t+s]})\\
&=\mathcal{V}^c([t+s]:\phi)
\end{align*}
regarding items (i) and (ii) of Definition \ref{def:algmkrtmod}. Now, regarding (a), if $(F)$ is an axiom scheme of $\mathbf{LJL}_{CS}$, we naturally have 
\[
\mathcal{V}^c(t:\phi) =v(\phi_t)\leq v(\phi^\star)=\mathcal{V}^c(\phi).
\]
As for item (b), if $(I)$ is an axiom scheme of $\mathbf{LJL}_{CS}$, we have
\[
\mathcal{V}^c(t:\phi)=v(\phi_t)\leq^\mathbf{A}v((t:\phi)_{!t})=\mathcal{V}^c(!t:t:\phi).
\]
\end{proof}
\begin{theorem}\label{thm:algmkrtmodcomp}
Let $\mathbf{L}$ be an intermediate logic and let $\mathbf{LJL}_0\in\{\mathbf{LJ}_0,\mathbf{LJT}_0,\mathbf{LJ4}_0,\mathbf{LJT4}_0\}$ where $CS$ is a constant specification for $\mathbf{LJL}_0$. Further, let $\mathsf{C}\in\mathsf{Alg}(\mathbf{L})$ and let $\mathsf{CAMJL}$ be the class of algebraic Mkrtychev models corresponding to $\mathbf{LJL}_0$ and $\mathsf{C}$.

For any $\Gamma\cup\{\phi\}\subseteq\mathcal{L}_J$, the following are equivalent:
\begin{enumerate}
\item $\Gamma\vdash_{\mathbf{LJL}_{CS}}\phi$;
\item $\Gamma\models_{\mathsf{CAMJL}_{CS}}\phi$;
\item $\Gamma\models^1_{\mathsf{CAMJL}_{CS}}\phi$.
\end{enumerate}
\end{theorem}
\begin{proof}
(1) implies (2) comes from Lemma \ref{lem:algmkrtmodsoundness} and (2) implies (3) is natural. For (3) implies (1), suppose $\Gamma\not\vdash_{\mathbf{LJL}_{CS}}\phi$. Then, by Lemma \ref{lem:startrans}, we know
\[
\Gamma^\star\cup (Th_{\mathbf{LJL}_{CS}})^\star\not\vdash_{\mathbf{L}^\star}\phi^\star
\]
which implies that there exists an $\mathbf{A}\in\mathsf{C}$ and a $v\in\mathsf{Ev}(\mathbf{A};\mathcal{L}_0^\star)$ such that
\[
v[\Gamma^\star]\subseteq\{1^\mathbf{A}\}, v[(Th_{\mathbf{LJL}_{CS}})^\star]\subseteq\{1^\mathbf{A}\}\text{ and }v(\phi)<1^\mathbf{A}
\]
by assumption on $\mathsf{C}$. By Lemma \ref{lem:algmkrtmodwelldef}, we have that $\mathfrak{M}^{c,M}_{\mathbf{\mathbf{A},v}}(\mathbf{LJL}_{CS})$ is a well-defined $\mathsf{CAMJL}$-model and by definition, it follows that:
\[
\mathfrak{M}\models\Gamma\text{ and }\mathfrak{M}\not\models\phi.
\]
Also, $\mathfrak{M}^{c,M}_{\mathbf{A},v}(\mathbf{LJL}_{CS})$ respects $CS$. As we have $\vdash_{\mathbf{LJL}_{CS}}c:\phi$ for $c:\phi\in CS$, we have $\phi_c\in (Th_{\mathbf{LJL}_{CS}})^\star$ and thus $v(\phi_c)=1^\mathbf{A}$. By definition, this yields
\[
\mathcal{V}^c(c:\phi)=v(\phi_c)=1^\mathbf{A}.
\]
Hence, we obtain $\Gamma\not\models^1_{\mathsf{CAMJL}_{CS}}\phi$.
\end{proof}
\subsection{Completeness w.r.t. algebraic Fitting models}
\begin{definition}
Let $\mathbf{L}$ be an intermediate propositional logic and let $\mathbf{LJL}_0\in\{\mathbf{LJ}_0,\mathbf{LJT}_0,\mathbf{LJ4}_0,\mathbf{LJT4}_0\}$ where $CS$ is a constant specification for $\mathbf{LJL}_0$. Let $\mathbf{A}$ be a complete Heyting algebra. The \emph{canonical algebraic Fitting model w.r.t. $\mathbf{A}$} is the structure $\mathfrak{M}_\mathbf{A}^{c,F}(\mathbf{LJL}_{CS}):=\langle \mathbf{A},\mathcal{W}^c,\mathcal{W}^c_0,\mathcal{E}^c,\mathcal{V}^c\rangle$ defined as follows:
\begin{itemize}
\item $\mathcal{W}^c:=\{v\in\mathsf{Ev}(\mathbf{A};\mathcal{L}_0^\star)\mid v[(Th_{\mathbf{LJL}_{CS}})^\star]\subseteq\{1^\mathbf{A}\}\}$;
\item $\mathcal{R}^c(v,w):=\begin{cases}1^\mathbf{A}&\text{if }\forall t\in Jt\forall\phi\in\mathcal{L}_J\left(v(\phi_t)\leq^\mathbf{A} w(\phi^\star)\right);\\0^\mathbf{A}&\text{otherwise};\end{cases}$
\item $\mathcal{E}^c_v(t,\phi):=v(\phi_t)$;
\item $\mathcal{V}^c(v,\phi):=v(\phi^\star)$.
\end{itemize}
\end{definition}
\begin{lemma}\label{lem:algfittingmodwelldef}
For any complete Heyting algebra $\mathbf{A}$ and any choice of $\mathbf{LJL}_{CS}$, $\mathfrak{M}^{c,F}_\mathbf{A}(\mathbf{LJL}_{CS})$ is a well-defined $\mathbf{A}$-valued algebraic Fitting model. Further:
\begin{enumerate}[(a)]
\item if $(F)$ is an axiom scheme of $\mathbf{LJL}_{CS}$, then $\mathfrak{M}^{c,F}_\mathbf{A}(\mathbf{LJL}_{CS})$ is reflexive;
\item if $(I)$ is an axiom scheme of $\mathbf{LJL}_{CS}$, then $\mathfrak{M}^{c,F}_\mathbf{A}(\mathbf{LJL}_{CS})$ is introspective. 
\end{enumerate}
\end{lemma}
\begin{proof}
Conditions (1) - (4) from Definition \ref{def:algfittingmod} follow immediately for any $v\in \mathcal{W}^c$ as $v\in\mathsf{Ev}(\mathbf{A};\mathcal{L}_0^\star)$ and by definition of $\star$. For item (5), we have
\[
v(\phi_t)\leq^\mathbf{A} w(\phi^\star)
\]
for any $w\in \mathcal{W}^c$ with $\mathcal{R}^c(v,w)=1^\mathbf{A}$. Thus, we obtain
\[
v(\phi_t)\leq^\mathbf{A}\Wedge{\mathbf{A}}\{w(\phi^\star)\mid w\in \mathcal{W}^c, \mathcal{R}^c(v,w)=1^\mathbf{A}\}=\Wedge{\mathbf{A}}\{\mathcal{R}^c(v,w)\rightarrow^\mathbf{A}w(\phi^\star)\mid w\in \mathcal{W}^c\}.
\]
Therefore
\[
\mathcal{E}^c_v(t,\phi)\land^\mathbf{A}\Wedge{\mathbf{A}}\{\mathcal{R}^c(v,w)\rightarrow^\mathbf{A}w(\phi^\star)\mid w\in \mathcal{W}^c\}=v(\phi_t).
\]
For item (i), note that
\begin{align*}
\mathcal{E}^c_v(t,\phi\rightarrow\psi)\land^\mathbf{A}\mathcal{E}^c_v(s,\phi)&=v((\phi\rightarrow\psi)_t)\land^\mathbf{A} v(\phi_s)\\
&\leq^\mathbf{A} v(\psi_{[t\cdot s]})\\
&=\mathcal{E}^c_v(t\cdot s,\psi)
\end{align*}
where the inequality follows using the axiom scheme ($J$) as $v[(Th_{\mathbf{LJL}_{CS}})^\star]\subseteq\{1^\mathbf{A}\}$ and $v\in\mathsf{Ev}(\mathbf{A};\mathcal{L}_0^\star)$.

For item (ii), note that
\[
\mathcal{E}^c_v(t,\phi)=v(\phi_t)\leq^\mathbf{A}v(\phi_{[t+s]})=\mathcal{E}^c_v(t+s,\phi)
\]
and similarly for $s$ through the axiom scheme ($+$) as again $v[(Th_{\mathbf{LJL}_{CS}})^\star]\subseteq\{1^\mathbf{A}\}$ and $v\in\mathsf{Ev}(\mathbf{A};\mathcal{L}_0^\star)$. Thus, we have
\[
\mathcal{E}^c_v(t,\phi)\lor^\mathbf{A}\mathcal{E}^c_v(s,\phi)\leq^\mathbf{A}\mathcal{E}^c_v(t+s,\phi).
\]
On to item (a), if $(F)$ is an axiom scheme of $\mathbf{LJL}_{CS}$, then we naturally have
\[
v(\phi_t)\leq^\mathbf{A}v(\phi^\star)
\]
for any $\phi\in\mathcal{L}_J$ and any $t\in Jt$. Thus, in particular we have $\mathcal{R}^c(v,v)=1^\mathbf{A}$ by definition and hence $\mathcal{R}^c$ is reflexive.

For item (b), note at first that by the axiom scheme $(I)$, we have
\[
\mathcal{E}^c_v(t,\phi)=v(\phi_t)\leq^\mathbf{A}v((t:\phi)_{!t})=\mathcal{E}^c_v(!t,t:\phi)
\]
for any $\phi\in\mathcal{L}_J$ and any $t\in Jt$ by definition of the canonical model. Further, we have that $\mathcal{R}^c$ is transitive. For this, let $\mathcal{R}^c(v,w)=\mathcal{R}^c(w,u)=1^\mathbf{A}$. Then, we have for any $\phi\in\mathcal{L}_J$ and any $t\in Jt$:
\[
v(\phi_t)\leq^\mathbf{A}v((t:\phi)_{!t})\leq^\mathbf{A} w(\phi_t)\leq^\mathbf{A} u(\phi^\star)
\]
and thus $\mathcal{R}^c(v,u)=1^\mathbf{A}$. For the property of monotonicity, suppose $\mathcal{R}^c(v,w)=1^\mathbf{A}$. Then, we obtain
\[
\mathcal{E}^c_v(t,\phi)=v(\phi_t)\leq^\mathbf{A}v((t:\phi)_{!t})\leq^\mathbf{A}w(\phi_t)=\mathcal{E}^c_w(t,\phi)
\]
which is monotonicity.
\end{proof}
Similarly as before, we obtain the following completeness theorem.
\begin{theorem}\label{thm:algfittingmodcomp}
Let $\mathbf{L}$ be an intermediate logic and let $\mathbf{LJL}_0\in\{\mathbf{LJ}_0,\mathbf{LJT}_0,\mathbf{LJ4}_0,\mathbf{LJT4}_0\}$ where $CS$ is a constant specification for $\mathbf{LJL}_0$. Further, let $\mathsf{C}\in\mathsf{Alg}_{com}(\mathbf{L})$ and let $\mathsf{CAFJL}$ be the class of algebraic Fitting models corresponding to $\mathbf{LJL}_0$ and $\mathsf{C}$.

For any $\Gamma\cup\{\phi\}\subseteq\mathcal{L}_J$, the following are equivalent:
\begin{enumerate}
\item $\Gamma\vdash_{\mathbf{LJL}_{CS}}\phi$;
\item $\Gamma\models_{\mathsf{CAFJL}_{CS}}\phi$;
\item $\Gamma\models^1_{\mathsf{CAFJL}_{CS}}\phi$;
\item $\Gamma\models^1_{\mathsf{CAFJL^c}_{CS}}\phi$.
\end{enumerate}
\end{theorem}
\subsection{Completeness w.r.t. algebraic subset models}
\begin{definition}
Let $\mathbf{L}$ be an intermediate propositional logic and let $\mathbf{LJL}_0\in\{\mathbf{LJ}_0,\mathbf{LJT}_0,\mathbf{LJ4}_0,\mathbf{LJT4}_0\}$ where $CS$ is a constant specification for $\mathbf{LJL}_0$. Let $\mathbf{A}$ be a complete Heyting algebra. The \emph{canonical algebraic subset model w.r.t. $\mathbf{A}$} is the structure $\mathfrak{M}_\mathbf{A}^{c,S}(\mathbf{LJL}_{CS}):=\langle \mathbf{A},\mathcal{W}^c,\mathcal{W}^c_0,\mathcal{E}^c,\mathcal{V}^c\rangle$ defined as follows:
\begin{itemize}
\item $\mathcal{W}^c:=A^{\mathcal{L}_J}$;
\item $\mathcal{W}^c_0:=\{v\in\mathsf{Ev}(\mathbf{A};\mathcal{L}_0^\star)\mid v[(Th_{\mathbf{LJL}_{CS}})^\star]\subseteq\{1^\mathbf{A}\}\}$;
\item $\mathcal{E}^c_t(v,w):=\begin{cases}1^\mathbf{A}&\text{if }\forall\phi\in\mathcal{L}_J\left(v(\phi_t)\leq^\mathbf{A} w(\phi^\star)\right);\\0^\mathbf{A}&\text{otherwise};\end{cases}$
\item $\mathcal{V}^c(v,\phi):=v(\phi^\star)$.
\end{itemize}
\end{definition}
\begin{lemma}\label{lem:algsubsetmodwelldef}
For any complete Heyting algebra $\mathbf{A}$ and any choice of $\mathbf{LJL}_{CS}$, $\mathfrak{M}^{c,S}_\mathbf{A}(\mathbf{LJL}_{CS})$ is a well-defined $\mathbf{A}$-valued algebraic subset model. Further:
\begin{enumerate}[(a)]
\item if $(F)$ is an axiom scheme of $\mathbf{LJL}_{CS}$, then $\mathfrak{M}^{c,S}_\mathbf{A}(\mathbf{LJL}_{CS})$ is reflexive;
\item if $(I)$ is an axiom scheme of $\mathbf{LJL}_{CS}$, then $\mathfrak{M}^{c,S}_\mathbf{A}(\mathbf{LJL}_{CS})$ is introspective.
\end{enumerate}
\end{lemma}
\begin{proof}
To show that $\mathfrak{M}^c_\mathbf{A}(\mathbf{LJL}_{CS})$ is well-defined, we have to verify the conditions (1) - (5) and (i), (ii) from Definition \ref{def:algsubsetmod}. For this, let $v\in \mathcal{W}^c_0$. We only show (5) from the former, as (1) - (4) follows naturally from $\mathsf{v}\in\mathsf{Ev}(\mathbf{A};\mathcal{L}_0^\star)$.

For (5), we show the equality in two steps. At first, note that
\begin{align*}
\sideset{}{^\mathbf{A}}\bigwedge\{\mathcal{E}^c_t(v,w)\rightarrow^\mathbf{A}\mathcal{V}^c(w,\phi)\mid w\in \mathcal{W}^c\}&=\sideset{}{^\mathbf{A}}\bigwedge\{\mathcal{E}^c_t(v,w)\rightarrow^\mathbf{A}w(\phi^\star)\mid w\in \mathcal{W}^c\}\\
&=\sideset{}{^\mathbf{A}}\bigwedge\{w(\phi^\star)\mid w\in \mathcal{W}^c, \mathcal{E}^c_t(v,w)=1^\mathbf{A}\}.
\end{align*}
Now, by definition we have
\[
\mathcal{V}^c(v,t:\phi)=v(\phi_t)\leq^\mathbf{A} w(\phi^\star)
\]
for any $w\in \mathcal{W}^c$. Thus, we naturally have
\[
\mathcal{V}^c(v,t:\phi)\leq^\mathbf{A}\sideset{}{^\mathbf{A}}\bigwedge\{w(\phi^\star)\mid w\in \mathcal{W}^c, \mathcal{E}^c_t(v,w)=1^\mathbf{A}\}.
\]
For the other direction, consider
\[
v_t:\mathcal{L}_J\to A, \psi^\star\mapsto\psi_t.
\]
Then, we have that $v_t\in \mathcal{W}^c$ and further
\[
v(\psi_t)\leq^\mathbf{A}v_t(\phi^\star)
\]
by definition. Thus $\mathcal{E}^c_t(v,v_t)=1^\mathbf{A}$ and therefore
\[
\sideset{}{^\mathbf{A}}\bigwedge\{w(\phi^\star)\mid w\in \mathcal{W}^c, \mathcal{E}^c_t(v,w)=1^\mathbf{A}\}\leq^\mathbf{A}v_t(\phi^\star)=v(\phi_t).
\]
Let further $w\in \mathcal{W}^c$.
\begin{enumerate}[(i)]
\item Suppose $\mathcal{E}^c_{t+s}(v,w)=1^\mathbf{A}$. Then, we have (as $v\in\mathsf{Ev}(\mathbf{A};\mathcal{L}_0^\star)$ and $v[(Th_{\mathbf{LJL}_{CS}})^\star]\subseteq\{1^\mathbf{A}\}$)
\[
v(\phi_t)\leq^\mathbf{A}v(\phi_{[t+s]})\leq^\mathbf{A}w(\phi^\star)
\]
through axiom scheme ($+$) for any $\phi\in\mathcal{L}_J$ and similarly for $v(\phi_s)$. Thus, we have $\mathcal{E}^c_s(v,w)=\mathcal{E}^c_t(v,w)=1^\mathbf{A}$.
\item Suppose $\mathcal{E}^c_{t\cdot s}(v,w)=1^\mathbf{A}$. We write $(\mathfrak{M}^c)^v_{t,s}$ as a shorthand for $(\mathfrak{M}^{c,S}_\mathbf{A}(\mathbf{LJL}_{CS}))^v_{t,s}$. Then, to show
\[
(\mathfrak{M}^c)^{v}_{t,s}(\psi)\leq^\mathbf{A} w(\psi^\star)
\]
for every $\psi\in\mathcal{L}_J$, it suffices to show (for an arbitrary $\psi\in\mathcal{L}_J$):
\[
(\mathfrak{M}^c)^{v}_{t,s}(\psi)\leq^\mathbf{A} v(\psi_{t\cdot s})\tag{$\dagger$}.
\]
($\dagger$) however follows from
\begin{align*}
(\mathfrak{M}^c)^{v}_{t,s}(\psi)&=\sideset{}{^\mathbf{A}}\bigvee\{\mathcal{V}^c(v,t:(\phi\rightarrow\psi))\land^\mathbf{A}\mathcal{V}^c(v,s:\phi)\mid\phi\in\mathcal{L}_J\}\\
                                &=\sideset{}{^\mathbf{A}}\bigvee\{v((\phi\rightarrow\psi)_t)\land^\mathbf{A}v(\phi_s)\mid\phi\in\mathcal{L}_J\}\\
                                &\leq^\mathbf{A}v(\psi_{t\cdot s}).\\
\end{align*}
\end{enumerate}

It remains to show items (a) and (b).
\begin{enumerate}[(a)]
\item Assume that ($F$) is an axiom scheme of $\mathbf{LJL}_{CS}$. Let $v\in \mathcal{W}_0^c$ and let $t\in Jt$. We then naturally have that
\[
v(\phi^\star\rightarrow\psi^\star)\land^\mathbf{A}v(\phi^\star)\leq^\mathbf{A} v(\psi^\star)
\]
for any $\phi,\psi\in\mathcal{L}_J$ as $v\in\mathsf{Ev}(\mathbf{A};\mathcal{L}_0^\star)$. Now, using that ($F$) is an axiom scheme of $\mathbf{LJL}_{CS}$, we obtain for any $\phi\in\mathcal{L}_J$ that
\[
v(\phi_t\rightarrow\phi^\star)=1^\mathbf{A}
\]
through $v[(Th_{\mathbf{LJL}_{CS}})^\star]\subseteq\{1^\mathbf{A}\}$ and thus
\[
v(\phi_t)\leq^\mathbf{A} v(\phi^\star)
\]
as $v\in\mathsf{Ev}(\mathbf{A};\mathcal{L}_0^\star)$ again. This gives $\mathcal{E}^c_t(v,v)=1^\mathbf{A}$.
\item Assume that ($I$) is an axiom scheme of $\mathbf{LJL}_{CS}$ and let again $v\in \mathcal{W}_0^c$, $w\in \mathcal{W}^c$ and $t\in Jt$. Assume $\mathcal{E}_{!t}(v,w)=1^\mathbf{A}$ and let $\phi\in\mathcal{L}_J$ be arbitrary. We have, as 
\[
v(\phi_t)\leq^\mathbf{A} v((t:\phi)_{!t})
\]
through $v\in \mathcal{W}_0^c$, that
\begin{align*}
V(v,t:\phi)&=v(\phi_t)\\
           &\leq^\mathbf{A} v((t:\phi)_{!t})\\
           &\leq^\mathbf{A} w(\phi_t)\\
           &= V(w,t:\phi)
\end{align*}
where the last inequality follows from $\mathcal{E}_{!t}(v,w)=1^\mathbf{A}$.
\end{enumerate}
\end{proof}
With a similar proof as in Theorem \ref{thm:algmkrtmodcomp}, we also obtain the following theorem.
\begin{theorem}\label{thm:algsubsetmodcomp}
Let $\mathbf{L}$ be an intermediate logic and let $\mathbf{LJL}_0\in\{\mathbf{LJ}_0,\mathbf{LJT}_0,\mathbf{LJ4}_0,\mathbf{LJT4}_0\}$ where $CS$ is a constant specification for $\mathbf{LJL}_0$. Let further $\mathsf{C}\in\mathsf{Alg}_{com}(\mathbf{L})$ and let $\mathsf{CASJL}$ be the class of algebraic subset models corresponding to $\mathbf{LJL}_0$ and $\mathsf{C}$.

For any $\Gamma\cup\{\phi\}\subseteq\mathcal{L}_J$, the following are equivalent:
\begin{enumerate}
\item $\Gamma\vdash_{\mathbf{LJL}_{CS}}\phi$;
\item $\Gamma\models_{\mathsf{CASJL}_{CS}}\phi$;
\item $\Gamma\models^1_{\mathsf{CASJL}_{CS}}\phi$;
\item $\Gamma\models^1_{\mathsf{CASJL^c}_{CS}}\phi$.
\end{enumerate}
\end{theorem}
\section{Frame semantics for intermediate justification logics}
As a second semantic approach, we extend not Heyting algebras but intuitionistic Kripke frames for intermediate logics with the semantic machinery of the models of Mkrtychev, Fitting or of Lehmann and Studer.

This extends the work on intuitionistic Mkrtychev and Fitting models (under different terminology) from Marti and Studer in \cite{MS2016} to wider classes of logics.
\subsection{Kripke frames and propositional semantics}
We review some concepts from Kripke frames for propositional intermediate logics (see e.g. \cite{Gab1981,Ono1971}). For this, we need some terminology from the context of the theory of partial orders first.
\begin{definition}
A \emph{Kripke frame} is a structure $\langle F,\leq\rangle$ such that $\leq$ is a reflexive, transitive and antisymmetric binary relation on the non-empty set $F$ (that is, a Kripke frame is just a partial order).
\end{definition}
A set $X\subseteq F$ is called a \emph{cone} (or \emph{upset}) if
\[
\forall x\in X\forall y\in F\left(x\leq y\Rightarrow y\in X\right).
\]
We denote the smallest cone containing a set $X$ of a partial order $\langle F,\leq\rangle$ by $\uparrow X$. A cone $X$ is called \emph{principal} if $X=\;\uparrow\{x\}$ for some element $x$. It is straightforward that
\[
\uparrow\{x\}=\{y\in F\mid y\geq x\}
\]
and that
\[
\uparrow X=\bigcup_{x\in X}\uparrow\{x\}.
\]
A Kripke frame $\mathfrak{G}=\langle G,\leq'\rangle$ is an \emph{(induced) subframe} of a Kripke frame $\mathfrak{F}=\langle F,\leq\rangle$ if $G\subseteq F$ and $\leq'=\leq\cap (G\times G)$. In this case, we also write $\mathfrak{G}=\mathfrak{F}\upharpoonright G$. A Kripke frame is called \emph{principal} if its domain is principal.
\begin{definition}
Let $\mathfrak{F}=\langle F,\leq\rangle$ be a Kripke frame. A (\emph{$\mathcal{L}_0(X)$-})\emph{Kripke model} based on $\mathfrak{F}$ is a structure $\mathfrak{M}=\langle\mathfrak{F},\Vdash\rangle$ with $\Vdash\subseteq F\times X$ which satisfies
\[
x\leq y\text{ and }x\Vdash p\text{ implies }y\Vdash p
\]
for all $p\in X$.
\end{definition}
A Kripke model $\mathfrak{N}=\langle\mathfrak{G},\Vdash'\rangle$ is called an (\emph{induced}) \emph{submodel} of a Kripke model $\mathfrak{M}=\langle\mathfrak{F},\Vdash\rangle$ if $\mathfrak{G}$ is an induced subframe of $\mathfrak{F}$ and for all $p\in X$:
\[
\{x\in G\mid x\Vdash' p\}=\{x\in F\mid x\Vdash p\}\cap G.
\]
We write $\mathfrak{N}=\mathfrak{M}\upharpoonright G$ in this case. 

Given a Kripke model $\mathfrak{M}=\langle\mathfrak{F},\Vdash\rangle$, we introduce the satisfaction relation $\models$ for formulae from $\mathcal{L}_0(X)$ as follows. Given a $x\in F$, we define recursively:
\begin{itemize}
\item $(\mathfrak{M},x)\not\models\bot$;
\item $(\mathfrak{M},x)\models p$ if $x\Vdash p$;
\item $(\mathfrak{M},x)\models\phi\land\psi$ if $(\mathfrak{M},x)\models\phi$ and $(\mathfrak{M},x)\models\psi$;
\item $(\mathfrak{M},x)\models\phi\lor\psi$ if $(\mathfrak{M},x)\models\phi$ or $(\mathfrak{M},x)\models\psi$;
\item $(\mathfrak{M},x)\models\phi\rightarrow\psi$ if $\forall y\in F\left( x\leq y\Rightarrow (\mathfrak{M},x)\not\models\phi\text{ or }(\mathfrak{M},x)\models\psi\right)$.
\end{itemize}
We write $\mathfrak{M}\models\phi$ if $(\mathfrak{M},x)\models\phi$ for any $x\in F$, $(\mathfrak{M},x)\models\Gamma$ if $(\mathfrak{M},x)\models\gamma$ for all $\gamma\in\Gamma$ and $\mathfrak{M}\models\Gamma$ if $(\mathfrak{M},x)\models\Gamma$ for all $x\in F$.

A fundamental property of Kripke models is that the monotonicity of propositional variables extends to all formulae. More precisely, we have the following:
\begin{lemma}\label{lem:propmono}
Let $\mathfrak{M}=\langle\mathfrak{F},\Vdash\rangle$ be a $\mathcal{L}_0(X)$-Kripke model. Then, for all $\phi\in\mathcal{L}_0(X)$ and all $x,y\in F$:
\[
x\leq y\text{ and }(\mathfrak{M},x)\models\phi\text{ implies }(\mathfrak{M},y)\models\phi.
\]
\end{lemma}
The proof is an easy induction on the structure of $\mathcal{L}_0(X)$. Given a class of Kripke frames $\mathsf{C}$, we write $\mathsf{Mod}(\mathsf{C};\mathcal{L}_0(X))$ for the class of all Kripke models over $\mathcal{L}_0(X)$ with underlying Kripke frames from $\mathsf{C}$. Given a single frame $\mathfrak{F}$, we also write $\mathsf{Mod}(\mathfrak{F};\mathcal{L}_0(X))$ for $\mathsf{Mod}(\{\mathfrak{F}\};\mathcal{L}_0(X))$.

Using these definitions, there are now two definitions of consequence to consider.
\begin{definition}
Let $\Gamma\cup\{\phi\}\subseteq\mathcal{L}_0(X)$ and $\mathsf{C}$ be a class of Kripke models. Then, we write:
\begin{enumerate}
\item $\Gamma\models_\mathsf{C}\phi$ if $\forall\mathfrak{M}\in\mathsf{C}\forall x\in\mathcal{D}(\mathfrak{M})\Big((\mathfrak{M},x)\models\Gamma\Rightarrow (\mathfrak{M},x)\models\phi\Big)$;
\item $\Gamma\models^g_\mathsf{C}\phi$ if $\forall\mathfrak{M}\in\mathsf{C}\Big(\mathfrak{M}\models\Gamma\Rightarrow \mathfrak{M}\models\phi\Big)$.
\end{enumerate}
Further, if $\mathsf{C}$ is now a class of Kripke frames, we write:
\begin{enumerate}
\setcounter{enumi}{2}
\item  $\Gamma\models_\mathsf{C}\phi$ if $\Gamma\models_{\mathsf{Mod}(\mathsf{C};\mathcal{L}_0(X))}\phi$;
\item $\Gamma\models^g_\mathsf{C}\phi$ if $\Gamma\models^g_{\mathsf{Mod}(\mathsf{C};\mathcal{L}_0(X))}\phi$.
\end{enumerate}
\end{definition}
\begin{definition}
Let $\mathbf{L}$ be an intermediate logic, $X$ a countably infinite set of variables and $\mathsf{C}$ be a class of Kripke frames. 
\begin{enumerate}
\item We say that $\mathbf{L}(X)$ is \emph{strongly complete} w.r.t. $\mathsf{C}$ if $\Gamma\vdash_{\mathbf{L}(X)}\phi$ is equivalent to $\Gamma\models_\mathsf{C}\phi$.
\item We say that $\mathbf{L}(X)$ is \emph{strongly globally complete} w.r.t. $\mathsf{C}$ if $\Gamma\vdash_{\mathbf{L}(X)}\phi$ is equivalent to $\Gamma\models_\mathsf{C}^g\phi$.
\end{enumerate}
Given a class of Kripke frames $\mathsf{C}$, we write $\mathsf{C}\in\mathsf{KFr}(\mathbf{L})$ or $\mathsf{C}\in\mathsf{KFr}^g(\mathbf{L})$ if $\mathbf{L}$ is strongly (locally) complete or strongly globally complete w.r.t. $\mathsf{C}$, respectively. We also write $\mathsf{C}\in\mathsf{KFr}(\mathbf{L})\cap\mathsf{KFr}^g(\mathbf{L})$ for $\mathsf{C}\in\mathsf{KFr}(\mathbf{L})$ and $\mathsf{C}\in\mathsf{KFr}^g(\mathbf{L})$.
\end{definition}
The global version will later prove to be important in the completeness considerations. Two things shall be noted in this context. First, it is well known that there are Kripke incomplete intermediate logics, that is intermediate logics where there is no class of Kripke frames for which the logic is (even weakly) complete. The first such logic was constructed in \cite{She1977}. All following considerations involving propositional completeness w.r.t. classes of Kripke frames thus implicitly assume that such a class exists.

Further, if an intermediate logic is characterized by a class of Kripke frames \emph{locally}, there is a simple extended class of frames which characterizes the logic \emph{globally}. More precisely, we have the following:
\begin{lemma}\label{lem:localglobalequiv}
Let $\mathsf{C}$ be a class of Kripke frames and let $\overline{\mathsf{C}}$ be the closure of $\mathsf{C}$ under principal subframes. Let $\Gamma\cup\{\phi\}\subseteq\mathcal{L}_0(X)$. Then, we have:
\begin{enumerate}
\item $\Gamma\models_\mathsf{C}\phi$ iff $\Gamma\models_{\overline{\mathsf{C}}}\phi$;
\item $\Gamma\models_{\overline{\mathsf{C}}}\phi$ iff $\Gamma\models^g_{\overline{\mathsf{C}}}\phi$.
\end{enumerate}
\end{lemma}
\begin{proof}
For (1), we naturally have the direction from right to left. For the converse, note that for all $\phi\in\mathcal{L}_0$, all $\mathfrak{M}$ over frames from $\mathsf{C}$ and all $x\in\mathcal{D}(\mathfrak{M})$, we have
\[
(\mathfrak{M},x)\models\phi\text{ iff }(\mathfrak{M}\upharpoonright(\uparrow\{x\}),x)\models\phi.
\]
Thus, the claim follows from the fact that for every $\mathfrak{F}\in\overline{\mathsf{C}}$, we have $\mathfrak{F}\in\mathsf{C}$ or $\mathfrak{F}=\mathfrak{G}\upharpoonright (\uparrow\{x\})$ for some $\mathfrak{G}\in\mathsf{C}$ and some $x\in\mathcal{D}(\mathfrak{G})$.

For (2), we immediately get the direction from left to right. For the converse, consider $\Gamma\models^g_{\overline{\mathsf{C}}}\phi$, that is
\[
\forall\mathfrak{N}\in\mathsf{Mod}(\overline{\mathsf{C}};\mathcal{L}_0) \Big(\mathfrak{N}\models\Gamma\Rightarrow\mathfrak{N}\models\phi\Big).
\]
Let $\mathfrak{F}\in\overline{\mathsf{C}}$ and $\mathfrak{M}\in\mathsf{Mod}(\mathfrak{F};\mathcal{L}_0)$ as well as $x\in F$ and suppose $(\mathfrak{M},x)\models\Gamma$. Consider 
\[
\mathfrak{M}':=\mathfrak{M}\upharpoonright(\uparrow\{x\}).
\]
Then, we obtain $\mathfrak{M}'\models\Gamma$ by Lemma \ref{lem:propmono} and we have $\mathfrak{M}'\in\mathsf{Mod}(\overline{\mathsf{C}},\mathcal{L}_0)$ as $\overline{\mathsf{C}}$ is closed under principal subframes. We thus have $\mathfrak{M}'\models\phi$ by $\Gamma\models_{\overline{\mathsf{C}}}^g\phi$. This gives especially $(\mathfrak{M},x)\models\phi$ as we have
\[
(\mathfrak{M},x)\models\phi\text{ iff }(\mathfrak{M}\upharpoonright(\uparrow\{x\}),x)\models\phi.
\]
as above. Thus, we have $\Gamma\models_{\overline{\mathsf{C}}}\phi$.
\end{proof}
\subsection{Intuitionistic Mkrtychev models}
We continue our semantical investigations into intermediate justification logics by extending the approach of Mkrtychevs syntactic models by intuitionistic Kripke frames. These intuitionistic Mkrtychev models are akin to the previously considered models from \cite{MS2016} for $\mathbf{IPCJT4}$ (under the name of intuitionistic basic models).
\begin{definition}\label{def:genmkrtmod}
Let $\mathfrak{F}=\langle F,\leq\rangle$ be a Kripke frame. An \emph{intuitionistic Mkrtychev model based on $\mathfrak{F}$} is a structure $\mathfrak{M}=\langle\mathfrak{F},\mathcal{E},\Vdash\rangle$ such that $\Vdash\subseteq F\times Var$ and $\mathcal{E}:Jt\times F\to 2^{\mathcal{L}_J}$ satisfy
\begin{enumerate}
\item $x\leq y$ and $x\Vdash p$ implies $y\Vdash p$ for all $p\in Var$,
\item $x\leq y$ and $\phi\in\mathcal{E}_t(x)$ implies $\phi\in\mathcal{E}_t(x)$ for all $\phi\in\mathcal{L}_J$ and all $t\in Jt$,
\end{enumerate}
for all $x,y\in F$ as well as 
\begin{enumerate}[(i)]
\item $\mathcal{E}_t(x)\sqsupset\mathcal{E}_s(x)\subseteq\mathcal{E}_{[t\cdot s]}(x)$ for all $x\in F$ and all $t,s\in Jt$,
\item $\mathcal{E}_t(x)\cup\mathcal{E}_s(x)\subseteq\mathcal{E}_{[t+s]}(x)$ for all $x\in F$ and all $t,s\in Jt$,
\end{enumerate}
where
\[
\Gamma\sqsupset\Delta:=\{\phi\in\mathcal{L}_J\mid\psi\rightarrow\phi\in\Gamma,\psi\in\Delta\text{ for some }\psi\in\mathcal{L}_J\}
\]
for $\Gamma,\Delta\subseteq\mathcal{L}_J$.
\end{definition}
Given an intuitionistic Mkrtychev model $\mathfrak{M}$ over a Kripke frame $\mathfrak{F}=\langle F,\leq\rangle$, we also write $\mathcal{D}(\mathfrak{M}):= F$ and call $F$ the domain of $\mathfrak{M}$. Note that we use $F$ to denote the domain of a model but also $(F)$ to denote the axiom scheme of factivity for the intermediate justification logics. 

Over an intuitionistic Mkrtychev model $\mathfrak{M}=\langle\mathfrak{F},\mathcal{E},\Vdash\rangle$, we introduce the following local satisfaction relation by recursion:
\begin{itemize}
\item $(\mathfrak{M},x)\not\models\bot$;
\item $(\mathfrak{M},x)\models p$ if $x\Vdash p$;
\item $(\mathfrak{M},x)\models\phi\land\psi$ if $(\mathfrak{M},x)\models\phi$ and $(\mathfrak{M},x)\models\psi$;
\item $(\mathfrak{M},x)\models\phi\lor\psi$ if $(\mathfrak{M},x)\models\phi$ or $(\mathfrak{M},x)\models\psi$;
\item $(\mathfrak{M},x)\models\phi\rightarrow\psi$ if $\forall y\in F\left( x\leq y\Rightarrow (\mathfrak{M},x)\not\models\phi\text{ or }(\mathfrak{M},x)\models\psi\right)$;
\item $(\mathfrak{M},x)\models t:\phi$ if $\phi\in\mathcal{E}_t(x)$.
\end{itemize}
We write $(\mathfrak{M},x)\models\Gamma$ if $(\mathfrak{M},x)\models\gamma$ for all $\gamma\in\Gamma$. Further, we have the following immediate lemma.
\begin{lemma}
Let $\mathfrak{F}=\langle F,\leq\rangle$ be a Kripke frame and let $\mathfrak{M}$ be an intuitionistic Mkrtychev model over $\mathfrak{F}$. For any $\phi\in\mathcal{L}_J$ and all $x,y\in F$: 
\[
x\leq y\text{ and }(\mathfrak{M},x)\models\phi\text{ implies }(\mathfrak{M},y)\models\phi.
\]
\end{lemma}
\begin{definition}\label{def:genmkrtmodspecialclasses}
Let $\mathfrak{F}=\langle F,\leq\rangle$ be a Kripke frame and $\mathfrak{M}=\langle\mathfrak{F},\mathcal{E},\Vdash\rangle$ be an intuitionistic Mkrtychev model. We call $\mathfrak{M}$
\begin{enumerate}
\item \emph{factive} if $\phi\in\mathcal{E}_t(x)$ implies $(\mathfrak{M},x)\models\phi$ for all $t\in Jt$, all $\phi\in\mathcal{L}_J$ and all $x\in F$, and
\item \emph{introspective} if $t:\mathcal{E}_t(x)\subseteq\mathcal{E}_{!t}(x)$ where $t:\Gamma=\{t:\gamma\mid\gamma\in\Gamma\}$ for all $t\in Jt$ and all $x\in F$.
\end{enumerate}
\end{definition}
\begin{definition}
Let $\mathsf{C}$ be a class of Kripke frames. Then, we write:
\begin{enumerate}
\item $\mathsf{CKMJ}$ for the class of all intuitionistic Mkrtychev models over frames from $\mathsf{C}$;
\item $\mathsf{CKMJT}$ for the class of all factive intuitionistic Mkrtychev models over frames from $\mathsf{C}$;
\item $\mathsf{CKMJ4}$ for the class of all introspective intuitionistic Mkrtychev models over frames from $\mathsf{C}$;
\item $\mathsf{CKMJT4}$ for the class of all factive and introspective intuitionistic Mkrtychev models over $\mathsf{C}$.
\end{enumerate}
\end{definition}
\begin{definition}
Let $\mathfrak{M}=\langle\mathfrak{F},\mathcal{E},\Vdash\rangle$ be an intuitionistic Mkrtychev model and let $CS$ be a constant specification (for some proof calculus). We say that $\mathfrak{M}$ \emph{respects} $CS$ if for all $x\in F$ and all $c:\phi\in CS$: $\phi\in\mathcal{E}_c(x)$.
\end{definition}
Given a class $\mathsf{C}$ of intuitionistic Mkrtychev models, we denote the class of all intuitionistic Mkrtychev models respecting a constant specification $CS$ by $\mathsf{C}_{CS}$.
\begin{definition}
Let $\mathsf{C}$ be a class of intuitionistic Mkrtychev models. We write $\Gamma\models_\mathsf{C}\phi$ if for all $\mathfrak{M}\in\mathsf{C}$ and all $x\in\mathcal{D}(\mathfrak{M})$: $(\mathfrak{M},x)\models\Gamma$ implies $(\mathfrak{M},x)\models\phi$.
\end{definition}
\begin{lemma}\label{lem:genmkrtmodsoundness}
Let $\mathbf{L}$ be an intermediate logic and $\mathbf{LJL}_0\in\{\mathbf{LJ}_0,\mathbf{LJT}_0,\mathbf{LJ4}_0,\mathbf{LJT4}_0\}$. Let $CS$ be a constant specification for $\mathbf{LJL}_0$ and let $\mathsf{C}\in\mathsf{KFr}(\mathbf{L})$. Let $\mathsf{CKMJL}$ be the class of intuitionistic Mkrtychev models corresponding to $\mathbf{LJL}_0$ and $\mathsf{C}$. For any $\Gamma\cup\{\phi\}\subseteq\mathcal{L}_J$:
\[
\Gamma\vdash_{\mathbf{LJL}_{CS}}\phi\text{ implies }\Gamma\models_{\mathsf{CKMJL}_{CS}}\phi.
\]
\end{lemma}
\begin{proof}
By an argument similar to the one of Lemma \ref{lem:algmkrtmodsoundness}, we may reduce strong to weak soundness. Thus, we only show that $\vdash_{\mathbf{LJL}_0}\phi$ implies $\models_{\mathsf{CKMJL}}\phi$. As before, by definition of $\mathbf{LJL}_0$, it suffices to show $\models_{\mathsf{CKMJL}}\phi$ for $\phi\in\overline{\mathbf{L}}$ or $\phi$ being an instance of the justification axioms (depending on $\mathbf{LJL}_0$). For both, let $\mathfrak{M}=\langle\mathfrak{F},\mathcal{E},\Vdash\rangle\in\mathsf{CKMJL}$ as well as $x\in\mathcal{D}(\mathfrak{M})$.

If $\phi\in\overline{\mathbf{L}}$, then there is a substitution $\sigma:Var\to\mathcal{L}_J$ such that $\phi=\sigma(\psi)$ for some $\psi\in\mathbf{L}$. By the choice of $\mathsf{C}$, we have that $(\mathfrak{N},y)\models\psi$ for any $\mathfrak{N}=\langle\mathfrak{F},\Vdash'\rangle$ and any $y\in\mathcal{D}(\mathfrak{F})$. Define a particular $\Vdash'$ by 
\[
y\Vdash' p\text{ iff }(\mathfrak{M},y)\models\sigma(p)
\]
for any $p\in Var$ and any $y\in F$ and define $\mathfrak{N}:=\langle\mathfrak{F},\Vdash'\rangle$. Then, it is straightforward to see that $(\mathfrak{M},y)\models\sigma(\chi)$ iff $(\mathfrak{N},y)\models\chi$ for any $\chi\in\mathcal{L}_0$ and thus especially, we have $(\mathfrak{M},x)\models\phi$. This gives $\models_{\mathsf{CKMJL}}\phi$.

If $\phi$ is an instance of $(J)$ or $(+)$, then the conditions (i) and (ii) of Definition \ref{def:genmkrtmod}, respectively, give the validity of $\phi$ immediately.

Similarly if $\phi$ is an instance of $(F)$ and $\mathfrak{M}$ is factive or $\phi$ is an instance of $(I)$ and $\mathfrak{M}$ is introspective, the respective validity of $\phi$ follows immediately by the definition of factive or introspective intuitionistic Mkrtychev models, that is (1) or (2) of Definition \ref{def:genmkrtmodspecialclasses}.
\end{proof}
\subsection{Intuitionistic Fitting models}
We continue with intuitionistic Fitting models, combining various streams of semantics in non-classical modal logics by extending the approach using intuitionistic modal Kripke models of \cite{Ono1977} for intuitionistic modal logics by the machinery of evidence functions for explicit modalities in the sense of Fitting (or conversely extending Fitting's models with the machinery of intuitionistic Kripke frames). In any way, the models which we introduce are akin to a model class from \cite{MS2016} for $\mathbf{IPCJT4}$ (which are called intuitionistic modular models there).
\begin{definition}\label{def:genfittingmod}
Let $\mathfrak{F}=\langle F,\leq\rangle$ be a Kripke frame. An \emph{intuitionistic Fitting model based on $\mathfrak{F}$} is a structure $\mathfrak{M}=\langle\mathfrak{F},\mathcal{R},\mathcal{E},\Vdash\rangle$ such that $\Vdash\subseteq F\times Var$, $\mathcal{R}\subseteq F\times F$ and $\mathcal{E}:Jt\times F\to 2^{\mathcal{L}_J}$ satisfy
\begin{enumerate}
\item $x\leq y$ and $x\Vdash p$ imply $y\Vdash p$ for all $p\in Var$,
\item $x\leq y$ and $\phi\in\mathcal{E}_t(x)$ imply $\phi\in\mathcal{E}_t(y)$ for all $\phi\in\mathcal{L}_J$ and all $t\in Jt$,
\item $x\leq y$ implies $\mathcal{R}[y]\subseteq\mathcal{R}[x]$ where $\mathcal{R}[x]:=\{z\in F\mid (x,z)\in\mathcal{R}\}$,
\end{enumerate}
for all $x,y\in F$ as well as 
\begin{enumerate}[(i)]
\item $\mathcal{E}_t(x)\sqsupset\mathcal{E}_s(x)\subseteq\mathcal{E}_{[t\cdot s]}(x)$ for all $x\in F$ and all $t,s\in Jt$,
\item $\mathcal{E}_t(x)\cup\mathcal{E}_s(x)\subseteq\mathcal{E}_{[t+s]}(x)$ for all $x\in F$ and all $t,s\in Jt$.
\end{enumerate}
\end{definition}
Over an intuitionistic Fitting model $\mathfrak{M}=\langle\mathfrak{F},\mathcal{R},\mathcal{E},\Vdash\rangle$, we introduce the following local satisfaction relation by recursion:
\begin{itemize}
\item $(\mathfrak{M},x)\not\models\bot$;
\item $(\mathfrak{M},x)\models p$ if $x\Vdash p$;
\item $(\mathfrak{M},x)\models\phi\land\psi$ if $(\mathfrak{M},x)\models\phi$ and $(\mathfrak{M},x)\models\psi$;
\item $(\mathfrak{M},x)\models\phi\lor\psi$ if $(\mathfrak{M},x)\models\phi$ or $(\mathfrak{M},x)\models\psi$;
\item $(\mathfrak{M},x)\models\phi\rightarrow\psi$ if $\forall y\in F\left( x\leq y\Rightarrow (\mathfrak{M},x)\not\models\phi\text{ or }(\mathfrak{M},x)\models\psi\right)$;
\item $(\mathfrak{M},x)\models t:\phi$ if $\phi\in\mathcal{E}_t(x)$ and $\forall y\in\mathcal{R}[x]\;(\mathfrak{M},y)\models\phi$.
\end{itemize}
We write $(\mathfrak{M},x)\models\Gamma$ if $(\mathfrak{M},x)\models\gamma$ for all $\gamma\in\Gamma$. Also, given an intuitionistic Fitting model $\mathfrak{M}$ over a Kripke frame $\mathfrak{F}=\langle F,\leq,V\rangle$, we write again $\mathcal{D}(\mathfrak{M})=F$.
\begin{lemma}
Let $\mathfrak{M}$ be an intuitionistic Fitting model over a Kripke frame $\mathfrak{F}=\langle F,\leq\rangle$. For any $\phi\in\mathcal{L}_J$ and all $x,y\in F$: 
\[
x\leq y\text{ and }(\mathfrak{M},x)\models\phi\text{ imply }(\mathfrak{M},y)\models\phi.
\]
\end{lemma}
\begin{definition}
Let $\mathfrak{F}=\langle F,\leq\rangle$ be a Kripke frame and $\mathfrak{M}=\langle\mathfrak{F},\mathcal{R},\mathcal{E},\Vdash\rangle$ be an intuitionistic Fitting model. We call $\mathfrak{M}$
\begin{enumerate}
\item \emph{reflexive} if $\mathcal{R}$ is reflexive,
\item \emph{transitive} if $\mathcal{R}$ is transitive,
\item \emph{monotone} if $\mathcal{E}_t(x)\subseteq\mathcal{E}_t(y)$ for $y\in\mathcal{R}[x]$ and for all $t\in Jt$ and all $x\in F$,
\item \emph{introspective} if it is transitive, monotone and $t:\mathcal{E}_t(x)\subseteq\mathcal{E}_{!t}(x)$ for all $t\in Jt$ and all $x\in F$.
\end{enumerate}
\end{definition}
\begin{definition}
Let $\mathsf{C}$ be a class of intuitionistic Fitting models. We write $\Gamma\models_\mathsf{C}\phi$ if for all $\mathfrak{M}\in\mathsf{C}$ and all $x\in\mathcal{D}(\mathfrak{M})$: $(\mathfrak{M},x)\models\Gamma$ implies $(\mathfrak{M},x)\models\phi$.
\end{definition}
\begin{definition}
Let $\mathsf{C}$ be a class of Kripke frames. Then, we write:
\begin{enumerate}
\item $\mathsf{CKFJ}$ for the class of all intuitionistic Fitting models over frames from $\mathsf{C}$;
\item $\mathsf{CKFJT}$ for the class of all reflexive intuitionistic Fitting models over frames from $\mathsf{C}$;
\item $\mathsf{CKFJ4}$ for the class of all introspective intuitionistic Fitting models over frames from $\mathsf{C}$;
\item $\mathsf{CKFJT4}$ for the class of all reflexive and introspective intuitionistic Fitting models over frames from $\mathsf{C}$.
\end{enumerate}
\end{definition}
\begin{lemma}\label{lem:genfitmodsoundness}
Let $\mathbf{L}$ be an intermediate logic and $\mathbf{LJL}_0\in\{\mathbf{LJ}_0,\mathbf{LJT}_0,\mathbf{LJ4}_0,\mathbf{LJT4}_0\}$. Let $CS$ be a constant specification for  $\mathbf{LJL}_0$ and let $\mathsf{C}\in\mathsf{KFr}(\mathbf{L})$. Let $\mathsf{CKFJL}$ be the class of intuitionistic Fitting models corresponding to $\mathbf{LJL}_0$ and $\mathsf{C}$. Then, for any $\Gamma\cup\{\phi\}\subseteq\mathcal{L}_J$:
\[
\Gamma\vdash_{\mathbf{LJL}_{CS}}\phi\text{ implies }\Gamma\models_{\mathsf{CKFJL}_{CS}}\phi.
\]
\end{lemma}
\begin{proof}
As in Lemma \ref{lem:genmkrtmodsoundness}, we may restrict ourselves to weak soundness only. Here, it again suffices to only verify $(\mathfrak{M},x)\models\phi$ for any $\phi\in\overline{\mathbf{L}}$ or $\phi$ being an instance of a justification axiom (depending on $\mathbf{LJL}_0$) as well as any $\mathfrak{M}\in\mathsf{CKFJL}$ and any $x\in\mathcal{D}(\mathfrak{M})$.

The case for $\phi\in\overline{\mathbf{L}}$ can be handled similarly as in Lemma \ref{lem:genmkrtmodsoundness}. We thus only show the validity of (1) $(J)$, (2) $(+)$ as well as (3) $(F)$ and (4) $(I)$ in their respective model classes. For this, let $\mathfrak{M}=\langle\mathfrak{F},\mathcal{R},\mathcal{E},\Vdash\rangle$ be an intuitionistic Fitting model with $\mathfrak{F}\in\mathsf{C}$.
\begin{enumerate}
\item We show
\[
(\mathfrak{M},x)\models t:(\phi\rightarrow\psi)\text{ impl. }(\mathfrak{M},x)\models (s:\phi\rightarrow[t\cdot s]:\psi)
\]
for any $\phi,\psi\in\mathcal{L}_J$, any $t,s\in Jt$ and any $x\in F$. For this, suppose $(\mathfrak{M},x)\models t:(\phi\rightarrow\psi)$, that is by definition $\phi\rightarrow\psi\in\mathcal{E}_t(x)$ as well as
\[
\forall y\in\mathcal{R}[x]\:(\mathfrak{M},y)\models\phi\rightarrow\psi.
\]
Let $y\geq x$ and suppose $(\mathfrak{M},y)\models s:\phi$, that is $\phi\in\mathcal{E}_s(y)$ and
\[
\forall z\in\mathcal{R}[y]\:(\mathfrak{M},z)\models\phi.
\]
By condition (2) of Definition \ref{def:genfittingmod}, this yields $\phi\rightarrow\psi\in\mathcal{E}_t(y)$. By condition (i), we have thus that $\psi\in\mathcal{E}_{[t\cdot s]}(y)$. Now, let $z\in\mathcal{R}[y]$. As above, we have $(\mathfrak{M},z)\models\phi$ and by condition (3) of Definition \ref{def:genfittingmod}, we have that $z\in\mathcal{R}[x]$ and thus $(\mathfrak{M},z)\models\phi\rightarrow\psi$. Hence, we obtain $(\mathfrak{M},z)\models\psi$.

Therefore, we have $\forall z\in\mathcal{R}[y]\:(\mathfrak{M},z)\models\psi$ and in combination with $\psi\in\mathcal{E}_{[t\cdot s]}(y)$, we get $(\mathfrak{M},y)\models [t\cdot s]:\psi$. As $y$ was arbitrary, we have $(\mathfrak{M},x)\models s:\phi\rightarrow [t\cdot s]:\psi$.
 
As $x\in F$ was arbitrary, it is that $(\mathfrak{M},x)\models t:(\phi\rightarrow\psi)\rightarrow (s:\phi\rightarrow[t\cdot s]:\psi)$.
\item Let $x\in F$ be arbitrary. Suppose $(\mathfrak{M},x)\models t:\phi$, that is
\[
\phi\in\mathcal{E}_t(x)\text{ and }\forall y\in\mathcal{R}[x]\; (\mathfrak{M},x)\models\phi.
\]
The former gives $\phi\in\mathcal{E}_{[t+s]}(x)$ by condition (ii) of Definition \ref{def:genfittingmod} and this combined with the latter gives
\[
(\mathfrak{M},x)\models [t+s]:\phi.
\]
As $x\in F$ war arbitrary, this yields $(\mathfrak{M},x)\models t:\phi\rightarrow[t+s]:\phi$. One shows $(\mathfrak{M},x)\models s:\phi\rightarrow[t+s]:\phi$ in a similar way.
\item Suppose $\mathfrak{M}$ is reflexive and let $x\in F$. Suppose
\[
(\mathfrak{M},x)\models t:\phi
\]
that is especially we have $\forall y\in\mathcal{R}[x]\;(\mathfrak{M},y)\models\phi$. As $\mathcal{R}$ is reflexive, we have $x\in\mathcal{R}[x]$ and thus $(\mathfrak{M},x)\models\phi$. As $x$ was arbitrary, we have
\[
(\mathfrak{M},x)\models t:\phi\rightarrow\phi.
\]
\item Let $\mathfrak{M}$ be introspective and let $x\in F$. Suppose that $(\mathfrak{M},x)\models t:\phi$, that is
\[
\phi\in\mathcal{E}_t(x)\text{ and }\forall y\in\mathcal{R}[x]\; (\mathfrak{M},x)\models\phi.
\]
The former gives at first $t:\phi\in\mathcal{E}_{!t}(x)$ by introspectivity. Now, let $y\in\mathcal{R}[x]$ be arbitrary. By the monotonicity aspect of introspectivity, we have $\phi\in\mathcal{E}_t(y)$ as $\phi\in\mathcal{E}_t(x)$. Now, let $z\in\mathcal{R}[y]$. By transitivity of $\mathcal{R}$, this implies $z\in\mathcal{R}[x]$ and thus $(\mathfrak{M},z)\models \phi$ by assumption. Summarized, we get
\[
\phi\in\mathcal{E}_t(y)\text{ and }\forall z\in\mathcal{R}[y]\; (\mathfrak{M},z)\models\phi,
\]
that is $(\mathfrak{M},y)\models t:\phi$ for all $y\in\mathcal{R}[x]$ and this combined with $t:\phi\in\mathcal{E}_{!t}(x)$ gives $(\mathfrak{M},x)\models !t:t:\phi$. As $x$ was arbitrary, we have $(\mathfrak{M},x)\models t:\phi\rightarrow !t:t:\phi$.
\end{enumerate}
\end{proof}
\subsection{Intuitionistic subset models}
The last semantics which we introduce, based on intuitionistic Kripke frames, extends the considerations of Lehmann and Studer from \cite{LS2019} about their subset models to these intermediate cases. This semantics seems to have not appeared in the literature before.
\begin{definition}\label{def:gensubsetmod}
Let $\mathfrak{F}=\langle F_0,\leq\rangle$ be a Kripke frame. An \emph{intuitionistic subset model over $\mathfrak{F}$} is a structure $\mathfrak{M}=\langle\mathfrak{F},F,\mathcal{E},\Vdash\rangle$ with $F\supseteq F_0$ $\mathcal{E}:Jt\to 2^{F\times F}$ and $\Vdash\subseteq F\times\mathcal{L}_J$ and which satisfies
\begin{enumerate}
\item $x\leq y$ and $x\Vdash p$ imply $y\Vdash p$ for all $p\in Var$,
\item $x\leq y$ implies $\mathcal{E}_t[y]\subseteq\mathcal{E}_t[x]$ where $\mathcal{E}_t[x]:=\{z\in F\mid (x,z)\in\mathcal{E}_t\}$ for all $t\in Jt$,
\end{enumerate}
for all $x,y\in F_0$ as well as 
\begin{enumerate}[(i)]
\item $x\not\Vdash\bot$,
\item $x\Vdash\phi\land\psi$ iff $x\Vdash\phi$ and $x\Vdash\psi$,
\item $x\Vdash\phi\lor\psi$ iff $x\Vdash\phi$ or $x\Vdash\psi$,
\item $x\Vdash\phi\rightarrow\psi$ iff $\forall y\geq x: y\not\Vdash\phi$ or $y\Vdash\psi$,
\item $x\Vdash t:\phi$ iff $\forall y\in\mathcal{E}_t[x]: y\Vdash\phi$,
\end{enumerate}
for any $x\in F_0$ and such that it satisfies:
\begin{enumerate}[(a)]
\item $\mathcal{E}_{[t+s]}[x]\subseteq\mathcal{E}_t[x]\cap\mathcal{E}_s[x]$;
\item $\mathcal{E}_{[t\cdot s]}[x]\subseteq\{y\in F\mid \forall\phi\in(\mathfrak{M})^x_{t,s}(y\Vdash\phi)\}$ where we define
\[
(\mathfrak{M})^v_{t,s}:=\{\phi\in\mathcal{L}_J\mid\exists\psi\in\mathcal{L}_J\forall y\in F(y\in\mathcal{E}_t[x]\Rightarrow y\Vdash\psi\rightarrow\phi\text{ and }y\in\mathcal{E}_s[x]\Rightarrow y\Vdash\psi)\}.
\]
\end{enumerate}
\end{definition}
We write $\mathcal{D}_0(\mathfrak{M})$ for $F_0$ and $\mathcal{D}(\mathfrak{M})$ for $F$. Also, given $x\in\mathcal{D}(\mathfrak{M})$, we write $(\mathfrak{M},x)\models\phi$ if $x\Vdash\phi$ and $(\mathfrak{M},x)\models\Gamma$ if $(\mathfrak{M},x)\models\gamma$ for all $\gamma\in\Gamma$, given $\Gamma\cup\{\phi\}\subseteq\mathcal{L}_J$. We write $\mathfrak{M}\models\phi$ if for all $x\in\mathcal{D}_0(\mathfrak{M})$, we have $(\mathfrak{M},x)\models\phi$ and similarly for sets $\Gamma$. Note the emphasis on $\mathcal{D}_0(\mathfrak{M})$, not $\mathcal{D}(\mathfrak{M})$.
\begin{lemma}\label{lem:gensubsetmodmono}
Let $\mathfrak{F}=\langle F_0,\leq\rangle$ be a Kripke frame and $\mathfrak{M}=\langle\mathfrak{F},F,\mathcal{E},\Vdash\rangle$ be an intuitionistic subset model over $\mathfrak{F}$. Then, for all $\phi\in\mathcal{L}_J$ and all $x,y\in F_0$:
\[
x\leq y\text{ and }x\Vdash\phi\text{ imply }y\Vdash\phi. 
\]
\end{lemma}
\begin{definition}
Let $\mathfrak{F}=\langle F_0,\leq\rangle$ be a Kripke frame and $\mathfrak{M}=\langle\mathfrak{F},F,\mathcal{E},\Vdash\rangle$ be an intuitionistic subset model. We call $\mathfrak{M}$
\begin{enumerate}
\item \emph{reflexive} if $x\in\mathcal{E}_t[x]$ for all $x\in F_0$ and all $t\in Jt$,
\item \emph{introspective} if $\mathcal{E}_{!t}[x]\subseteq\{y\in F\mid\forall\phi\in\mathcal{L}_J(x\Vdash t:\phi\Rightarrow y\Vdash t:\phi)\}$ for all $x\in F_0$ and all $t\in Jt$.
\end{enumerate}
\end{definition}
\begin{definition}
Let $\mathsf{C}$ be a class of intuitionistic subset models. We write $\Gamma\models_\mathsf{C}\phi$ if for all $\mathfrak{M}\in\mathsf{C}$ and all $x\in\mathcal{D}_0(\mathfrak{M})$: $(\mathfrak{M},x)\models\Gamma$ implies $(\mathfrak{M},x)\models\phi$.
\end{definition}
\begin{definition}
Let $\mathsf{C}$ be a class of Kripke frames. Then, we write:
\begin{enumerate}
\item $\mathsf{CKSJ}$ for the class of all intuitionistic subset models over frames from $\mathsf{C}$;
\item $\mathsf{CKSJT}$ for the class of all reflexive intuitionistic subset models over frames from $\mathsf{C}$;
\item $\mathsf{CKSJ4}$ for the class of all introspective intuitionistic subset models over frames from $\mathsf{C}$;
\item $\mathsf{CKSJT4}$ for the class of all reflexive and introspective intuitionistic subset models over frames from $\mathsf{C}$.
\end{enumerate}
\end{definition}
\begin{definition}
Let $\mathfrak{F}=\langle F_0,\leq\rangle$ be a Kripke frame and $\mathfrak{M}=\langle\mathfrak{F},F,\mathcal{E},\Vdash\rangle$ be an intuitionistic subset model over $\mathfrak{F}$. Let $CS$ be a constant specification (for some proof system). We say that $\mathfrak{M}$ \emph{respects} $CS$ if
\[
x\Vdash c:\phi
\]
for all $c:\phi\in CS$ and all $x\in F_0$.
\end{definition}
Given a class $\mathsf{C}$ of intuitionistic subset models, we write $\mathsf{C}_{CS}$ for the subclass of all models respecting a constant specification $CS$.
\begin{lemma}\label{lem:gensubsetmodsoundness}
Let $\mathbf{L}$ be an intermediate logic and $\mathbf{LJL}_0\in\{\mathbf{LJ}_0,\mathbf{LJT}_0,\mathbf{LJ4}_0,\mathbf{LJT4}_0\}$. Let $CS$ be a constant specification for  $\mathbf{LJL}_0$. Let $\mathsf{C}\in\mathsf{KFr}(\mathbf{L})$ and let $\mathsf{CKSJL}$ be the class of intuitionistic subset models corresponding to $\mathbf{LJL}_0$ and $\mathsf{C}$. For any $\Gamma\cup\{\phi\}\subseteq\mathcal{L}_J$:
\[
\Gamma\vdash_{\mathbf{LJL}_{CS}}\phi\text{ implies }\Gamma\models_{\mathsf{CKSJL}_{CS}}\phi.
\]
\end{lemma}
\begin{proof}
Reasoning as in Lemmas \ref{lem:genmkrtmodsoundness} and \ref{lem:genfitmodsoundness}, we restrict the argument and only show the validity of $(J),(+),(F)$ and $(I)$ in their respective model classes. For this, let $\mathfrak{M}=\langle\mathfrak{F},F,\mathcal{E},\Vdash\rangle$ be an intuitionistic subset model over a Kripke frame $\mathfrak{F}$ and let $x\in F_0$.
\begin{enumerate}
\item For $(J)$, suppose $x\Vdash t:(\phi\rightarrow\psi)$. Now, we want to show $x\Vdash s:\phi\rightarrow [t\cdot s]:\psi$. For this, let $y\geq x$. Note that then $y\in F_0$ as $\leq$ is only a relation on $F_0$. Suppose $y\Vdash s:\phi$, that is
\[
\forall z\in\mathcal{E}_s[y]\;z\Vdash\phi.\tag{$\dagger$}
\]
Further, $x\Vdash t:(\phi\rightarrow\psi)$ implies $y\Vdash t:(\phi\rightarrow\psi)$ by Lemma \ref{lem:gensubsetmodmono}, that is
\[
\forall z\in\mathcal{E}_t[y]\;z\Vdash\phi\rightarrow\psi.\tag{$\ddagger$}
\]
Let $z\in\mathcal{E}_{[t\cdot s]}[y]$. Then, by property (b) of Definition \ref{def:gensubsetmod}, we have 
\[
z\in\{w\in F\mid\forall\chi\in(\mathfrak{M})^y_{t,s}\; w\Vdash\chi\}
\]
and thus it suffices to show $\psi\in(\mathfrak{M})^y_{t,s}$. But this is clear from $(\dagger)$ and $(\ddagger)$. Thus,
\[
\forall z\in\mathcal{E}_{[t\cdot s]}[y]\;z\Vdash\psi
\]
which is $y\Vdash [t\cdot s]:\psi$ and consequently $x\Vdash s:\phi\rightarrow [t\cdot s]:\psi$. As $x$ was arbitrary, we have $x\Vdash t:(\phi\rightarrow\psi)\rightarrow (s:\phi\rightarrow [t\cdot s]:\psi)$ for all $x\in F_0$.
\item Suppose $x\Vdash t:\phi$. That is,
\[
\forall y\in\mathcal{E}_t[x]\;y\Vdash\phi.
\]
and by condition (a) of Definition \ref{def:gensubsetmod}, we have
\[
\forall y\in\mathcal{E}_{[t+s]}[x]\subseteq\mathcal{E}_t[x]\; y\Vdash\phi
\]
which is $x\Vdash [t+s]:\phi$. As $x$ was arbitrary, we have $x\Vdash t:\phi\rightarrow [t+s]:\phi$ for any $x\in F_0$. Similarly, one shows $x\Vdash s:\phi\rightarrow [t+s]:\phi$ for any $x\in F_0$.
\item Let $\mathfrak{M}$ be reflexive with $x\Vdash t:\phi$. Then, we have
\[
\forall y\in\mathcal{E}_t[x]\; y\Vdash\phi,
\]
that is as $\mathfrak{M}$ is reflexive $x\in\mathcal{E}_t[x]$ and thus $x\Vdash\phi$. As $x$ was arbitrary, we obtain $x\Vdash t:\phi\rightarrow\phi$ for any $x\in F_0$.
\item Let $\mathfrak{M}$ be introspective $x\Vdash t:\phi$. Then, we have
\[
\forall y\in\mathcal{E}_{!t}[x]\; y\Vdash t:\phi
\]
by definition of introspectivity but this is exactly $x\Vdash !t:t:\phi$. Again, we have $x\Vdash t:\phi\rightarrow !t:t:\phi$ for any $x\in F_0$ as $x$ was arbitrary.
\end{enumerate}
\end{proof}
\section{Completeness for frame semantics}\label{sec:comptheoframes}
In this section, we prove the corresponding completeness theorems for the intermediate justification logics together with their previously introduced semantics based on Mkrtychev, Fitting or subset models over intuitionistic Kripke frames. The permissible classes of frames for the completeness theorems derive, similarly as the permissible classes of Heyting algebras from the completeness theorems for the algebraic models, from the underlying intermediate logic where we especially rely on the \emph{global} completeness statement introduced earlier.
\subsection{Completeness w.r.t. intuitionistic Mkrtychev models}
\begin{definition}
Let $\mathfrak{F}=\langle F,\leq\rangle$ be a Kripke frame and let $\mathfrak{N}\in\mathsf{Mod}(\mathfrak{F};\mathcal{L}_0^\star)$. We define the \emph{canonical intuitionistic Mkrtychev model over $\mathfrak{N}$} as the structure $\mathfrak{M}^{c,M}_\mathfrak{N}=\langle\mathfrak{F},\mathcal{E}^c,\Vdash^c\rangle$ by setting:
\begin{enumerate}
\item $x\Vdash^c p$ iff $x\Vdash^*p$;
\item $\mathcal{E}^c_t(x):=\{\phi\in\mathcal{L}_J\mid x\Vdash^*\phi_t\}$.
\end{enumerate}
\end{definition}
\begin{lemma}\label{lem:genmkrtmodtruthlem}
Let $\mathfrak{F}=\langle F,\leq\rangle$ be a Kripke frame, let $\mathfrak{N}\in\mathsf{Mod}(\mathfrak{F};\mathcal{L}_0^\star)$ and let $\mathfrak{M}^{c,M}_\mathfrak{N}=\langle\mathfrak{F},\mathcal{E}^c,\Vdash^c\rangle$ as above. Then, for all $\phi\in\mathcal{L}_J$ and all $x\in F$:
\[
(\mathfrak{M}^{c,M}_\mathfrak{N},x)\models\phi\text{ iff }(\mathfrak{N},x)\models\phi^\star.
\]
\end{lemma}
\begin{proof}
We prove the statement by induction on $\phi$. The claim is clear for $p\in Var$ by definition. Suppose the claim is true for $\phi,\psi$. Then, we have at first
\begin{align*}
(\mathfrak{M}^{c,M}_\mathfrak{N},x)\models\phi\land\psi&\text{ iff }(\mathfrak{M}^{c,M}_\mathfrak{N},x)\models\phi\text{ and }(\mathfrak{M}^{c,M}_\mathfrak{N},x)\models\psi\\
                                                    &\text{ iff }(\mathfrak{N},x)\models\phi^\star\text{ and }(\mathfrak{N},x)\models\psi^\star\\
                                                    &\text{ iff }(\mathfrak{N},x)\models(\phi\land\psi)^\star
\end{align*}
and similarly for $\lor$. For $\rightarrow$, we have
\begin{align*}
(\mathfrak{M}^{c,M}_\mathfrak{N},x)\models\phi\rightarrow\psi&\text{ iff }\forall y\geq x:(\mathfrak{M}^{c,M}_\mathfrak{N},y)\models\phi\text{ implies }(\mathfrak{M}^{c,M}_\mathfrak{N},y)\models\psi\\
                                                          &\text{ iff }\forall y\geq x:(\mathfrak{N},y)\models\phi^\star\text{ implies }(\mathfrak{N},y)\models\psi^\star\\
                                                          &\text{ iff }(\mathfrak{N},x)\models\phi^\star\rightarrow\psi^\star.
\end{align*}
Lastly, we have
\begin{align*}
(\mathfrak{M}^{c,M}_\mathfrak{N},x)\models t:\phi&\text{ iff }\phi\in\mathcal{E}^c_t(x)\\
                                              &\text{ iff }x\Vdash^*\phi_t\\
                                              &\text{ iff }(\mathfrak{N},x)\models\phi_t.
\end{align*}
\end{proof}
\begin{lemma}\label{lem:genmkrtmodwelldef}
Let $\mathfrak{F}=\langle F,\leq\rangle$ be a Kripke frame and let $\mathfrak{N}\in\mathsf{Mod}(\mathfrak{F};\mathcal{L}_0^\star)$ such that additionally $\mathfrak{N}\models (Th_{\mathbf{LJL}_{CS}})^\star$. Then $\mathfrak{M}^{c,M}_\mathfrak{N}$ is a well-defined intuitionistic Mkrtychev model. Further:
\begin{enumerate}[(a)]
\item if $(F)$ is an axiom scheme of $\mathbf{LJL}_0$, then $\mathfrak{M}^{c,M}_\mathfrak{N}$ is factive;
\item if $(I)$ is an axiom scheme of $\mathbf{LJL}_0$, then $\mathfrak{M}^{c,M}_\mathfrak{N}$ is introspective.
\end{enumerate}
\end{lemma}
\begin{proof}
We first show properties (1) and (2) of Definition \ref{def:genmkrtmod}. For this, let $x,y\in F$ with $x\leq y$. For (1), we have
\[
x\Vdash^c p\Rightarrow x\Vdash^* p\Rightarrow y\Vdash^*p\Rightarrow y\Vdash^c p
\]
and for (2), we have
\[
\phi\in\mathcal{E}_t(x)\Rightarrow x\Vdash^*\phi_t\Rightarrow y\Vdash^*\phi_t\Rightarrow \phi\in\mathcal{E}_t(y).
\]
Both follow from Lemma \ref{lem:propmono} applied to $\mathfrak{N}$.\\

For properties (i) and (ii), let $x\in F$ and $t,s\in Jt$. Then, at first for (i), let $\phi\in\mathcal{E}^c_t(x)\sqsupset\mathcal{E}^c_s(x)$, that is by definition $\exists\psi\in\mathcal{L}_J$:
\[
\psi\rightarrow\phi\in\mathcal{E}^c_t(x)\text{ and }\psi\in\mathcal{E}^c_s(x).
\]
Untangling the definition of $\mathcal{E}^c$, this is
\[
x\Vdash^*(\psi\rightarrow\phi)_t\text{ and }x\Vdash^*\psi_s.
\]
As we have $\mathfrak{N}\models(Th_{\mathbf{LJL}_{CS}})^\star$, it follows from the $\star$-translation of the axiom scheme $(J)$ that $\forall y\geq x$:
\[
y\Vdash^*(\psi\rightarrow\phi)_t\text{ and }y\Vdash^*\psi_s\text{ implies }y\Vdash^*\phi_{[t\cdot s]}
\]
and thus especially, as $x\geq x$, we have $x\Vdash^*\phi_{[t\cdot s]}$ and consequently $\phi\in\mathcal{E}^c_{[t\cdot s]}(x)$.

For (ii), let $\phi\in\mathcal{E}^c_t(x)\cup\mathcal{E}^c_s(x)$. Then, we have $x\Vdash^*\phi_t$ or $x\Vdash^*\phi_s$. By the $\star$-translation of the axiom scheme $(+)$ and $\mathfrak{N}\models(Th_{\mathbf{LJL}_{CS}})^\star$, we have in either case as before $x\Vdash^*\phi_{[t+s]}$.\\

Now, for (a), if $(F)$ is an axiom scheme of $\mathbf{LJL}_0$, then we have by $\mathfrak{N}\models(Th_{\mathbf{LJL}_{CS}})^\star$ again that
\[
x\Vdash^*\phi_t\text{ implies }(\mathfrak{N},x)\models\phi^\star.
\]
By the definition of $\mathcal{E}^c$ and Lemma \ref{lem:genmkrtmodtruthlem}, we obtain
\[
\phi\in\mathcal{E}^c_t(x)\text{ implies }(\mathfrak{M}^{c,M}_\mathfrak{N},x)\models\phi
\]
and thus $\mathfrak{M}^{c,M}_\mathfrak{N}$ is factive.\\

For (b), if $(I)$ is an axiom scheme of $\mathbf{LJL}_0$, then we have 
\[
x\Vdash^*\phi_t\text{ implies }x\Vdash^*(t:\phi)_{!t}
\]
and by definition that is
\[
\phi\in\mathcal{E}^c_t(x)\text{ implies }t:\phi\in\mathcal{E}^c_{!t}(x)
\]
which is $t:\mathcal{E}^c_t(x)\subseteq\mathcal{E}^c_{!t}(x)$ and thus $\mathfrak{M}^{c,M}_\mathfrak{N}$ is introspective.
\end{proof}
\begin{theorem}\label{thm:genmkrtmodcomp}
Let $\mathbf{L}$ be an intermediate logic, $\mathbf{LJL}_0\in\{\mathbf{LJ}_0,\mathbf{LJT}_0,\mathbf{LJ4}_0,\mathbf{LJT4}_0\}$ and let $CS$ be a constant specification for $\mathbf{LJL}_0$. Let $\mathsf{C}\in\mathsf{KFr}(\mathbf{L})\cap\mathsf{KFr}^g(\mathbf{L})$ and let $\mathsf{CKMJL}$ be the class of intuitionistic Mkrtychev models corresponding to $\mathbf{LJL}_0$ and $\mathsf{C}$. For any $\Gamma\cup\{\phi\}\subseteq\mathcal{L}_J$, we have:
\[
\Gamma\vdash_{\mathbf{LJL}_{CS}}\phi\text{ iff }\Gamma\models_{\mathsf{CKMJL}_{CS}}\phi.
\]
\end{theorem}
\begin{proof}
The direction from left to right follows from Lemma \ref{lem:genmkrtmodsoundness}. For the converse, suppose $\Gamma\not\vdash_{\mathbf{LJL}_{CS}}\phi$. By Lemma \ref{lem:startrans}, we have 
\[
\Gamma^\star\cup (Th_{\mathbf{LJL}_{CS}})^\star\not\vdash_{\mathbf{L}^\star}\phi^\star
\]
and by assumption on the global strong completeness of $\mathbf{L}$ w.r.t. $\mathsf{C}$, there is a $\mathfrak{N}=\langle\mathfrak{F},\Vdash^*\rangle\in\mathsf{Mod}(\mathsf{C};\mathcal{L}_0^\star)$ such that
\[
\mathfrak{N}\models\Gamma^\star\cup (Th_{\mathbf{LJL}_{CS}})^\star\text{ but }\mathfrak{N}\not\models\phi^\star.
\]
By Lemma \ref{lem:genmkrtmodwelldef}, we have $\mathfrak{M}^{c,M}_\mathfrak{N}\in\mathsf{CKMJL}$ for the corresponding canonical intuitionistic Mkrtychev model. By Lemma \ref{lem:genmkrtmodtruthlem}, this yields
\[
\mathfrak{M}^{c,M}_\mathfrak{N}\models CS
\]
and thus $\mathfrak{M}^{c,M}_\mathfrak{N}\in\mathsf{CKMJL}_{CS}$ as well as
\[
\mathfrak{M}^{c,M}_\mathfrak{N}\models\Gamma\text{ but }\mathfrak{M}^{c,M}_\mathfrak{N}\not\models\phi.
\]
Hence, we have $\Gamma\not\models_{\mathsf{CKMJL}_{CS}}\phi$.
\end{proof}
\subsection{Completeness w.r.t. intuitionistic Fitting models}
\begin{definition}
Let $\mathfrak{F}=\langle F,\leq\rangle$ be a Kripke frame and let $\mathfrak{N}\in\mathsf{Mod}(\mathfrak{F};\mathcal{L}_0^\star)$. We define the \emph{canonical intuitionistic Fitting model over $\mathfrak{N}$} as the structure $\mathfrak{M}^{c,F}_\mathfrak{N}=\langle\mathfrak{F},\mathcal{R}^c,\mathcal{E}^c,\Vdash^c\rangle$ by setting:
\begin{enumerate}
\item $x\Vdash^c p$ iff $x\Vdash^*p$;
\item $\mathcal{E}^c_t(x):=\{\phi\in\mathcal{L}_J\mid x\Vdash^*\phi_t\}$;
\item $(x,y)\in\mathcal{R}^c$ iff $\forall t\in Jt\forall\phi\in\mathcal{L}_J\left(x\Vdash^*\phi_t\Rightarrow (\mathfrak{N},y)\models\phi^\star\right)$.
\end{enumerate}
\end{definition}
\begin{lemma}\label{lem:genfitmodtruthlem}
Let $\mathfrak{F}=\langle F,\leq\rangle$ be a Kripke frame, let $\mathfrak{N}\in\mathsf{Mod}(\mathfrak{F};\mathcal{L}_0^\star)$ and define $\mathfrak{M}^{c,F}_\mathfrak{N}$ as above. For any $\phi\in\mathcal{L}_J$ and all $x\in F$:
\[
(\mathfrak{M}^{c,F}_\mathfrak{N},x)\models\phi\text{ iff }(\mathfrak{N},x)\models\phi^\star.
\]
\end{lemma}
\begin{proof}
The claim is again proved by induction on the structure of the formula. We only consider the modal case. Suppose the claim holds for all $x\in F$ and some $\phi\in\mathcal{L}_J$. 

At first, suppose $(\mathfrak{N},x)\models\phi_t$, i.e. $x\Vdash^*\phi_t$. Then, naturally $\phi\in\mathcal{E}^c_t(x)$ by definition. Further, let $y\in\mathcal{R}^c[x]$. Then, as $x\Vdash^*\phi_t$, we have $(\mathfrak{N},y)\models\phi^\star$ by definition and thus $(\mathfrak{M}^{c,F}_\mathfrak{N},x)\models\phi$ by induction hypothesis. Hence, we get 
\[
\phi\in\mathcal{E}^c_t(x)\text{ and }\forall y\in\mathcal{R}^c[x]\;(\mathfrak{M}^{c,F}_\mathfrak{N},x)\models\phi
\]
and consequently $(\mathfrak{M}^{c,F}_\mathfrak{N},x)\models t:\phi$. 

Conversely, suppose $(\mathfrak{N},x)\not\models\phi_t$, that is $x\not\Vdash^*\phi_t$. Then $\phi\not\in\mathcal{E}^c_t(x)$ by definition and thus
\[
(\mathfrak{M}^{c,F}_\mathfrak{N},x)\not\models t:\phi
\]
immediately by definition.
\end{proof}
\begin{lemma}\label{lem:genfitmodwelldef}
Let $\mathfrak{F}=\langle F,\leq\rangle$ be a Kripke frame and let $\mathfrak{N}\in\mathsf{Mod}(\mathfrak{F};\mathcal{L}_0^\star)$ such that additionally $\mathfrak{N}\models (Th_{\mathbf{LJL}_{CS}})^\star$. Then, $\mathfrak{M}^{c,F}_\mathfrak{N}$ is a well-defined intuitionistic Fitting model. Further:
\begin{enumerate}[(a)]
\item if $(F)$ is an axiom scheme of $\mathbf{LJL}_0$, then $\mathfrak{M}^{c,F}_\mathfrak{N}$ is reflexive;
\item if $(I)$ is an axiom scheme of $\mathbf{LJL}_0$, then $\mathfrak{M}^{c,F}_\mathfrak{N}$ is introspective.
\end{enumerate}
\end{lemma}
\begin{proof}
For properties (1) - (3) of Definition \ref{def:genfittingmod}, let $x,y\in F$ with $x\leq y$. For (1) and (2), we have as before
\[
x\Vdash^c p\Rightarrow x\Vdash^* p\Rightarrow y\Vdash^* p\Rightarrow y\Vdash^c p
\]
and
\[
\phi\in\mathcal{E}_t(x)\Rightarrow x\Vdash^*\phi_t\Rightarrow y\Vdash^*\phi_t\Rightarrow \phi\in\mathcal{E}_t(y)
\]
by Lemma \ref{lem:propmono} for $\mathfrak{N}$. For (3), let $z\in\mathcal{R}[y]$, that is we have
\[
\forall t\in Jt\forall\phi\in\mathcal{L}_J\;(y\Vdash^*\phi_t\Rightarrow (\mathfrak{N},z)\models\phi^\star).
\]
Then, for any $t\in Jt$ and any $\phi\in\mathcal{L}_J$ we have, if $x\Vdash^*\phi_t$ that $y\Vdash^*\phi_t$ by Lemma \ref{lem:propmono} and thus by the above $(\mathfrak{N},z)\models\phi^\star$. Hence, $z\in\mathcal{R}[x]$ and so $\mathcal{R}[y]\subseteq\mathcal{R}[x]$.\\

We obtain properties (1) and (2) of Definition \ref{def:genfittingmod} in the same way as in the proof of Lemma \ref{lem:genmkrtmodwelldef}. For (3), let $z\in\mathcal{R}^c[y]$ for $x\leq y$, that is
\[
\forall t\in Jt\forall\phi\in\mathcal{L}_J\;(y\Vdash^*\phi_t\Rightarrow (\mathfrak{N},z)\models\phi^\star).
\]
If $x\Vdash^*\phi_t$, then $y\Vdash^*\phi_t$ by Lemma \ref{lem:propmono} and thus by the above $(\mathfrak{N},z)\models\phi^\star$. Hence, $z\in\mathcal{R}^c[y]$.\\

For property (i), let $\phi\in\mathcal{E}^c_t(x)\sqsupset\mathcal{E}^c_s(x)$, i.e.
\[
\exists\psi\in\mathcal{L}_J\left(\psi\rightarrow\phi\in\mathcal{E}^c_t(x)\text{ and }\psi\in\mathcal{E}^c_s(x)\right).
\]
Then, by definition, we have 
\[
x\Vdash^*(\psi\rightarrow\phi)_t\text{ and }x\Vdash^*\psi_s
\]
and thus, as $\mathfrak{N}\models(Th_{\mathbf{LJL}_{CS}})^\star$, we get
\[
x\Vdash^*\phi_{[t\cdot s]},
\]
that is $\phi\in\mathcal{E}^c_{t\cdot s}(x)$. 

For property (ii), note that $x\Vdash^*\phi_t$ implies $x\Vdash^*\phi_{[t+s]}$ again by $\mathfrak{N}\models(Th_{\mathbf{LJL}_{CS}})^\star$, hence $\phi\in\mathcal{E}^c_t(x)$ implies $\phi\in\mathcal{E}^c_{t+s}(x)$ and similarly for $\phi\in\mathcal{E}^c_s(x)$.\\

Suppose that $(F)$ is an axiom scheme of $\mathbf{LJL}_0$. Then, we have
\[
\forall t\in Jt\forall\phi\in\mathcal{L}_J\left(x\Vdash^*\phi_t\Rightarrow (\mathfrak{N},x)\models\phi^\star\right)
\]
as $\mathfrak{N}\models(Th_{\mathbf{LJL}_{CS}})^\star$ and this is exactly $(x,x)\in\mathcal{R}^c$.\\

Suppose that $(I)$ is an axiom scheme of $\mathbf{LJL}_0$. As in the case of intuitionistic Mkrtychev models, one shows
\[
t:\mathcal{E}^c_t(x)\subseteq\mathcal{E}^c_{!t}(x).
\]
For the transitivity of $\mathcal{R}^c$, let $(x,y),(y,z)\in\mathcal{R}^c$, that is, we have
\[
\forall t\in Jt\forall\phi\in\mathcal{L}_J\left(x\Vdash^*\phi_t\Rightarrow(\mathfrak{N},y)\models\phi^\star\right)
\]
as well as
\[
\forall t\in Jt\forall\phi\in\mathcal{L}_J\left(y\Vdash^*\phi_t\Rightarrow(\mathfrak{N},z)\models\phi^\star\right).
\]
As $(I)$ is an axiom scheme and $\mathfrak{N}\models(Th_{\mathbf{LJL}_{CS}})^\star$, we have
\[
w\Vdash^*\phi_t\Rightarrow w\Vdash^* (t:\phi)_{!t}
\]
for any $w\in F$. Thus, in particular, we have
\[
x\Vdash^*\phi_t\Rightarrow x\Vdash^*(t:\phi)_{!t}\Rightarrow y\Vdash^*\phi_t\Rightarrow (\mathfrak{N},z)\models\phi^\star
\]
using Lemma \ref{lem:genfitmodtruthlem}. By definition, this yields $(x,z)\in\mathcal{R}^c$.

For the monotonicity, let $y\in\mathcal{R}^c[x]$ and let $\phi\in\mathcal{E}_t(x)$. The former gives 
\[
\forall t\in Jt\forall\phi\in\mathcal{L}_J\;(x\Vdash^*\phi_t\Rightarrow (\mathfrak{N},y)\models\phi^\star)
\]
and the latter gives $x\Vdash^*\phi_t$. As $\mathfrak{N}\models(Th_{\mathbf{LJL}_{CS}})^\star$, we have especially $x\Vdash^*(t:\phi_{!t})$. By the above, this gives us $(\mathfrak{N},y)\models\phi_t$, that is $y\Vdash^*\phi_t$ and thus $\phi\in\mathcal{E}^c_t(y)$. Hence, $\mathfrak{M}^{c,F}_\mathfrak{N}$ is monotone and it follows that $\mathfrak{M}^{c,F}_\mathfrak{N}$ is introspective.
\end{proof}
\begin{theorem}\label{thm:genfittingmodcomp}
Let $\mathbf{L}$ be an intermediate logic, $\mathbf{LJL}_0\in\{\mathbf{LJ}_0,\mathbf{LJT}_0,\mathbf{LJ4}_0,\mathbf{LJT4}_0\}$ and let $CS$ be a constant specification for $\mathbf{LJL}_0$. Let $\mathsf{C}\in\mathsf{KFr}(\mathbf{L})\cap\mathsf{KFr}^g(\mathbf{L})$ and let $\mathsf{CKFJL}$ be the class of intuitionistic Fitting models corresponding to $\mathbf{LJL}_0$ and $\mathsf{C}$. For any $\Gamma\cup\{\phi\}\subseteq\mathcal{L}_J$, we have:
\[
\Gamma\vdash_{\mathbf{LJL}_{CS}}\phi\text{ iff }\Gamma\models_{\mathsf{CKFJL}_{CS}}\phi.
\]
\end{theorem}
\subsection{Completeness w.r.t. intuitionistic subset models}
\begin{definition}\label{def:cansubmod}
Let $\mathfrak{F}=\langle F_0,\leq\rangle$ be a Kripke frame and let $\mathfrak{N}\in\mathsf{Mod}(\mathfrak{F};\mathcal{L}_0^\star)$. We define the \emph{canonical intuitionistic subset model over $\mathfrak{N}$} as the structure $\mathfrak{M}^{c,S}_\mathfrak{N}=\langle\mathfrak{F},F^c,\mathcal{E}^c,\Vdash^c\rangle$ by setting:
\begin{enumerate}
\item $F^c=F_0\cup\bigcup_{x\in F_0}\{x_t\mid t\in Jt\}$;
\item $(x,y)\in\mathcal{E}^c_t$ iff $\forall\phi\in\mathcal{L}_J\left(x\Vdash^*\phi_t\Rightarrow y\Vdash^c\phi\right)$ for all $x,y\in F^c$;
\item for $x\in F_0$ and $t\in Jt$:
\begin{enumerate}
\item $x\Vdash^c\phi$ iff $(\mathfrak{N},x)\models\phi^\star$;
\item $x_t\Vdash^c\phi$ iff $(\mathfrak{N},x)\models\phi_t$.
\end{enumerate}
\end{enumerate}
\end{definition}
\begin{lemma}\label{lem:gensubsetmodtruthlem}
Let $\mathfrak{F}=\langle F_0,\leq\rangle$ be a Kripke frame, let $\mathfrak{N}\in\mathsf{Mod}(\mathfrak{F};\mathcal{L}_0^\star)$ and define $\mathfrak{M}^{c,S}_\mathfrak{N}$ as above. For any $\phi\in\mathcal{L}_J$ and any $x\in F_0$:
\[
(\mathfrak{M}^{c,S}_\mathfrak{N},x)\models\phi\text{ iff }(\mathfrak{N},x)\models\phi^\star.
\]
\end{lemma}
\begin{proof}
This is clear by definition as we have $(\mathfrak{M}^{c,S}_\mathfrak{N},x)\models\phi$ iff $x\Vdash^c\phi$ iff $(\mathfrak{N},x)\models\phi^\star$, given a $x\in F_0$.
\end{proof}
The simplicity of the above lemma is in contrast to the truth lemmas for the previous canonical models over Kripke frames. In the context of intuitionistic subset models, the relation $\Vdash$ completely encodes the truth values of formulae to be able to cope with "irregular" worlds. This comes with the expense of conditions of well-definedness for $\Vdash$ and thus, the previous complexity of showing an equivalence like the one of the above lemma is shifted into the following result.
\begin{lemma}\label{lem:gensubsetmodwelldef}
Let $\mathfrak{F}=\langle F_0,\leq\rangle$ be a Kripke frame and let $\mathfrak{N}\in\mathsf{Mod}(\mathfrak{F};\mathcal{L}_0^\star)$ such that additionally $\mathfrak{N}\models (Th_{\mathbf{LJL}_{CS}})^\star$. Then, $\mathfrak{M}^{c,S}_\mathfrak{N}$ is a well-defined intuitionistic subset model. Further:
\begin{enumerate}[(I)]
\item if $(F)$ is an axiom scheme of $\mathbf{LJL}_0$, then $\mathfrak{M}^{c,S}_\mathfrak{N}$ is reflexive;
\item if $(I)$ is an axiom scheme of $\mathbf{LJL}_0$, then $\mathfrak{M}^{c,S}_\mathfrak{N}$ is introspective.
\end{enumerate}
\end{lemma}
\begin{proof}
We begin with properties (i) - (v) from Definition \ref{def:gensubsetmod}. For this, let $x\in F_0$. The properties (i) - (iii) are immediate by using the respective properties of $\Vdash^*$ and the fact that $\star$ commutes with $\bot,\land,\lor$.

For (iv), note that we have
\begin{align*}
x\Vdash^c\phi\rightarrow\psi&\text{ iff }(\mathfrak{N},x)\models\phi^\star\rightarrow\psi^\star\\
                            &\text{ iff }\forall y\geq x\left((\mathfrak{N},y)\not\models\phi^\star\text{ or }(\mathfrak{N},y)\models\psi^\star\right)\\
                            &\text{ iff }\forall y\geq x\left(y\not\Vdash^c\phi\text{ or }y\Vdash^c\psi\right)
\end{align*}
where it is instrumental that $\leq$ is a relation on $F_0$ only.

For (v), we have for one by definition that
\[
\forall y\in\mathcal{E}^c_t[x]\forall\phi\in\mathcal{L}_J\left((\mathfrak{N},x)\models\phi_t\Rightarrow (\mathfrak{N},y)\models\phi^\star\right),
\]
that is we have
\begin{align*}
x\Vdash^ct:\phi&\Rightarrow x\Vdash^*\phi_t\\
               &\Rightarrow\forall y\in\mathcal{E}^c_t[x]\left((\mathfrak{N},y)\models\phi^\star\right)\\
               &\Leftrightarrow\forall y\in\mathcal{E}^c_t[x]\left(y\Vdash^c\phi\right).
\end{align*}
For another, we have $x_t\in\mathcal{E}_t[x]$ as we know $x_t\Vdash^c\phi$ iff $x\Vdash^*\phi_t$ by definition. Thus, if $x\not\Vdash^ct:\phi$, then $x\not\Vdash^*\phi_t$ and $x_t\not\Vdash^c\phi$. Therefore
\[
x\not\Vdash^ct:\phi\Rightarrow \exists y\in\mathcal{E}_t[x] \left(y\not\Vdash^c\phi\right).
\]
Concluding, we have $x\Vdash^ct:\phi$ iff $\forall y\in\mathcal{E}^c_t[x]\left(y\Vdash^c\phi\right)$.\\

Regarding properties (1) and (2) of Definition \ref{def:gensubsetmod}, let $x\leq y$ for $x,y\in F_0$. Property (1) follows as in the proof of Lemma \ref{lem:genfitmodwelldef} by (3).(a) of Definition \ref{def:cansubmod}. For property (2), let $z\in\mathcal{E}^c_t[y]$. Thus, we have
\[
\forall\phi\in\mathcal{L}_J\;(y\Vdash^ct:\phi\Rightarrow z\Vdash^c\phi)
\]
and if $x\Vdash^ct:\phi$, then as $x,y\in F_0$, we get $x\Vdash^*\phi_t$ and thus $y\Vdash^*\phi_t$ which is $y\Vdash^ct:\phi$. By the above, we have $z\Vdash^c\phi$ and hence $z\in\mathcal{E}^c_t[x]$.\\

Now, on to properties (a), (b) of Definition \ref{def:gensubsetmod}. For (a), let $y\in\mathcal{E}^c_{[t+s]}[x]$, that is we have
\[
\forall\phi\in\mathcal{L}_J (x\Vdash^*\phi_{[t+s]}\Rightarrow (\mathfrak{N},y)\models\phi^\star)
\]
by definition. Now, by assumption as $\mathfrak{N}\models(Th_{\mathbf{LJL}_{CS}})^\star$ we have $x\Vdash^*\phi_t$ implies $x\Vdash^*\phi_{[t+s]}$ and $x\Vdash^*\phi_s$ implies $x\Vdash^*\phi_{[t+s]}$. Therefore, we obtain
\[
\forall\phi\in\mathcal{L}_J (x\Vdash^*\phi_t\Rightarrow x\Vdash^*\phi_{[t+s]}\Rightarrow (\mathfrak{N},y)\models\phi^\star)
\]
and
\[
\forall\phi\in\mathcal{L}_J (x\Vdash^*\phi_s\Rightarrow x\Vdash^*\phi_{[t+s]}\Rightarrow (\mathfrak{N},y)\models\phi^\star)
\]
which is $y\in\mathcal{E}^c_t[x]\cap\mathcal{E}^c_s[x]$.

For (b), let $y\in\mathcal{E}^c_{[t\cdot s]}[x]$, that is
\[
\forall\phi\in\mathcal{L}_J (x\Vdash^*\phi_{[t\cdot s]}\Rightarrow (\mathfrak{N},y)\models\phi^\star).\tag{$\dagger$}
\]
Let $\phi\in (\mathfrak{M}^{c,S}_\mathfrak{N})^x_{t,s}$, that is there is a $\psi\in\mathcal{L}_J$ such that
\[
\forall z\in F^c(z\in\mathcal{E}^c_t[x]\Rightarrow z\Vdash^c\psi\rightarrow\phi\text{ and } z\in\mathcal{E}^c_s[x]\Rightarrow z\Vdash^c\psi).
\]
By property (v), we have that $x\Vdash^c t:(\psi\rightarrow\phi)$ and $x\Vdash^c s:\psi$, i.e. by definition as $x\in F_0$:
\[
(\mathfrak{N},x)\models(\psi\rightarrow\phi)_t\text{ and }(\mathfrak{N},x)\models\psi_s
\]
and by $\mathfrak{N}\models(Th_{\mathbf{LJL}_{CS}})^\star$ and axiom ($J$), we get
\[
(\mathfrak{N},x)\models\phi_{[t\cdot s]}.
\]
Thus, by $(\dagger)$, we have $(\mathfrak{N},y)\models\phi^\star$ and by definition this gives $y\Vdash^c\phi$.\\

Assume second to last that $(F)$ is an axiom scheme of $\mathbf{LJL}_0$. Then, we have
\[
x\Vdash^*\phi_t\Rightarrow (\mathfrak{N},x)\models\phi^\star
\]
for all $x\in F_0$ and all $\phi\in\mathcal{L}_J$, $t\in Jt$ as $\mathfrak{N}\models (Th_{\mathbf{LJL}_{CS}})^\star$ and thus, by definition we have $x\in\mathcal{E}^c_t[x]$ for all $t\in Jt$.\\

Assume last that $(I)$ is an axiom scheme of $\mathbf{LJL}_0$. Let $y\in\mathcal{E}_{!t}[x]$, that is 
\[
\forall\phi\in\mathcal{L}_J (x\Vdash^*\phi_{!t}\Rightarrow (\mathfrak{N},y)\models\phi^\star).\tag{$\ddagger$}
\]
Let $\phi\in\mathcal{L}_J$ and assume $x\Vdash^c t:\phi$, that is $x\Vdash^*\phi_t$. Then, as $\mathfrak{N}\models (Th_{\mathbf{LJL}_{CS}})^\star$, we have $x\Vdash^*(t:\phi)_{!t}$. By $(\ddagger)$ we get $(\mathfrak{N},y)\models (t:\phi)^\star$, that is $(\mathfrak{N},y)\models\phi_t$ and thus by definition $y\Vdash^c t:\phi$.
\end{proof}
\begin{theorem}\label{thm:gensubsetmodcomp}
Let $\mathbf{L}$ be an intermediate logic, $\mathbf{LJL}_0\in\{\mathbf{LJ}_0,\mathbf{LJT}_0,\mathbf{LJ4}_0,\mathbf{LJT4}_0\}$ and let $CS$ be a constant specification for $\mathbf{LJL}_0$. Let $\mathsf{C}\in\mathsf{KFr}(\mathbf{L})\cap\mathsf{KFr}^g(\mathbf{L})$ and let $\mathsf{CKSJL}$ be the class of intuitionistic subset models corresponding to $\mathbf{LJL}_0$ and $\mathsf{C}$. Then, for any $\Gamma\cup\{\phi\}\subseteq\mathcal{L}_J$, we have:
\[
\Gamma\vdash_{\mathbf{LJL}_{CS}}\phi\text{ iff }\Gamma\models_{\mathsf{CKSJL}_{CS}}\phi.
\]
\end{theorem}
\section{Intermediate Modal Logics}
In this section, we give an overview of the (semantic) theory of intermediate modal logics as commonly defined. These will prove to be the natural choice for corresponding modal logics with respect to our intermediate justification logics.

To define intermediate modal logics, we consider a usual modal language with a single modality $\Box$ given by
\[
\mathcal{L}_\Box:\phi::=\bot\mid p\mid (\phi\land\phi)\mid (\phi\lor\phi)\mid (\phi\rightarrow\phi)\mid \Box\phi
\]
where again $p\in Var$. Propositional substitutions naturally generalize to $\sigma:Var\to\mathcal{L}_\Box$ and we still write $\sigma(\phi)$ for the image of $\phi$ under the natural extension of $\sigma$ to $\mathcal{L}_\Box$ by commuting with $\land,\lor,\rightarrow,\bot$ and $\Box$. 
\subsection{Proof systems}
\begin{definition}
An \emph{intermediate modal logic} is a set $\mathbf{IML}\subsetneq\mathcal{L}_\Box$ such that
\begin{enumerate}
\item $\mathbf{IPC}\cup (K)\subseteq\mathbf{IML}$ where $(K)$ is given by 
\[
\Box(\phi\rightarrow\psi)\rightarrow (\Box\phi\rightarrow\Box\psi),
\]
\item $\mathbf{IML}$ is closed under substitution in $\mathcal{L}_\Box$,
\item $\mathbf{IML}$ is closed under modus ponens,
\item $\mathbf{IML}$ is closed under \emph{necessitation}, that is $\phi\in\mathbf{IML}$ implies $\Box\phi\in\mathbf{IML}$.
\end{enumerate}
\end{definition} 
We denote the smallest such logic by $\mathbf{IPCK}$. To give more direct definitions of axiomatic extensions, we introduce the following notation. Given sets $\Gamma,\Delta\subseteq\mathcal{L}_\Box$, we write $\Gamma\oplus\Delta$ for the \emph{normal closure of} $\Gamma\cup\Delta$, that is for the smallest set $\Phi\subseteq\mathcal{L}_\Box$ with:
\begin{enumerate}
\item $\Gamma\cup\Delta\subseteq\Phi$;
\item $\Phi$ is closed under modus ponens;
\item $\Phi$ is closed under substitution in $\mathcal{L}_\Box$;
\item $\Phi$ is closed under necessitation.
\end{enumerate}
In particular, in the following, we consider the two axiom schemes
\begin{enumerate}
\item [($T$)] $\Box\phi\rightarrow\phi$,
\item [($4$)] $\Box\phi\rightarrow\Box\Box\phi$,
\end{enumerate}
and, given an intermediate logic $\mathbf{L}$, we write:
\begin{enumerate}
\item $\mathbf{LK}:=\mathbf{IPCK}\oplus\mathbf{L}$;
\item $\mathbf{LT}:=\mathbf{LK}\oplus (T)$;
\item $\mathbf{LK4}:=\mathbf{LK}\oplus (4)$;
\item $\mathbf{LS4}:=\mathbf{LK}\oplus (T)\oplus (4)$.
\end{enumerate}
As before with the intermediate justification logics, given a set $\Gamma\subseteq\mathcal{L}_\Box$ and an intermediate modal logic $\mathbf{IML}$, we write
\[
\Gamma\vdash_{\mathbf{IML}}\phi\text{ iff }\exists\gamma_1,\dots,\gamma_n\left(\bigwedge_{i=1}^n\gamma_i\rightarrow\phi\in\mathbf{IML}\right).
\]
Naturally, all these intermediate modal logics enjoy the deduction theorem.

Given an intermediate logic $\mathbf{L}$, the corresponding modal logics $\mathbf{LK}$, $\mathbf{LT}$, $\mathbf{LK4}$, $\mathbf{LS4}$ are natural correspondents  for the intermediate justification logics $\mathbf{LJ}$, $\mathbf{LJT}$, $\mathbf{LJ4}$, $\mathbf{LJT4}$, respectively, and we will later confirm this in many cases by a general realization theorem.
\subsection{Semantics and Completeness}
We will need some semantical notions regarding intermediate modal logics for the to-follow model theoretical considerations regarding realizations. For this, we introduce so called intuitionistic modal Kripke models which go back to Ono's work \cite{Ono1977}.
\begin{definition}
An \emph{intuitionistic modal Kripke frame} is a structure $\mathfrak{F}=\langle F,\leq,\mathcal{R}\rangle$ where $\langle F,\leq\rangle$ is a partial order and $\mathcal{R}\subseteq F\times F$ with
\[
x\leq y\Rightarrow\mathcal{R}[y]\subseteq\mathcal{R}[x].
\]
An \emph{intuitionistic modal Kripke model} over $\mathfrak{F}$ is a structure $\mathfrak{M}=\langle\mathfrak{F},\Vdash\rangle$ where $\Vdash\subseteq F\times Var$ such that
\[
\left(x\Vdash p\text{ and }x\leq y\right)\Rightarrow y\Vdash p.
\]
\end{definition}
We write $\mathcal{D}(\mathfrak{M}):=\mathcal{D}(\mathfrak{F}):=F$. Further, given a class $\mathsf{C}$ of intuitionistic modal Kripke frames, we write $\mathsf{Mod}(\mathsf{C})$ for the class of all models over these frames.

Let $\mathfrak{M}=\langle F,\leq,\mathcal{R},\Vdash\rangle$ be an intuitionistic modal Kripke model and let $x\in\mathcal{D}(\mathfrak{M})$. We define the relation $\models$ recursively as follows:
\begin{itemize}
\item $(\mathfrak{M},x)\not\models\bot$;
\item $(\mathfrak{M},x)\models p$ iff $x\Vdash p$ for $p\in Var$;
\item $(\mathfrak{M},x)\models \phi\land\psi$ iff $(\mathfrak{M},x)\models\phi$ and $(\mathfrak{M},x)\models\psi$;
\item $(\mathfrak{M},x)\models \phi\lor\psi$ iff $(\mathfrak{M},x)\models\phi$ or $(\mathfrak{M},x)\models\psi$;
\item $(\mathfrak{M},x)\models \phi\rightarrow\psi$ iff $\forall y\geq x\left((\mathfrak{M},y)\models\phi\text{ implies }(\mathfrak{M},y)\models\psi\right)$;
\item $(\mathfrak{M},x)\models\Box\phi$ iff $\forall y\in\mathcal{R}[x]\left((\mathfrak{M},y)\models\phi\right)$.
\end{itemize}
We write $(\mathfrak{M},x)\models\Gamma$ if $(\mathfrak{M},x)\models\gamma$ for all $\gamma\in\Gamma$, given $\Gamma\subseteq\mathcal{L}_\Box$.
\begin{definition}
Let $\mathsf{C}$ be a class of intuitionistic modal Kripke models and $\Gamma\cup\{\phi\}\subseteq\mathcal{L}_\Box$. We write $\Gamma\models_\mathsf{C}\phi$ if
\[
\forall\mathfrak{M}\in\mathsf{C}\forall x\in\mathcal{D}(\mathfrak{M})\Big((\mathfrak{M},x)\models\Gamma\Rightarrow(\mathfrak{M},x)\models\phi\Big).
\]
If $\mathsf{C}$ is a class of intuitionistic modal Kripke frames, we write $\Gamma\models_\mathsf{C}\phi$ if $\Gamma\models_{\mathsf{Mod}(\mathsf{C})}\phi$.
\end{definition}
\begin{definition}
Let $\mathbf{IML}$ be an intermediate modal logic. $\mathbf{IML}$ is \emph{(strongly) Kripke complete w.r.t. a class $\mathsf{C}$} of intuitionistic modal frames, if
\[
\Gamma\vdash_{\mathbf{IML}}\phi\text{ iff }\Gamma\models_\mathsf{C}\phi
\]
for all $\Gamma\cup\{\phi\}\subseteq\mathcal{L}_\Box$.
\end{definition}
We write $\mathsf{C}\in\mathsf{MFr}(\mathbf{IML})$ in the above case. As there are Kripke incomplete intermediate logics (recall \cite{She1977}), there are Kripke incomplete intermediate modal logics. So again, the above notation is meant to be read as to include an existence statement.
\section{Substitutions and Realization}
As touched upon in the introduction, our approach for establishing a realization theorem between the intermediate justification logics and the previously introduced intermediate modal logics is based on Fitting's work \cite{Fit2016} about a semantic, but non-constructive, proof of the classical realization theorems. We adapt this proof using the previously introduced intuitionistic Fitting models.

The central points of this approach to the realization theorem are, for one, using a specific canonical model for the (intermediate) justification logic and translating this to an appropriate model for the (intermediate) modal logic to obtain the existence of so-called quasi-realizations (which will be precisely defined later on). The second central point is a constructive extraction of normal realizations from quasi-realizations using a recursion on the structure of the modal formula. 

For all these considerations, the following subsections introduce some technical tools.
\subsection{Annotated Modal Formulae}
To keep track over different $\Box$-symbols in modal formulae, which are potentially realized by different justification terms, we consider an annotated modal language given by 
\[
\mathcal{L}'_{\Box}:\phi'::=\bot\mid p\mid (\phi'\land\phi')\mid (\phi'\lor\phi')\mid (\phi'\rightarrow\phi')\mid \Box_n\phi'
\]
with $n\in\mathbb{N}$ and $p\in Var$. There is a natural projection from $\mathcal{L}'_{\Box}$ to $\mathcal{L}_{\Box}$ by dropping all $\Box$-annotations and we denote it by $(\cdot)^\bullet$. Followingly, we call an annotated formula $\chi\in\mathcal{L}'_\Box$ an annotation of a non-annotated formula $\phi$ if $\chi^\bullet=\phi$ and then also often write $\phi'$ for $\chi$.

We call an annotation $\phi'$ uniquely annotated if no $\Box$-index occurs more than once.
\subsection{Substitutions and Realizations}
We define the following operation $\real{\cdot}$ (following Fitting's \cite{Fit2016}), collecting all possible realizations of an annotated formula, given extra information on polarity.

More precisely, we define $\real{\cdot}:\{T,F\}\times\mathcal{L}'_{\Box}\to\mathcal{P}(\{T,F\}\times\mathcal{L}_{J})$ by recursion on $\mathcal{L}'_{\Box}$:
\begin{enumerate}
\item  For $\phi'\in Var\cup\{\bot\}$:\\
$\real{T,\phi'}:=\{(T,\phi')\}$; $\real{F,\phi'}=\{(F,\phi')\}$;\\
\item For $\circ\in\{\land,\lor\}$:\\
$\real{T,\phi'\circ\psi'}:=\{(T,\alpha\circ\beta)\mid (T,\alpha)\in\real{T,\phi'}, (T,\beta)\in\real{T,\psi'}\}$;\\
$\real{F,\phi'\circ\psi'}:=\{(F,\alpha\circ\beta)\mid (F,\alpha)\in\real{F,\phi'}, (F,\beta)\in\real{F,\psi'}\}$;\\
\item $\real{T,\phi'\rightarrow\psi'}:=\{(T,\alpha\rightarrow\beta)\mid (F,\alpha)\in\real{F,\phi'},(T,\beta)\in\real{T,\psi'}\}$;\\
$\real{F,\phi'\rightarrow\psi'}:=\{(F,\alpha\rightarrow\beta)\mid (T,\alpha)\in\real{T,\phi'}, (F,\beta)\in\real{F,\psi'}\}$;\\
\item $\real{T,\Box_n\phi'}:=\{(T,v_n:\alpha)\mid (T,\alpha)\in\real{T,\phi'}\}$;\\
$\real{F,\Box_n\phi'}:=\{(F,t:\alpha)\mid (F,\alpha)\in\real{F,\phi'}, t\in Jt\}$.
\end{enumerate}
Recalling a comment from \cite{Fit2016}, we remark that the symbols $T,F$ stem from the proof theoretic context of using tableau theorem proving. Here however, they are used just as syntactic bookkeeping of polarities.
\begin{definition}
A \emph{realization of a formula $\phi\in\mathcal{L}_\Box$} is any $\psi$ with $(F,\psi)\in\real{F,\phi'}$ where $\phi'$ is an unique annotation of $\phi'$.
\end{definition}
\begin{definition}
A \emph{justification substitution} is a function $\sigma:V\to Jt$. This function naturally extends to a function $\sigma:Jt\to Jt$ by 
\begin{enumerate}
\item $\sigma([t*s]):=\sigma(t)*\sigma(s)$ for $*\in\{\cdot,+\}$,
\item $\sigma(!t):=!\sigma(t)$.
\end{enumerate}
Further, $\sigma$ also extends to a function $\sigma:\mathcal{L}_{J}\to\mathcal{L}_{J}$ by commuting with all connectives $\land,\lor,\rightarrow,\bot$ and setting
\[
\sigma(t:\phi):=\sigma(t):\sigma(\phi).
\]
Note, that we write $\sigma$ throughout for all extensions.
\end{definition}
Given a formula $\phi$ or term $t$, we also write $\phi\sigma$ or $t\sigma$ for the image of it under a justification substitution $\sigma$.
\begin{lemma}\label{lem:sublemma}
Let $\mathbf{L}$ be an intermediate logic $\mathbf{LJL}\in\{\mathbf{LJ},\mathbf{LJT},\mathbf{LJ4},\mathbf{LJT4}\}$. If $\mathbf{LJL}\vdash\phi$, then $\mathbf{LJL}\vdash\phi\sigma$ for any justification substitution $\sigma$.
\end{lemma}
The proof is similar to the classical case (see e.g. \cite{KS2019}). For the upcoming investigations, we will need some further vocabulary regarding justification substitutions. At first, given a justification formula $\phi\in\mathcal{L}_J$ or justification term $t\in Jt$, we write $\mathrm{jvar}(\phi)$ or $\mathrm{jvar}(t)$ for the sets of all justification variables $x\in V$ occurring in $\phi$ or $t$, respectively.
\begin{definition}[Fitting \cite{Fit2016}]
Let $\sigma$ be a justification substitution and $\phi'$ be a uniquely annotated formula. We write $\mathrm{dom}(\sigma)=\{x\in V\mid x\sigma\neq x\}$. Further, we say:
\begin{enumerate}
\item $\sigma$ meets the \emph{no-new-variables} condition if $\mathrm{jvar}(x\sigma)\subseteq\{x\}$ for all $x\in V$;
\item $\sigma$ \emph{lives on} $\phi'$ if $x_k\in\mathrm{dom}(\sigma)$ implies that $\Box_k$ occurs in $\phi'$;
\item $\sigma$ \emph{lives away from} $\phi'$ if $x_k\in\mathrm{dom}(\sigma)$ implies that $\Box_k$ does not occur in $\phi'$.
\end{enumerate}
\end{definition}
There is a natural way of combining justification substitutions by iteratively applying them. For this, given two justification substitutions $\sigma$, $\sigma'$, we write $\sigma\sigma'$ for the substitution defined by
\[
x\mapsto\sigma'(\sigma(x))
\]
for all $x\in V$.
\begin{lemma}[Fitting \cite{Fit2016}]
Let $\phi'\in\mathcal{L}'_\Box$ be uniquely annotated, $\sigma_0$ a justification substitution that lives on $\phi'$ and $\sigma_1$ be a justification substitution that lives away from $\phi'$. Then:
\begin{enumerate}
\item $(T,\alpha)\in\real{T,\phi'}$ implies $(T,\alpha\sigma_1)\in\real{T,\phi'}$;\\
$(F,\alpha)\in\real{F,\phi'}$ implies $(F,\alpha\sigma_1)\in\real{F,\phi'}$.
\item If $\sigma_0,\sigma_1$ meet the no-new-variable condition, then $\sigma_0\sigma_1=\sigma_1\sigma_0$.
\end{enumerate} 
\end{lemma}
A proof can also be found in \cite{Fit2016}.
\section{Unified Quasi-Realizations for Intermediate Modal Logics}
Following Fitting in \cite{Fit2016}, we define the mapping $\quasi{\cdot}:\{T,F\}\times\mathcal{L}'_{\Box}\to\mathcal{P}(\{T,F\}\times\mathcal{L}_{J})$ by recursion on $\mathcal{L}'_{\Box}$:
\begin{enumerate}
\item For $\phi'\in Var\cup \{\bot\}$:\\
$\quasi{T,\phi'}:=\{(T,\phi')\}$; $\quasi{F,\phi'}=\{(F,\phi')\}$;\\
\item For $\circ\in\{\land,\lor\}$:\\
$\quasi{T,\phi'\circ\psi'}:=\{(T,\alpha\circ\beta)\mid (T,\alpha)\in\quasi{T,\phi'},(T,\beta)\in\quasi{T,\psi'}\}$;\\
$\quasi{F,\phi'\circ\psi'}:=\{(F,\alpha\circ\beta)\mid (F,\alpha)\in\quasi{F,\phi'},(F,\beta)\in\quasi{F,\psi'}\}$;\\
\item $\quasi{T,\phi'\rightarrow\psi'}:=\{(T,\alpha\rightarrow\beta)\mid (F,\alpha)\in\quasi{F,\phi'},(T,\beta)\in\quasi{T,\psi'}\}$;\\
$\quasi{F,\phi'\rightarrow\psi'}:=\{(F,\bigwedge_{i=1}^k\alpha_i\rightarrow\bigvee_{j=1}^l\beta_j)\mid (T,\alpha_i)\in\quasi{T,\phi'}, (F,\beta_j)\in\quasi{F,\psi'}\}$;\\
\item $\quasi{T,\Box_n\phi'}:=\{(T,x_n:\alpha)\mid (T,\alpha)\in\quasi{T,\phi'}\}$;\\
$\quasi{F,\Box_n\phi'}:=\{(F,t:(\alpha_1\lor\dots\lor\alpha_n))\mid (F,\alpha_i)\in\quasi{F,\phi'} (i\leq n),t\in Jt\}$;
\end{enumerate}
\begin{definition}
Given an uniquely annotated modal formula $\phi'$, a \emph{quasi-realization for} $\phi'$ is a formula $\alpha_1\lor\dots\lor\alpha_n$ with $(F,\alpha_i)\in\quasi{F,\phi'}$ ($i\leq n$). A quasi-realization for a modal formula $\phi$ is any quasi-realization for any unique annotation $\phi'$ of $\phi$.
\end{definition}
\subsection{Canonical models for intermediate justification logics, revisited}\label{sec:canmodcon}
The completeness proofs for intermediate justification logics (w.r.t. intuitionistic Fitting models) provided in Section \ref{sec:comptheoframes} of this paper are relying on a kind of canonical model construction, relative, however, to a given intuitionistic frame (and a corresponding propositional model) to achieve the respective generality w.r.t. the class of frames.

Achieving these results with the ``usual`` canonical model construction based on a frame over partially ordered maximal consistent sets or tableaux (such as the one defined later) is \emph{sometimes} even impossible, as one can not always control the properties of the partial order such that it lays in a desired class of frames. 

For instance, the canonical model which we are about to present is, if constructed over a classical justification logic, not a single world model, but one with isolated single worlds w.r.t. the partial order. The previous completeness theorems, however, provide completeness for classical justification logic w.r.t. single world models based on the corresponding completeness theorem for classical propositional logic and single world frames.\\

This usual (or standard) canonical model construction is, however, the main tool of the present section as we need to have precise control over the frame of the (canonical) model in question. So, this subsection now recalls and appropriately adapts this construction from the case of propositional intermediate logics (see \cite{CZ1997} for a comprehensive treatment of this propositional case for intermediate logics).

Throughout, let $\mathbf{L}$ be an intermediate logic and $\mathbf{LJL}$ be one of the associated intermediate justification logics with total constant specification.
\begin{definition}
A \emph{tableau} is a tuple $\tau=(\Gamma,\Delta)$ with $\Gamma,\Delta\subseteq\mathcal{L}_{J}$. $\tau$ is called $\mathbf{LJL}$-consistent if
\[
\Gamma\not\vdash_{\mathbf{LJL}}\phi_1\lor\dots\lor\phi_n
\]
for all $\phi_i\in\Delta$. $\tau$ is called maximal if $\Gamma\cup\Delta=\mathcal{L}_J$.
\end{definition}
\begin{lemma}[Lindenbaum]
Every $\mathbf{LJL}$-consistent tableau $\tau$ can be extended to a maximal $\mathbf{LJL}$-consistent tableau.
\end{lemma}
The proof is a straightforward generalization of the propositional case (see e.g. \cite{CZ1997}). We then can form the desired model. The main difference to Fitting's canonical model used in \cite{Fit2016}, besides the additional partial order to handle the intuitionistic implication, is the use of these tableaux instead of single maximal consistent sets as common in the study of intermediate logics (see e.g.  \cite{CZ1997}) as one can not control falsified formulae by their negation.
\begin{definition}
The \emph{standard canonical intuitionistic Fitting model for $\mathbf{LJL}$} is the structure $\mathfrak{M}^{sc}(\mathbf{LJL})=\langle\mathcal{W}^{sc},\preceq^{sc},\mathcal{R}^{sc},\mathcal{E}^{sc},\Vdash^{sc}\rangle$ which is defined by
\begin{enumerate}
\item $\mathcal{W}^{sc}:=\{\tau=(\Gamma,\Delta)\mid \tau\text{ is }\mathbf{LJL}\text{-consistent and maximal}\}$,
\item $\tau\preceq^{sc}\tau'$ iff $\Gamma\subseteq\Gamma'$ iff $\Delta\supseteq\Delta'$,
\item $\tau\mathcal{R}^{sc}\tau'$ iff $\Gamma^\#\subseteq\Gamma'$,
\item $\mathcal{E}^{sc}_t(\tau)=\{\phi\mid t:\phi\in\Gamma\}$,
\item $\Vdash^{sc}(p)=\{\tau=(\Gamma,\Delta)\mid p\in\Gamma\}$,
\end{enumerate}
where $\tau=(\Gamma,\Delta)$ and $\tau'=(\Gamma',\Delta')$.
\end{definition}
\begin{theorem}
Let $\mathbf{L}$ be an intermediate logic, $\mathbf{LJL}\in\{\mathbf{LJ},\mathbf{LJT},\mathbf{LJ4},\mathbf{LJT4}\}$ and $\mathfrak{M}^{sc}(\mathbf{LJL})=\langle\mathcal{W}^{sc},\preceq^{sc},\mathcal{R}^{sc},\mathcal{E}^{sc},\Vdash^{sc}\rangle$ its canonical model. For any $\phi\in\mathcal{L}_J$ and any $\tau=(\Gamma,\Delta)\in\mathcal{W}^{sc}$:
\begin{enumerate}
\item $\phi\in\Gamma\Rightarrow (\mathfrak{M}^{sc}(\mathbf{LJL}),\tau)\models\phi$;
\item $\phi\in\Delta\Rightarrow (\mathfrak{M}^{sc}(\mathbf{LJL}),\tau)\not\models\phi$.
\end{enumerate}
\end{theorem}
\begin{proof}
The proof is a simple extension of the similar results in the propositional case which can be found in \cite{CZ1997}. Similarly, we proceed by induction on $\phi$ and as the reasoning for the propositional connectives and atomic formulae given in \cite{CZ1997} also applies here, we only consider the case for $t:\phi$ where we assume
\begin{enumerate}
\item $\phi\in\Gamma'\Rightarrow (\mathfrak{M}^{sc}(\mathbf{LJL}),\tau')\models\phi$,
\item $\phi\in\Delta'\Rightarrow (\mathfrak{M}^{sc}(\mathbf{LJL}),\tau')\not\models\phi$,
\end{enumerate}
for any $\tau'=(\Gamma',\Delta')\in\mathcal{W}^{sc}$.\\

For (1), let $\tau=(\Gamma,\Delta)\in\mathcal{W}^{sc}$ and assume $t:\phi\in\Gamma$. Then, by definition $\phi\in\mathcal{E}^{sc}_t(\tau)$ and also for any $\tau'\in\mathcal{R}^{sc}[\tau]$: $\phi\in\Gamma'$. By induction hypothesis, we have $(\mathfrak{M}^{sc}(\mathbf{LJL}),\tau')\models\phi$ for any $\tau'\in\mathcal{R}^{sc}[\tau]$ and combined with $\phi\in\mathcal{E}^{sc}_t(\tau)$, we have $(\mathfrak{M}^{sc}(\mathbf{LJL}),\tau)\models t:\phi$.\\

Conversely, for (2), assume $t:\phi\in\Delta$. As $\tau$ is $\mathbf{LJL}$-consistent, we have $t:\phi\not\in\Gamma$ and thus $\phi\not\in\mathcal{E}^{sc}_t(\tau)$. Thus, immediately we have $(\mathfrak{M}^{sc}(\mathbf{LJL}),\tau)\not\models t:\phi$.
\end{proof}
\subsection{The main results}
We formulate the main lemma and the theorem on existence of quasi-realizations in the vein of Fitting's \cite{Fit2016}. For this, we also introduce the following notation:
\begin{definition}
Let $\mathfrak{F}$ be a Kripke frame and let $\mathfrak{M}=\langle\mathfrak{F},\mathcal{R},\mathcal{E},\mathcal{V}\rangle$ be a intuitionistic Fitting model over $\mathfrak{F}$. Let $x\in\mathcal{D}(\mathfrak{F})$ and let further $\phi'$ be an uniquely annotated modal formula. We write
\begin{enumerate}
\item $(\mathfrak{M},x)\models\quasi{T,\phi'}$ if $(\mathfrak{M},x)\models\alpha$ for all $(T,\alpha)\in\quasi{T,\phi'}$,
\item $(\mathfrak{M},x)\models\quasi{F,\phi'}$ if $(\mathfrak{M},x)\not\models\alpha$ for all $(F,\alpha)\in\quasi{F,\phi'}$.
\end{enumerate}
\end{definition}
\begin{lemma}
Let $\mathbf{L}$ be an intermediate logic and let $\mathbf{LJL}\in\{\mathbf{LJ},\mathbf{LJT},\mathbf{LJ4},\mathbf{LJT4}\}$ be a corresponding intermediate justification logic.

Let $\mathfrak{M}^{sc}:=\mathfrak{M}^{sc}(\mathbf{LJL})=\langle\mathcal{W}^{sc},\preceq^{sc},\mathcal{R}^{sc},\mathcal{E}^{sc},\Vdash^{sc}\rangle$ be the canonical model for $\mathbf{LJL}$ and define $\mathfrak{N}:=\langle\mathcal{W}^{sc},\preceq^{sc},\mathcal{R}^{sc},\Vdash^{sc}\rangle$. For all uniquely annotated $\phi'\in\mathcal{L}'_{\Box}$ and all $\tau\in\mathcal{W}^{sc}$:
\begin{enumerate}
\item $(\mathfrak{M}^{sc},\tau)\models\quasi{T,\phi'}\Rightarrow(\mathfrak{N},\tau)\models(\phi')^\bullet$;
\item $(\mathfrak{M}^{sc},\tau)\models\quasi{F,\phi'}\Rightarrow(\mathfrak{N},\tau)\not\models(\phi')^\bullet$.
\end{enumerate}
\end{lemma}
\begin{proof}
The proof is a modification of Fitting's from \cite{Fit2016}. We recite essential parts here for completeness. The proof proceeds by induction on the structure of $\phi'$. In the following, let $\tau=(\Gamma,\Delta)\in\mathcal{W}^{sc}$.\\

The statement is immediate for $\bot$ and for $p\in Var$. Suppose for the induction step that $\phi',\psi'$ are formulae with properties (1) and (2). We omit the induction steps for $\land$ and $\lor$. For $\rightarrow$ and $\Box_n$, we give, however, the following arguments:
\begin{enumerate}[(i)]
\item For (1), assume that $(\mathfrak{M}^{sc},\tau)\models\quasi{T,\phi'\rightarrow\psi'}$. Let $\tau'\in\mathcal{W}^{sc}$ be such that $\tau\preceq^{sc}\tau'$. Then, if $(\mathfrak{M}^{sc},\tau')\models\quasi{F,\phi'}$, then by induction hypothesis, we have $(\mathfrak{N},\tau')\not\models(\phi')^\bullet$. If $(\mathfrak{M}^{sc},\tau')\not\models\quasi{F,\phi'}$, then there is a $\alpha$ with $(F,\alpha)\in\quasi{F,\phi'}$ such that $(\mathfrak{M}^{sc},\tau)\models\alpha$. For any $\beta$ with $(T,\beta)\in\quasi{T,\psi'}$, we have $(T,\alpha\rightarrow\beta)\in\quasi{T,\phi'\rightarrow\psi'}$ by definition. By assumption of $(\mathfrak{M}^{sc},\tau)\models\quasi{T,\phi'\rightarrow\psi'}$, we have $(\mathfrak{M}^{sc},\tau)\models\alpha\rightarrow\beta$ and thus $(\mathfrak{M}^{sc},\tau')\models\alpha$ implies $(\mathfrak{M}^{sc},\tau')\models\beta$. As $\beta$ was arbitrary, we have $(\mathfrak{M}^{sc},\tau')\models\quasi{T,\psi'}$ and thus $(\mathfrak{N},\tau')\models(\psi')^\bullet$. Combined, we have $(\mathfrak{N},\tau)\models(\phi')^\bullet\rightarrow(\psi')^\bullet$.\\

Assume for (2) that $(\mathfrak{M}^{sc},\tau)\models\quasi{F,\phi'\rightarrow\psi'}$. For any $(T,\alpha_i)\in\quasi{T,\phi'}$ and any $(F,\beta_j)\in\quasi{F,\psi'}$, we have $(F,\bigwedge_i\alpha_i\rightarrow\bigvee_j\beta_j)\in\quasi{F,\phi'\rightarrow\psi'}$. Then, for $\tau=(\Gamma,\Delta)$, the tableau 
\[
\mu=(\Gamma\cup\{\alpha\mid (T,\alpha)\mid\quasi{T,\phi'}\},\{\beta\mid (F,\beta)\in\quasi{F,\psi'}\})
\]
is $\mathbf{LJL}$-consistent. For this, suppose not. Then, there are $\alpha_i\in\{\alpha\mid (T,\alpha)\mid\quasi{T,\phi'}\}$, $\beta_j\in\{\beta\mid (F,\beta)\in\quasi{F,\psi'}\}$ such that
\[
\Gamma\cup\{\alpha_1,\dots,\alpha_k\}\vdash_{\mathbf{LJL}}\beta_1\lor\dots\lor\beta_l
\]
By the deduction theorem we have
\[
\Gamma\vdash_{\mathbf{LJL}}\bigwedge_{i=1}^k\alpha_i\rightarrow\bigvee_{j=1}^l\beta_j
\]
which gives $\bigwedge_{i=1}^k\alpha_i\rightarrow\bigvee_{j=1}^l\beta_j\in\Gamma$ and therefore $(\mathfrak{M}^{sc},\tau)\models\bigwedge_{i=1}^k\alpha_i\rightarrow\bigvee_{j=1}^l\beta_j$ and this is a contradiction to $(\mathfrak{M}^{sc},\tau)\models\quasi{F,\phi'\rightarrow\psi'}$.

As $\mu$ is consistent, there is an extension to a maximal $\mathbf{LJL}$-consistent tableau $\tau'$. By construction, we have $\tau\preceq^{sc}\tau'$. Also, we have again by construction that
\[
(\mathfrak{M}^{sc},\tau')\models\quasi{T,\phi'}\text{ and }(\mathfrak{M}^{sc},\tau')\models\quasi{F,\psi'}.
\]
By the induction hypothesis, we have 
\[
(\mathfrak{N},\tau')\models(\phi')^\bullet\text{ and }(\mathfrak{N},\tau')\not\models (\psi')^\bullet
\]
which is $(\mathfrak{N},\tau)\not\models (\phi')^\bullet\rightarrow (\psi')^\bullet$ as $\tau\preceq^{sc}\tau'$.
\item For (1), suppose $(\mathfrak{M}^{sc},\tau)\models\quasi{T,\Box_n\phi'}$. This gives by definition $(\mathfrak{M}^{sc},\tau)\models x_n:\alpha$ for all $(T,\alpha)\in\quasi{T,\phi'}$. Thus, by definition of $\mathfrak{M}^{sc}$ we have $x_n:\alpha\in\Gamma$ and therefore $\alpha\in\Gamma^\#$. Let $\tau'=(\Gamma',\Delta')\in\mathcal{W}^{sc}$ with $\tau\mathcal{R}^{sc}\tau'$, then especially $\alpha\in\Gamma'$ by definition of $\mathcal{R}^{sc}$. As $\alpha$ was arbitrary, this entails
\[
(\mathfrak{M}^{sc},\tau')\models\quasi{T,\phi'}
\]
and by induction hypothesis, we have $(\mathfrak{N},\tau')\models(\phi')^\bullet$. As $\tau'$ was arbitrary, we have $(\mathfrak{N},\tau)\models\Box(\phi')^\bullet$.\\

For (2), suppose $(\mathfrak{M}^{sc},\tau)\models\quasi{F,\Box_n\phi'}$. Then, the tableau
\[
\mu=(\Gamma^\#,\{\alpha\mid (F,\alpha)\in\quasi{F,\phi'}\})
\]
is $\mathbf{LJL}$-consistent. Suppose not, then there are $\gamma_i\in\Gamma^\#$ with
\[
\{\gamma_1,\dots,\gamma_n\}\vdash_{\mathbf{LJL}}\alpha_1\lor\dots\lor\alpha_n.
\]
As $\gamma_i\in\Gamma^\#$, we have $t_i:\gamma_i\in\Gamma$ for some $t_i$. Thus, by the lifting lemma there is a $t\in Jt$ with
\[
\{t_1:\gamma_1,\dots,t_n:\gamma_n\}\vdash_{\mathbf{LJL}}t:(\alpha_1\lor\dots\lor\alpha_n).
\]
But then, by maximality of $\tau$ we have $t:(\alpha_1\lor\dots\lor\alpha_n)\in\Gamma$, i.e. $(\mathfrak{M}^{sc},\tau)\models t:(\alpha_1\lor\dots\lor\alpha_n)$, a contradiction to $(\mathfrak{M}^{sc},\tau)\models\quasi{F,\Box_n\phi'}$.

As $\mu$ is consistent, its has an extension to a maximal consistent tableau $\tau'$. Now, for this tableau $\tau'$, we have $\tau\mathcal{R}^{sc}\tau'$ by construction and we have, also by construction, that
\[
(\mathfrak{M}^{sc},\tau')\models\quasi{F,\phi'}.
\]
By the induction hypothesis, we have $(\mathfrak{N},\tau')\not\models(\phi')^\bullet$, and therefore by definition $(\mathfrak{N},\tau)\not\models\Box(\phi')^\bullet$.
\end{enumerate}
\end{proof}
\begin{theorem}\label{thm:bimodquasireal}
Let $\mathbf{L}$ be an intermediate logic and let $\mathbf{LJL}\in\{\mathbf{LJ},\mathbf{LJT},\mathbf{LJ4},\mathbf{LJT4}\}$ be a corresponding justification logic. Further, let $\mathbf{LML}$ be the modal logic corresponding to $\mathbf{LJL}$ and let $\mathsf{C}\in\mathsf{MFr}(\mathbf{LML})$. Suppose that $\langle\mathcal{W}^{sc},\preceq^{sc},\mathcal{R}^{sc}\rangle\in\mathsf{C}$ where $\mathfrak{M}^{sc}(\mathbf{LJL})=\langle\mathcal{W}^{sc},\preceq^{sc},\mathcal{R}^{sc},\mathcal{E}^{sc},\Vdash^{sc}\rangle$ is the canonical model for $\mathbf{LJL}$.

If $\mathbf{LML}\vdash\phi$, then there exists a quasi-realization $\alpha_1\lor\dots\lor\alpha_n$ of $\phi$ with $\mathbf{LJL}\vdash\alpha_1\lor\dots\lor\alpha_n$.
\end{theorem}
\begin{proof}
Suppose $\not\vdash_{\mathbf{LJL}}\alpha_1\lor\dots\lor\alpha_n$ for all $\alpha_i\in\quasi{F,\phi'}$. Then the tableau
\[
(\emptyset,\{\alpha\mid (F,\alpha)\in\quasi{F,\phi'}\})
\]
is $\mathbf{LJL}$-consistent. Thus, it extends to a maximal consistent tableau $\tau=(\Gamma,\Delta)\in\mathcal{W}^{sc}$. As $\{\alpha\mid (F,\alpha)\in\quasi{F,\phi'}\}\subseteq\Delta$, we have $(\mathfrak{M}^{sc}(\mathbf{LJL}),\tau)\models\quasi{F,\phi'}$. By the previous lemma, we have $(\mathfrak{N},\tau)\not\models(\phi')^\bullet$ for $\mathfrak{N}:=\langle\mathcal{W}^{sc},\preceq^{sc},\mathcal{R}^{sc},\Vdash^{sc}\rangle$. As $\langle\mathcal{W}^{sc},\preceq^{sc},\mathcal{R}^{sc}\rangle\in\mathsf{C}$, we have by the choice of $\mathsf{C}$ that $\mathbf{LML}\not\vdash\phi$.
\end{proof}
\section{Unified Realization for Intermediate Modal Logics}
In this section, we adapt Fittings algorithm for the construction of realizations from quasi-realizations to the intermediate case, culminating in a realization theorem for intermediate modal logics. This amounts, modulo some modifications in the $\rightarrow$-case, to verifying that Fitting's construction and proof from \cite{Fit2016} also works in $\mathbf{IPCJ}$.\\

Following  \cite{Fit2016}, we introduce a special notation for the following algorithm transforming Quasi-Realizations into Realizations.
\begin{definition}
Let $\phi'\in\mathcal{L}'_{\Box}$ and $\Gamma\cup\{\psi\}\subseteq\mathcal{L}_{J}$ where $\Gamma$ is finite. Let $\sigma$ be a justification substitution. We write:
\begin{enumerate}
\item $\Gamma\overset{T\phi'}{\longrightarrow}(\psi,\sigma)$ if (1) $\{T\}\times\Gamma\subseteq\quasi{T,\phi'}$, (2) $(T,\psi)\in\real{T,\phi'}$ and
\[
(3)\;\mathbf{IPCJ}\vdash\psi\rightarrow \left(\bigwedge\Gamma\right)\sigma;
\]
\item $\Gamma\overset{F\phi'}{\longrightarrow}(\psi,\sigma)$ if (1) $\{F\}\times\Gamma\subseteq\quasi{F,\phi'}$, (2) $(F,\psi)\in\real{F,\phi'}$ and 
\[
(3)\;\mathbf{IPCJ}\vdash\left(\bigvee\Gamma\right)\sigma\rightarrow\psi.
\]
\end{enumerate}
\end{definition}
The following algorithm is a slight modification of that Fitting from \cite{Fit2016}.
\begin{algorithm}\label{alg:quasitoreal}
\begin{description}
\item[Atomic Case] The atomic propositions have a trivial realization through the empty realization function $\varepsilon$:
\begin{gather*}\{p\}\overset{Tp}{\longrightarrow}(p,\varepsilon)\qquad\{p\}\overset{Fp}{\longrightarrow}(p,\varepsilon)\\
\{\bot\}\overset{T\bot}{\longrightarrow}(\bot,\varepsilon)\qquad\{\bot\}\overset{F\bot}{\longrightarrow}(\bot,\varepsilon)
\end{gather*}
\item[$T\land$ Case] \[
\frac{\{\alpha_1,\dots,\alpha_k\}\overset{F\phi'}{\longrightarrow}(\chi,\sigma_{\phi'})\quad\{\beta_1,\dots,\beta_k\}\overset{T\psi'}{\longrightarrow}(\xi,\sigma_{\psi'})}{\{\alpha_1\land \beta_1,\dots,\alpha_k\land \beta_k\}\overset{T\phi'\land \psi'}{\longrightarrow}((\chi\land \xi)\sigma_{\phi'}\sigma_{\psi'},\sigma_{\phi'}\sigma_{\psi'})}
\]
\item[$F\land$ Case] \[
\frac{\{\alpha_1,\dots,\alpha_k\}\overset{T\phi'}{\longrightarrow}(\chi,\sigma_{\phi'})\quad\{\beta_1,\dots,\beta_k\}\overset{F\psi'}{\longrightarrow}(\xi,\sigma_{\psi'})}{\{\alpha_1\land \beta_1,\dots,\alpha_k\land \beta_k\}\overset{F\phi'\land \psi'}{\longrightarrow}((\chi\land \xi)\sigma_{\phi'}\sigma_{\psi'},\sigma_{\phi'}\sigma_{\psi'})}
\]
\item[$T\lor$ Case] \[
\frac{\{\alpha_1,\dots,\alpha_k\}\overset{F\phi'}{\longrightarrow}(\chi,\sigma_{\phi'})\quad\{\beta_1,\dots,\beta_k\}\overset{T\psi'}{\longrightarrow}(\xi,\sigma_{\psi'})}{\{\alpha_1\lor \beta_1,\dots,\alpha_k\lor \beta_k\}\overset{T\phi'\lor \psi'}{\longrightarrow}((\chi\lor \xi)\sigma_{\phi'}\sigma_{\psi'},\sigma_{\phi'}\sigma_{\psi'})}
\]
\item[$F\lor$ Case] \[
\frac{\{\alpha_1,\dots,\alpha_k\}\overset{T\phi'}{\longrightarrow}(\chi,\sigma_{\phi'})\quad\{\beta_1,\dots,\beta_k\}\overset{F\psi'}{\longrightarrow}(\xi,\sigma_{\psi'})}{\{\alpha_1\lor \beta_1,\dots,\alpha_k\lor \beta_k\}\overset{F\phi'\lor \psi'}{\longrightarrow}((\chi\lor \xi)\sigma_{\phi'}\sigma_{\psi'},\sigma_{\phi'}\sigma_{\psi'})}
\]
\item[$T\rightarrow$ Case] \[
\frac{\{\alpha_1,\dots,\alpha_k\}\overset{F\phi'}{\longrightarrow}(\chi,\sigma_{\phi'})\quad\{\beta_1,\dots,\beta_k\}\overset{T\psi'}{\longrightarrow}(\xi,\sigma_{\psi'})}{\{\alpha_1\rightarrow \beta_1,\dots,\alpha_k\rightarrow \beta_k\}\overset{T\phi'\rightarrow \psi'}{\longrightarrow}(\chi\sigma_{\psi'}\rightarrow \xi\sigma_{\phi'},\sigma_{\phi'}\sigma_{\psi'})}
\]
\item[$F\rightarrow$ Case] \[
\frac{\Gamma_1\cup\dots\cup\Gamma_k\overset{T\phi'}{\longrightarrow}(\chi,\sigma_{\phi'})\quad\Delta_1\cup\dots\cup\Delta_k\overset{F\psi'}{\longrightarrow}(\xi,\sigma_{\psi'})}{\left\{\bigwedge\Gamma_1\rightarrow\bigvee\Delta_1,\dots,\bigwedge\Gamma_k\rightarrow\bigvee\Delta_k\right\}\overset{F\phi'\rightarrow \psi'}{\longrightarrow}(\chi\sigma_{\psi'}\rightarrow \xi\sigma_{\phi'},\sigma_{\phi'}\sigma_{\psi'})}
\]
\item[$T\Box$ Case] \[
\frac{\{\alpha_1,\dots,\alpha_k\}\overset{T\phi'}{\longrightarrow}(\chi,\sigma_{\phi'})}{\{v_n:\alpha_1,\dots,v_n:\alpha_k\}\overset{T\Box_n\phi'}{\longrightarrow}(v_n:\chi\sigma,\sigma_{\phi'}\sigma)}\quad\parbox{14em}{with $\mathbf{IPCJ}\vdash t_i:(\chi\rightarrow \alpha_i\sigma_{\phi'})$ and $\sigma(v_n)=[s\cdot v_n]$, $\sigma(v_k)=v_k$ for $k\neq n$ where $s=[t_1+\dots+t_n]$.}
\]
\item[$F\Box$ Case] \[
\frac{\Gamma_1\cup\dots\cup\Gamma_k\overset{F\phi'}{\longrightarrow}(\chi,\sigma_{\phi'})}{\left\{t_1:\bigvee\Gamma_1,\dots,t_k:\bigvee\Gamma_k\right\}\overset{F\Box_n\phi'}{\longrightarrow}(t\sigma_{\phi'}:\chi,\sigma_{\phi'})}\quad\parbox{15em}{with $\mathbf{IPCJ}\vdash u_i:(\bigvee\Gamma_1\sigma_{\phi'}\rightarrow \chi)$ and $t=[[u_1\cdot t_1]+\dots+[u_k\cdot t_k]]$.}
\]
\end{description}
\end{algorithm}
\begin{theorem}\label{thm:algcorrect}
Let $\phi'\in\mathcal{L}'_{\Box}$ and let $\Gamma\subseteq\mathcal{L}_{J}$ be nonempty and finite. Then:
\begin{enumerate}
\item If $\{T\}\times\Gamma\subseteq\quasi{T,\phi'}$, then there are $\psi\in\mathcal{L}_{J}$ and a justification substitution $\sigma$ such that $\Gamma\overset{T\phi'}\longrightarrow(\psi,\sigma)$.
\item If $\{F\}\times\Gamma\subseteq\quasi{F,\phi'}$, then there are $\psi\in\mathcal{L}_{J}$ and a justification substitution $\sigma$ such that $\Gamma\overset{F\phi'}\longrightarrow(\psi,\sigma)$.
\end{enumerate}
\end{theorem}
\begin{proof}
By recursion on $\phi'$, using Algorithm \ref{alg:quasitoreal}, we construct the desired pair $(\psi,\sigma)$. Note, that the $T$ and $F$ rules of the algorithm preserve $\{T\}\times\Gamma\subseteq\quasi{T,\phi'}$ or $\{F\}\times\Gamma\subseteq\quasi{F,\phi'}$, respectively.

The cases for $T\rightarrow$, $T\Box$, $F\Box$ as well as the atomic case were handled in \cite{Fit2016} and the arguments also apply here. We omit the cases for $\land$ and $\lor$ as they are quite elementary. We give the one of $T\rightarrow$ in some detail. The main difference is however in the case for $F\rightarrow$, as we have modified the definition of $\quasi{F,\phi'\rightarrow\psi'}$.
\begin{description}
\item [($F\rightarrow$)] 
Assume that we have
\begin{enumerate}[(i)]
\item $\Gamma_1\cup\dots\cup\Gamma_k\overset{T\phi'}{\longrightarrow}(\chi,\sigma_{\phi'})$,
\item $\Delta_1\cup\dots\cup\Delta_k\overset{F\psi'}{\longrightarrow}(\xi,\sigma_{\psi'})$.
\end{enumerate}
Then we have
\[
\{\bigwedge\Gamma_1\rightarrow\bigvee\Delta_1,\dots,\bigwedge\Gamma_k\rightarrow\bigvee\Delta_k\}\subseteq\quasi{F,\phi'\rightarrow \psi'}.
\]
We have, by the requirements on substitutions (similarly as in \cite{Fit2016}), that $\sigma_{\phi'}\sigma_{\psi'}=\sigma_{\psi'}\sigma_{\phi'}$ and by (i) and (ii), we have 
\[
(T,\chi\sigma_{\psi'})\in\real{T,{\phi'}}\text{ and }(F,\xi\sigma_{\phi'})\in\real{F,{\psi'}}
\]
as we have $(T,\chi)\in\real{T,{\phi'}}$ and $\real{\cdot}$ is closed under substitutions (and similarly for $\real{F,B}$). Thus, we have
\[
(F,\chi\sigma_{\psi'}\rightarrow \xi\sigma_{\phi'})\in\real{F,{\phi'}\rightarrow {\psi'}}.
\]
Now, we obtain
\[
\mathbf{IPCJ}\vdash\bigvee_{i=1}^k\left(\bigwedge\Gamma_i\rightarrow\bigvee\Delta_i\right)\sigma_{\phi'}\sigma_{\psi'}\rightarrow (\chi\sigma_{\psi'}\rightarrow \xi\sigma_{\phi'})
\]
as by (i), we have
\[
\mathbf{IPCJ}\vdash \chi\sigma_{\psi'}\rightarrow \left(\bigwedge_{i=1}^k\bigwedge\Gamma_i\right)\sigma_{\phi'}\sigma_{\psi'}
\]
and by (ii), we have
\[
\mathbf{IPCJ}\vdash\left(\bigvee_{i=1}^k\bigvee\Delta_i\right)\sigma_{\psi'}\sigma_{\phi'}\rightarrow \xi\sigma_{\phi'}.
\]
As $\sigma_{\phi'}\sigma_{\psi'}=\sigma_{\psi'}\sigma_{\phi'}$ and by the other properties of substitutions, we obtain the claim by utilizing the following validity of intuitionistic logic:
\[
\mathbf{IPCJ}\vdash\bigwedge_{i=1}^k\bigwedge\Gamma_i\land\bigvee^k_{i=1}\left(\bigwedge\Gamma_i\rightarrow\bigvee\Delta_i\right)\rightarrow\bigvee_{i=1}^k\bigvee\Delta_i.
\]
\item [($T\rightarrow$)] Suppose that 
\begin{enumerate}[(i)]
\item $\{\alpha_1,\dots,\alpha_k\}\overset{F{\phi'}}{\longrightarrow}(\chi,\sigma_{\phi'})$,
\item $\{\beta_1,\dots,\beta_k\}\overset{T{\psi'}}{\longrightarrow}(\xi,\sigma_{\psi'})$.
\end{enumerate}
Then we naturally have
\[
\{\alpha_1\rightarrow \beta_1,\dots,\alpha_k\rightarrow \beta_k\}\subseteq\quasi{T,\phi'\rightarrow\psi'}.
\]
As before, one shows $\sigma_{\phi'}\sigma_{\psi'}=\sigma_{\psi'}\sigma_{\phi'}$ and also similarly one shows
\[
(T,\chi\sigma_{\psi'}\rightarrow \xi\sigma_{\phi'})\in\real{T,{\phi'}\rightarrow {\psi'}}.
\]
Now, have by (i) that
\[
\mathbf{IPCJ}\vdash\left(\bigvee_{i=1}^k \alpha_k\right)\sigma_{\phi'}\rightarrow \chi
\]
and by (ii):
\[
\mathbf{IPCJ}\vdash \xi\rightarrow\left(\bigwedge_{i=1}^k\beta_k\right)\sigma_{\psi'}
\]
We obtain
\[
\mathbf{IPCJ}\vdash (\chi\sigma_{\psi'}\rightarrow \xi\sigma_{\phi'})\rightarrow\bigwedge_{i=1}^k(\alpha_i\rightarrow \beta_i)\sigma_{\phi'}\sigma_{\psi'}
\]
by utilizing the following validity of intuitionistic logic:
\[
\mathbf{IPCJ}\vdash\left(\bigvee_{i=1}^k\alpha_i\rightarrow\bigwedge_{i=1}^k \beta_i\right)\rightarrow\bigwedge_{i=1}^k(\alpha_i\rightarrow \beta_i).
\]
\end{description}
\end{proof}
\begin{theorem}\label{thm:bimodreal}
Let $\mathbf{L}$ be an intermediate logic and let $\mathbf{LJL}\in\{\mathbf{LJ},\mathbf{LJT},\mathbf{LJ4},\mathbf{LJT4}\}$ be a corresponding justification logic. Further, let $\mathbf{LML}$ be the modal logic corresponding to $\mathbf{LJL}$ and let $\mathsf{C}\in\mathsf{MFr}(\mathbf{LML})$. Let $\mathfrak{M}^{sc}(\mathbf{LJL})=\langle\mathcal{W}^{sc},\preceq^{sc},\mathcal{R}^{sc},\mathcal{E}^{sc},\Vdash^{sc}\rangle$ be the canonical model of $\mathbf{LJL}$ and suppose that $\langle\mathcal{W}^{sc},\preceq^{sc},\mathcal{R}^{sc}\rangle\in\mathsf{C}$.

If $\mathbf{LML}\vdash\phi$, then there exists a $\psi\in\mathcal{L}_J$ with $(F,\psi)\in\real{F,\phi'}$ with $\mathbf{LJL}\vdash\psi$.
\end{theorem}
\begin{proof}
By Theorem \ref{thm:bimodquasireal}, $\mathbf{LML}\vdash\phi$ implies that there exist $(F,\alpha_i)\in\quasi{F,\phi'}$, for some annotation $\phi'$ of $\phi$, such that $\mathbf{LJL}\vdash\alpha_1\lor\dots\lor\alpha_n$.\\

Now, by Theorem \ref{thm:algcorrect}, there is a $\sigma$ and a $(F,\psi)\in\real{F,\phi'}$ such that $\{(F,\alpha_i)\mid i\leq n\}\overset{F\phi'}{\longrightarrow}(\psi,\sigma)$. By definition, $\mathbf{IPCJ}\vdash(\alpha_1\lor\dots\lor\alpha_n)\sigma\rightarrow\psi$ and thus $\mathbf{LJL}\vdash\psi$ as Lemma \ref{lem:sublemma} implies $\mathbf{LJL}\vdash(\alpha_1\lor\dots\lor\alpha_n)\sigma$.
\end{proof}
\section{Conclusion}
The completeness theorems proved in this paper show that a unified completeness theorem stands behind many central completeness result for justification logics in the literature, lifting classes of algebras or classes of Kripke frames, complete for some intermediate logic, to a complete model class for the corresponding justification logic. Key to this is of course the strong completeness assumption of the underlying propositional logic and, in particular, that of global completeness in the case of Kripke frames.

We want to acknowledge that the algebraic completeness results can be generalized in an immediate way. E.g., consider an algebraic Fitting model $\mathfrak{M}=\langle\mathbf{A},\mathcal{W},\mathcal{R},\mathcal{E},\mathcal{V}\rangle$ over a complete Heyting algebra $\mathbf{A}$. $\mathbf{A}$ can be generalized to not be complete but only complete for sets $X\subseteq A$ of cardinality $\leq\mathrm{card}(\mathcal{W})$, similarly as in the case of the Kripke-models taking values in Heyting algebras for intuitionistic modal logics from Ono \cite{Ono1977}. This of course also applies to the algebraic subset models.\\

In a similar sense as with the completeness theorem, the realization theorem proved here shows that there is a unified result behind many of the previous realization theorems from the literature. Regarding realization, similar remarks as in Fitting's \cite{Fit2016} apply here. For one, while the merging (or condensing) of quasi-realizations into realizations is proved constructively, the existence of quasi-realizations is proved non-constructively. 

Historically of course, the existence of realizations in the classical case was proved constructively using structural proof theory for the modal logic in question and similarly one can (in some cases) use proof-theoretic formulations in terms of e.g. cut-free sequent calculi to show the existence of quasi-realizations constructively (see e.g. \cite{AF2019}) which combined with the above condensing results gives constructive realization. In comparison to classical and intuitionistic modal logics, where such a sequent formulation is available, intermediate logics and their modal extensions usually require more sophisticated proof-theoretic formalisms, using, e.g., hypersequents. Following this vein, in \cite{Pis2019}, a realization theorem for the G\"odel justification logics is proved constructively using a hypersequent formulation of the corresponding modal logic and an appropriate adoption of Fitting's approach to realization merging from \cite{Fit2009}. This may be extended to a larger class of intermediate justification logics, using the work on structural proof calculi for intermediate logics from e.g. \cite{CGT2008} if appropriately extended to intermediate modal logics. Using the above condensing results, it even suffices to use structural proof theory to just construct quasi-realizations for these intermediate cases.

It shall be noted that even this non-constructive approach is limited to intermediate logics where the corresponding modal logic is Kripke complete with certain additional requirements on the class of frames. While this condition is very broad (as will be exhibited in the following examples) it for certain does not encompass Kripke incomplete intermediate logics and their corresponding modal logics (see again \cite{She1977}).\\

We end this section by a review of some old results reobtained as corollaries of the here proved completeness and realization theorems and by mentioning some new intermediate justification logics and their corresponding completeness and realization results as obtainable through this paper.
\subsection{Some intermediate logics and their semantics}
\subsubsection{$\mathbf{IPC}$, $\mathbf{G}$ and $\mathbf{C}$}
We have the following algebraic completeness theorems for $\mathbf{IPC}$, $\mathbf{G}$ and $\mathbf{C}$ as introduced before.

We use $\mathsf{H}$ to denote the class of all Heyting algebras. An important instance of a linearly ordered Heyting algebra is the \emph{standard G\"odel algebra} $\mathbf{[0,1]_G}$ given by
\[
\mathbf{[0,1]_G}:=\langle[0,1],\min,\max,\rightsquigarrow,0,1\rangle
\]
where
\[
x\rightsquigarrow y:=\begin{cases}1&\text{if }x\leq y\\y&\text{otherwise}\end{cases}
\]
for $x,y\in [0,1]$. The natural canonical choice of a Boolean algebra is the algebra
\[
\mathbf{\{0,1\}_B}:=\langle\{0,1\},\min,\max,\rightsquigarrow,0,1\rangle
\]
with the above function $\rightsquigarrow$ restricted to $\{0,1\}$.
\begin{theorem}
We have the following algebraic completeness results:
\begin{enumerate}
\item $\mathbf{IPC}$ is strongly complete with respect to $\mathsf{H}_{fin}$;
\item $\mathbf{G}$ is strongly complete with respect to $\mathbf{[0,1]_G}$;
\item $\mathbf{C}$ is strongly complete with respect to $\mathbf{\{0,1\}_B}$.
\end{enumerate}
\end{theorem}
All items are folklore. For item (2), see especially \cite{Dum1959}.\\

Further, we have the following Kripke-style completeness theorems for $\mathbf{IPC}$, $\mathbf{G}$ and $\mathbf{CPC}$.
\begin{definition}
Let $\mathfrak{F}=\langle F,\leq\rangle$ be a Kripke frame. $\mathfrak{F}$ is called
\begin{enumerate}
\item \emph{connected} if $\forall x,y,z\in F\left(x\leq y\land x\leq z\Rightarrow y\leq z\lor z\leq y\right)$,
\item \emph{of bounded cardinality $n$} if $\forall x_0,x_1,\dots,x_n\in F\left(\bigwedge_{i=1}^nx_0\leq x_i\Rightarrow \bigwedge_{i\neq j}x_i=x_j\right)$.
\end{enumerate}
\end{definition}
We write $\mathsf{IF}$ for the class of all intuitionistic Kripke frames and $\mathsf{CIF}$ be the class of all connected intuitionistic Kripke frames. Further, given a class $\mathsf{C}$ of intuitionistic Kripke frames, we write $\mathsf{C}_n$ for the subclass of all intuitionistic Kripke frames with bounded cardinality $n$ in $\mathsf{C}$.

As a well-known result, we have:
\begin{theorem}\label{thm:propcomp}
For any $\Gamma\cup\{\phi\}\subseteq\mathcal{L}_J$:
\begin{enumerate}
\item $\Gamma\vdash_{\mathbf{IPC}}\phi$ iff $\Gamma\models_{\mathsf{IF}}\phi$;
\item $\Gamma\vdash_{\mathbf{G}}\phi$ iff $\Gamma\models_{\mathsf{CIF}}\phi$;
\item $\Gamma\vdash_{\mathbf{C}}\phi$ iff $\Gamma\models_{\mathsf{IF}_1}\phi$.
\end{enumerate}
\end{theorem}
Item (1) goes back to Kripke's work \cite{Kri1965}. Item (2) is not that easily traceable but can be found in \cite{CZ1997}. Combining this with the fact that $\mathsf{IF}$ as well as $\mathsf{CIF}$ and $\mathsf{IF}_1$ are closed under principal subframes, we have by Lemma \ref{lem:localglobalequiv}:
\begin{corollary}[of Theorem \ref{thm:propcomp}]\label{cor:globalpropcomp}
For any $\Gamma\cup\{\phi\}\subseteq\mathcal{L}_J$:
\begin{enumerate}[(1)']
\item $\Gamma\vdash_{\mathbf{IPC}}\phi$ iff $\Gamma\models^g_{\mathsf{IF}}\phi$;
\item $\Gamma\vdash_{\mathbf{G}}\phi$ iff $\Gamma\models^g_{\mathsf{CIF}}\phi$;
\item $\Gamma\vdash_{\mathbf{C}}\phi$ iff $\Gamma\models^g_{\mathsf{IF}_1}\phi$.
\end{enumerate}
\end{corollary}
\subsubsection{$\mathbf{G_n}$ and $\mathbf{KC}$}
Prominent strengthenings of the infinite-valued G\"odel logic $\mathbf{G}$ (or G\"odel-Dummet logic) are the finite valued G\"odel logics $\mathbf{G_n}$. These actually pre-date $\mathbf{G}$ in the sense that this sequence of intermediate logics is the one used by G\"odel in \cite{Goe1932} for his investigations about intuitionistic logic, whereas $\mathbf{G}$ was later defined by Dummet in \cite{Dum1959}. Axiomatically, we can give the following description of $\mathbf{G_n}$. Consider the axiom scheme
\begin{enumerate}
\item [$(BC)_n$] $\bigvee_{i=0}^n\left(\bigwedge_{j<i}p_j\rightarrow p_i\right)$
\end{enumerate}
for any $n\geq 1$. Then, we define the $n$-valued G\"odel logic by
\[
\mathbf{G_n}:=\mathbf{G}+ (BC)_{n-1}
\]
for $n\geq 2$. The notation for the axiom scheme comes from its use in intermediate logics of \emph{bounded cardinality}. The usual semantics for $\mathbf{G_n}$, $n\geq 2$, is given by a characteristic matrix through the Heyting algebra
\[
\mathbf{V^{(n)}_G}:=\langle V^{(n)},\min,\max,\rightsquigarrow,0,1\rangle
\]
with
\[
V^{(n)}:=\left\{1-\frac{1}{k}\mid 1\leq k\leq n-1\right\}\cup\{1\}
\]
and the operation $\rightsquigarrow$ as before, restricted to $V^{(n)}$. Indeed, we then have the following completeness theorem:
\begin{theorem}\label{thm:propcompgnalg}
$\mathbf{G_n}$ is strongly complete with respect to $\mathbf{V^{(n)}_G}$.
\end{theorem}
For the above theorem and more background on the finite valued G\"odel logics, we refer to \cite{CF2001,Pre2010}.

We can also give the following Kripke-style completeness theorem (a proof can be found in \cite{CZ1997}):
\begin{theorem}\label{thm:propcompgn}
$\mathbf{G_n}$ is strongly complete w.r.t $\mathsf{CIF}_{n-1}$.
\end{theorem}
As also $\mathsf{CIF}_n$ is closed under principal subframes, we have the following corollary:
\begin{corollary}
$\mathbf{G_n}$ is strongly globally complete w.r.t. $\mathsf{CIF}_{n-1}$.
\end{corollary}
A second intermediate logic which we want to consider is based on a weakening of the law of the excluded middle. Consider the axiom scheme
\begin{description}
\item [($WLEM$)] $\neg\neg\phi\lor\neg\phi$
\end{description}
and the corresponding logic of the weak law of the excluded middle, also known as Jankov's logic (introduced in \cite{Jan1968}), given by
\[
\mathbf{KC}:=\mathbf{IPC}+(WLEM).
\]
A classical result is the completeness result in terms of directed Kripke frames. For this consider the following definition:
\begin{definition}
A Kripke frame $\langle F,\leq\rangle$ is called \emph{directed} if 
\[
\forall x,y,z\in F\left(x\leq y\land x\leq z\Rightarrow\exists w\in F\left(y\leq w\land z\leq w\right)\right).
\]
\end{definition}
Let $\mathsf{DIF}$ be the class of all directed intuitionistic Kripke frames. Then, one obtains the following semantical characterization.
\begin{theorem}\label{thm:propcompkcframe}
$\mathbf{KC}$ is strongly complete w.r.t. $\mathsf{DIF}$.
\end{theorem}
A proof can again be found in \cite{CZ1997}. Considering that the model class in question is closed under principal subframes, we again have the following corollary based on Lemma \ref{lem:localglobalequiv}.
\begin{corollary}\label{cor:propcompkcframe}
$\mathbf{KC}$ is strongly globally complete w.r.t. $\mathsf{DIF}$.
\end{corollary}
\subsection{Completeness Theorems}
\subsection{$\mathbf{IPC}$, $\mathbf{G}$ and $\mathbf{C}$}
Based on Theorems \ref{thm:algmkrtmodcomp}, \ref{thm:algfittingmodcomp} and \ref{thm:algsubsetmodcomp}, we obtain the following particular corollaries.
\begin{corollary}\label{cor:intcompalg}
Let $\mathbf{IPCJL}_0\in\{\mathbf{IPCJ}_0,\mathbf{IPCJT}_0,\mathbf{IPCJ4}_0,\mathbf{IPCJT4}_0\}$ where $CS$ is a constant specification for $\mathbf{IPCJL}_0$. For any $\Gamma\cup\{\phi\}\subseteq\mathcal{L}_J$, the following are equivalent:
\begin{enumerate}
\item $\Gamma\vdash_{\mathbf{IPCJL}_{CS}}\phi$;
\item $\Gamma\models^1_{\mathsf{H}_{fin}\mathsf{AMJL}_{CS}}\phi$;
\item $\Gamma\models^1_{\mathsf{H}_{fin}\mathsf{AFJL^c}_{CS}}\phi$;
\item $\Gamma\models^1_{\mathsf{H}_{fin}\mathsf{ASJL^c}_{CS}}\phi$.
\end{enumerate}
\end{corollary}
\begin{corollary}\label{cor:goecompalg}
Let $\mathbf{GJL}_0\in\{\mathbf{GJ}_0,\mathbf{GJT}_0,\mathbf{GJ4}_0,\mathbf{GJT4}_0\}$ where $CS$ is a constant specification for $\mathbf{GJL}_0$. For any $\Gamma\cup\{\phi\}\subseteq\mathcal{L}_J$, the following are equivalent:
\begin{enumerate}
\item $\Gamma\vdash_{\mathbf{GJL}_{CS}}\phi$;
\item $\Gamma\models^1_{\mathbf{[0,1]_G}\mathsf{AMJL}_{CS}}\phi$;
\item $\Gamma\models^1_{\mathbf{[0,1]_G}\mathsf{AFJL^c}_{CS}}\phi$;
\item $\Gamma\models^1_{\mathbf{[0,1]_G}\mathsf{ASJL^c}_{CS}}\phi$.
\end{enumerate}
\end{corollary}
\begin{corollary}\label{cor:clacompalg}
Let $\mathbf{CJL}_0\in\{\mathbf{CJ}_0,\mathbf{CJT}_0,\mathbf{CJ4}_0,\mathbf{CJT4}_0\}$ where $CS$ is a constant specification for $\mathbf{CJL}_0$. For any $\Gamma\cup\{\phi\}\subseteq\mathcal{L}_J$, the following are equivalent:
\begin{enumerate}
\item $\Gamma\vdash_{\mathbf{CJL}_{CS}}\phi$;
\item $\Gamma\models^1_{\mathbf{\{0,1\}_B}\mathsf{AMJL}_{CS}}\phi$;
\item $\Gamma\models^1_{\mathbf{\{0,1\}_B}\mathsf{AFJL^c}_{CS}}\phi$;
\item $\Gamma\models^1_{\mathbf{\{0,1\}_B}\mathsf{ASJL^c}_{CS}}\phi$.
\end{enumerate}
\end{corollary}
These theorems contain several well-known results from the literature on semantics for justification logics. At first, the equivalence between (1) and (2), (3), (4) in Corollary \ref{cor:clacompalg} are the known completeness theorems of Mkrtychev \cite{Mkr1997}, Fitting \cite{Fit2005} as well as Lehmann and Studer \cite{LS2019}, respectively. Further, the equivalence between (1) and (2), (3) in Corollary \ref{cor:goecompalg} are among the completeness results previously obtained for the G\"odel justification logics in \cite{Pis2020}.\\

The previous work on semantics of intuitionistic justification logics in the sense of the present paper is mainly \cite{MS2016} where the authors considered models for $\mathbf{IPCJT4}$ based on extensions of Kripke frames for intuitionistic propositional logic by the semantic machinery for justification logics from Mkrtychev's and Fitting's models, similar as with the intuitionistic Mkrtychev and Fitting models considered here. The above corollary \ref{cor:intcompalg} gives a different semantic approach to $\mathbf{IPCJT4}$.\\

However, we reobtain these results of \cite{MS2016} through the completeness theorems proved here regarding Mkrtychev, Fitting and subset models over Kripke frames. By Theorem \ref{thm:propcomp} and Corollary \ref{cor:globalpropcomp}, the completeness theorems based on Kripke frames apply and we obtain the following completeness theorems.
\begin{corollary}\label{cor:intcompframe}
Let $\mathbf{IPCJL}_0\in\{\mathbf{IPCJ}_0,\mathbf{IPCJT}_0,\mathbf{IPCJ4}_0,\mathbf{IPCJT4}_0\}$ where $CS$ is a constant specification for $\mathbf{IPCJL}_0$. For any $\Gamma\cup\{\phi\}\subseteq\mathcal{L}_J$, the following are equivalent:
\begin{enumerate}
\item $\Gamma\vdash_{\mathbf{IPCJL}_{CS}}\phi$;
\item $\Gamma\models_{\mathsf{IFKMJL}_{CS}}\phi$;
\item $\Gamma\models_{\mathsf{IFKFJL}_{CS}}\phi$;
\item $\Gamma\models_{\mathsf{IFKSJL}_{CS}}\phi$.
\end{enumerate}
\end{corollary}
\begin{corollary}\label{cor:goecompframe}
Let $\mathbf{GJL}_0\in\{\mathbf{GJ}_0,\mathbf{GJT}_0,\mathbf{GJ4}_0,\mathbf{GJT4}_0\}$ where $CS$ is a constant specification for $\mathbf{GJL}_0$. For any $\Gamma\cup\{\phi\}\subseteq\mathcal{L}_J$, the following are equivalent:
\begin{enumerate}
\item $\Gamma\vdash_{\mathbf{GJL}_{CS}}\phi$;
\item $\Gamma\models_{\mathsf{CIFKMJL}_{CS}}\phi$;
\item $\Gamma\models_{\mathsf{CIFKFJL}_{CS}}\phi$;
\item $\Gamma\models_{\mathsf{CIFKSJL}_{CS}}\phi$.
\end{enumerate}
\end{corollary}
\begin{corollary}\label{cor:clacompframe}
Let $\mathbf{CPCJL}_0\in\{\mathbf{CPCJ}_0,\mathbf{CPCJT}_0,\mathbf{CPCJ4}_0,\mathbf{CPCJT4}_0\}$ where $CS$ is a constant specification for $\mathbf{CPCJL}_0$. For any $\Gamma\cup\{\phi\}\subseteq\mathcal{L}_J$, the following are equivalent:
\begin{enumerate}
\item $\Gamma\vdash_{\mathbf{CPCJL}_{CS}}\phi$;
\item $\Gamma\models_{\mathsf{IF_1KMJL}_{CS}}\phi$;
\item $\Gamma\models_{\mathsf{IF_1KFJL}_{CS}}\phi$;
\item $\Gamma\models_{\mathsf{IF_1KSJL}_{CS}}\phi$.
\end{enumerate}
\end{corollary}
In particular, the equivalences between (1), (2) and (3) in Corollary \ref{cor:intcompframe}, for $\mathbf{IPC}$, are the completeness theorems obtained by Marti and Studer in \cite{MS2016}.
\subsubsection{$\mathbf{G_n}$ and $\mathbf{KC}$}
We give some exemplary completeness results for some intermediate justification logics not present in the previous literature to demonstrate some breadth of the here proved completeness theorems.

Using Theorems \ref{thm:algmkrtmodcomp}, \ref{thm:algfittingmodcomp} and \ref{thm:algsubsetmodcomp}, we obtain the following corollary based on Theorem \ref{thm:propcompgnalg}.
\begin{corollary}
Let $\mathbf{G_nJL}_0\in\{\mathbf{G_nJ}_0,\mathbf{G_nJT}_0,\mathbf{G_nJ4}_0,\mathbf{G_nJT4}_0\}$ and let $CS$ be a constant specification for $\mathbf{G_nJL}_0$. For any $\Gamma\cup\{\phi\}\subseteq\mathcal{L}J$, the following are equivalent:
\begin{enumerate}
\item $\Gamma\vdash_{\mathbf{G_nJL}_{CS}}\phi$;
\item $\Gamma\models^1_{\mathbf{V^{(n)}_G}\mathsf{AMJL}_{CS}}\phi$;
\item $\Gamma\models^1_{\mathbf{V^{(n)}_G}\mathsf{AFJL^c}_{CS}}\phi$;
\item $\Gamma\models^1_{\mathbf{V^{(n)}_G}\mathsf{ASJL^c}_{CS}}\phi$.
\end{enumerate}
\end{corollary}
Further, based on Theorems \ref{thm:genmkrtmodcomp}, \ref{thm:genfittingmodcomp} and \ref{thm:gensubsetmodcomp} together with Theorem \ref{thm:propcompgn}, we obtain:
\begin{corollary}
Let $\mathbf{G_nJL}_0\in\{\mathbf{G_nJ}_0,\mathbf{G_nJT}_0,\mathbf{G_nJ4}_0,\mathbf{G_nJT4}_0\}$ and let $CS$ be a constant specification for $\mathbf{G_nJL}_0$. For any $\Gamma\cup\{\phi\}\subseteq\mathcal{L}J$, the following are equivalent:
\begin{enumerate}
\item $\Gamma\vdash_{\mathbf{G_nJL}_{CS}}\phi$;
\item $\Gamma\models_{\mathsf{CIF_{n-1}KMJL}_{CS}}\phi$;
\item $\Gamma\models_{\mathsf{CIF_{n-1}KFJL^c}_{CS}}\phi$;
\item $\Gamma\models_{\mathsf{CIF_{n-1}KSJL^c}_{CS}}\phi$.
\end{enumerate}
\end{corollary}
For Jankov's logic, Theorem \ref{thm:propcompkcframe} and Corollary \ref{cor:propcompkcframe} result in the following completeness theorem for justification logics based on Jankov's logic as a corollary of Theorems \ref{thm:genmkrtmodcomp}, \ref{thm:genfittingmodcomp} and \ref{thm:gensubsetmodcomp}.
\begin{corollary}
Let $\mathbf{KCJL}_0\in\{\mathbf{KCJ}_0,\mathbf{KCJT}_0,\mathbf{KCJ4}_0,\mathbf{KCJT4}_0\}$ where $CS$ is a constant specification for $\mathbf{KCJL}_0$. For any $\Gamma\cup\{\phi\}\subseteq\mathcal{L}_J$, the following are equivalent:
\begin{enumerate}
\item $\Gamma\vdash_{\mathbf{KCJL}_{CS}}\phi$;
\item $\Gamma\models_{\mathsf{DIF}\mathsf{KMJL}_{CS}}\phi$;
\item $\Gamma\models_{\mathsf{DIF}\mathsf{KFJL}_{CS}}\phi$;
\item $\Gamma\models_{\mathsf{DIF}\mathsf{KSJL}_{CS}}\phi$.
\end{enumerate}
\end{corollary}
\subsection{Some intermediate modal logics and their semantics}
To consider realization results, based on the nature of our realization theorem, we need an appropriate semantical completeness result of the intermediate modal logic in question. It is not as straightforward, however, to establish such uniform completeness results for the intermediate modal logics as in the completeness theorems given in this paper for the intermediate justification logics in the sense of lifting \emph{arbitrary} classes of frames for the underlying intermediate logic to classes of frames for the intermediate modal logic. However, we can give some partial results. For this, we first introduce some notation:
\begin{definition}
Let $\mathsf{C}$ be a class of Kripke frames. Then, we write:
\begin{enumerate}
\item $\mathsf{CMK}$ for the class of all intuitionistic modal Kripke models over frames from $\mathsf{C}$;
\item $\mathsf{CMT}$ for the class of all $\mathsf{CMK}$ with reflexive $\mathcal{R}$;
\item $\mathsf{CMK4}$ for the class of all $\mathsf{CMK}$ with transitive $\mathcal{R}$;
\item $\mathsf{CMS4}$ for the class of all $\mathsf{CMK}$ with reflexive and transitive $\mathcal{R}$.
\end{enumerate}
\end{definition}
We state the following completeness results:
\begin{theorem}\label{thm:modalcomp}
Given an intermediate logic $\mathbf{L}$, let here $\mathbf{LML}\in\{\mathbf{LK},\mathbf{LT},\mathbf{LK4},\mathbf{LS4}\}$. Then:
\begin{enumerate}
\item $\mathbf{IPCML}$ is strongly complete w.r.t. $\mathsf{IFMML}$;
\item $\mathbf{GML}$ is strongly complete w.r.t. $\mathsf{CIFMML}$;
\item $\mathbf{CML}$ is strongly complete w.r.t. $\mathsf{IF_1MML}$;
\item $\mathbf{G_nML}$ is strongly complete w.r.t. $\mathsf{CIF_{n-1}MML}$;
\item $\mathbf{KCML}$ is strongly complete w.r.t. $\mathsf{DIFMML}$.
\end{enumerate}
\end{theorem}
This theorem can be obtained using a usual canonical model construction similar to that of Section \ref{sec:canmodcon} (which we just sketch in the following). As before, we consider tableaux, this time over $\mathcal{L}_\Box$ with appropriately adapted notions of $\mathbf{LML}$-consistency.
\begin{definition}
The standard canonical intuitionistic modal Kripke model for $\mathbf{LML}$ is the structure $\mathfrak{M}^{sc}(\mathbf{LML})=\langle\mathcal{W}^{sc},\preceq^{sc},\mathcal{R}^{sc},\Vdash^{sc}\rangle$ which is defined by
\begin{enumerate}
\item $\mathcal{W}^{sc}:=\{\tau=(\Gamma,\Delta)\mid \tau\text{ is }\mathbf{LML}\text{-consistent and maximal}\}$,
\item $\tau\preceq^{sc}\tau'$ iff $\Gamma\subseteq\Gamma'$ iff $\Delta\supseteq\Delta'$,
\item $\tau\mathcal{R}^{sc}\tau'$ iff $\Gamma^\#\subseteq\Gamma'$ for $\Gamma^\#:=\{\phi\mid \Box\phi\in\Gamma\}$,
\item $\Vdash^{sc}(p)=\{\tau=(\Gamma,\Delta)\mid p\in\Gamma\}$,
\end{enumerate}
where $\tau=(\Gamma,\Delta)$ and $\tau'=(\Gamma',\Delta')$.
\end{definition}
\begin{theorem}
Let $\mathbf{L}$ be an intermediate logic, $\mathbf{LML}\in\{\mathbf{LK},\mathbf{LT},\mathbf{LK4},\mathbf{LS4}\}$ and $\mathfrak{M}^{sc}(\mathbf{LML})=\langle\mathcal{W}^{sc},\preceq^{sc},\mathcal{R}^{sc},\Vdash^{sc}\rangle$ its canonical model. For any $\phi\in\mathcal{L}_\Box$ and any $\tau=(\Gamma,\Delta)\in\mathcal{W}^{sc}$:
\begin{enumerate}
\item $\phi\in\Gamma\Rightarrow (\mathfrak{M}^{sc}(\mathbf{LJL}),\tau)\models\phi$;
\item $\phi\in\Delta\Rightarrow (\mathfrak{M}^{sc}(\mathbf{LJL}),\tau)\not\models\phi$.
\end{enumerate}
\end{theorem}
\begin{proof}
A proof of the analogues propositional result can be found in \cite{CZ1997} and the arguments transfer to the case here. Similarly, the argument given in \cite{CZ1997} regarding the $\Box$ modality in the case for classical modal logic apply here for the $\Box$-connective.
\end{proof}
\begin{lemma}
Let $\mathbf{LML}\in\{\mathbf{LK},\mathbf{LT},\mathbf{LK4},\mathbf{LS4}\}$ and let $\mathfrak{M}^{sc}(\mathbf{LML})=\langle\mathcal{W}^{sc},\preceq^{sc},\mathcal{R}^{sc},\Vdash^{sc}\rangle$ be its canonical model. Then:
\begin{enumerate}
\item if $(T)\subseteq\mathbf{LML}$, then $\mathcal{R}^{sc}$ is reflexive;
\item if $(4)\subseteq\mathbf{LML}$, then $\mathcal{R}^{sc}$ is transitive.
\end{enumerate}
\end{lemma}
\begin{lemma}\label{lem:charaxioms}
Let $\mathbf{LML}\in\{\mathbf{LK},\mathbf{LT},\mathbf{LK4},\mathbf{LS4}\}$ and let $\mathfrak{M}^{sc}(\mathbf{LML})=\langle\mathcal{W}^{sc},\preceq^{sc},\mathcal{R}^{sc},\Vdash^{sc}\rangle$ be its canonical model.  Then:
\begin{enumerate}
\item $(LIN)\in\mathbf{LML}\Rightarrow \langle\mathcal{W}^{sc},\preceq^{sc}\rangle$ is connected;
\item $(BC)_n\in\mathbf{LML}\Rightarrow \langle\mathcal{W}^{sc},\preceq^{sc}\rangle$ is of bounded cardinality $n$;
\item $(WLEM)\in\mathbf{LML}\Rightarrow \langle\mathcal{W}^{sc},\preceq^{sc}\rangle$ is directed;
\item $(LEM)\in\mathbf{LML}\Rightarrow \langle\mathcal{W}^{sc},\preceq^{sc}\rangle$ is of bounded cardinality $1$.
\end{enumerate}
\end{lemma}
The two above lemmas can be obtained completely analogously to corresponding results in the study of propositional intermediate logics and classical modal logics from \cite{CZ1997} (see especially Theorem 5.16 from \cite{CZ1997}) and a proof is thus omitted.\\

The previous lemmas constitute the necessary parts for a proof of the completeness result given in Theorem \ref{thm:modalcomp} by using the canonical model as a countermodel construction. We omit a detailed proof.
\subsection{Realization theorems}
Essential for both the known and unknown realization theorems is the following result:
\begin{theorem}\label{thm:charaxioms}
Let $\mathbf{L}$ be an intermediate logic and let $\mathbf{LJL}$ be an intermediate justification logic. Let $\mathfrak{M}^{sc}=\mathfrak{M}^{sc}(\mathbf{LJL})=\langle\mathcal{W}^{sc},\preceq^{sc},\mathcal{R}^{sc},\mathcal{E}^{sc},\Vdash^{sc}\rangle$ be the canonical model of $\mathbf{LJL}$.
\begin{enumerate}
\item $(LIN)\in\mathbf{LJL}\Rightarrow \langle\mathcal{W}^{sc},\preceq^{sc}\rangle$ is connected;
\item $(BC)_n\in\mathbf{LJL}\Rightarrow \langle\mathcal{W}^{sc},\preceq^{sc}\rangle$ is of bounded cardinality $n$;
\item $(WLEM)\in\mathbf{LJL}\Rightarrow \langle\mathcal{W}^{sc},\preceq^{sc}\rangle$ is directed;
\item $(LEM)\in\mathbf{LJL}\Rightarrow \langle\mathcal{W}^{sc},\preceq^{sc}\rangle$ is of bounded cardinality $1$.
\end{enumerate}
\end{theorem}
Again, this theorem is in analogy to the corresponding result from \cite{CZ1997} for intermediate propositional and classical modal logics (see again Theorem 5.16 there) and to the previous Lemma \ref{lem:charaxioms} and followingly we again omit a proof.\\

Using the above Theorems \ref{thm:charaxioms} and \ref{thm:modalcomp}, one validates the condition of $\langle\mathcal{W}^{sc},\preceq^{sc},\mathcal{R}^{sc}\rangle\in\mathsf{C}$ in Theorem \ref{thm:bimodreal} for the respective classes $\mathsf{C}$ from above and obtains the following corollaries:
\begin{corollary}\label{cor:realipc}
$\mathbf{IPCJL}\in\{\mathbf{IPCJ},\mathbf{IPCJT},\mathbf{IPCJ4},\mathbf{IPCJT4}\}$ realizes the corresponding modal logic $\mathbf{IPCML}$.
\end{corollary}
\begin{corollary}\label{cor:realg}
$\mathbf{GJL}\in\{\mathbf{GJ},\mathbf{GJT},\mathbf{GJ4},\mathbf{GJT4}\}$ realizes the corresponding modal logic $\mathbf{GML}$.
\end{corollary}
\begin{corollary}\label{cor:realc}
$\mathbf{CJL}\in\{\mathbf{CJ},\mathbf{CJT},\mathbf{CJ4},\mathbf{CJT4}\}$ realizes the corresponding modal logic $\mathbf{CML}$.
\end{corollary}
Corollary \ref{cor:realipc} was first obtained by Marti and Studer in \cite{MS2016} (although only for $\mathbf{IPCJT4}$ explicitly) using a constructive proof. Corollary \ref{cor:realg} has been constructively obtained in \cite{Pis2019} and the part of $\mathbf{CJT4}$ of Corollary \ref{cor:realc} goes back to the original work of Artemov \cite{Art1995,Art2001}. The other parts of Corollary \ref{cor:realc} seem to be due to \cite{Bre2000} (see e.g. the remark in \cite{Kuz2000}) although we were not able to obtain that manuscript.\\

Theorem \ref{thm:charaxioms} can be used for more than just reobtaining old realization theorems. Recalling the axiomatizations of the previous subsection $\mathbf{G_n}$ and $\mathbf{KC}$, we find also the following corollaries:
\begin{corollary}\label{cor:realgn}
$\mathbf{G_nJL}\in\{\mathbf{G_nJ},\mathbf{G_nJT},\mathbf{G_nJ4},\mathbf{G_nJT4}\}$ realizes the corresponding modal logic $\mathbf{G_nML}$.
\end{corollary}
\begin{corollary}\label{cor:realkc}
$\mathbf{KCJL}\in\{\mathbf{KCJ},\mathbf{KCJT},\mathbf{KCJ4},\mathbf{KCJT4}\}$ realizes the corresponding modal logic $\mathbf{KCML}$.
\end{corollary}


\end{document}